\documentclass[leqno, letterpaper, 11pt]{article}

\usepackage{amssymb}
\usepackage{amsmath}
\usepackage{amsfonts}
\usepackage{amsthm}
\usepackage[usenames]{color}
\usepackage{graphicx}
\usepackage{bbm}
\usepackage{verbatim}
\usepackage{longtable}
\usepackage{subcaption}
\usepackage[utf8]{inputenc}

\usepackage{pgf,tikz}
\usetikzlibrary{arrows}

\newcommand{\de}{\delta}

\newcommand{\ca}{\xi}
\newcommand{\De}{\Delta}

\newcommand{\spr}{{\text{\tt spr}}}
\newcommand{\damage}{{\text{\tt damage}}}

\newcommand{\bR}{{\mathbb R}}
\newcommand{\bZ}{{\mathbb Z}}

\newcommand{\cC}{{\mathcal C}}

\newcommand{\cF}{{\mathcal F}}

\newcommand{\cN}{{\mathcal N}}

\newcommand{\cM}{{\mathcal M}}

\newcommand{\cP}{\mathcal P}

\newcommand{\cS}{\mathcal S}
\newcommand{\cE}{\mathcal E}

\renewcommand{\P}{{\mathbb P\/}}
\newcommand{\E}{{\mathbb E\/}}

\newcommand{\lsquare}{\text{\scalebox{2}[1]{$\square$}}}

\newcommand{\ind}{\mathbbm{1}}

\DeclareMathOperator{\mmod}{mod}

\newcounter{mycount}

\newtheorem{theorem}{Theorem}[section]
\newtheorem{prop}[theorem]{Proposition}

\newcounter{example}

\newcounter{open}

\newcounter{figno}

\newcounter{tableno}

\numberwithin{equation}{section}
\numberwithin{figure}{section}
\numberwithin{table}{section}

\begin{document}
\pagestyle{empty}

\null

\begin{center}\Large
{\bf  Stability 
of cellular automata 
trajectories
revisited: \\branching walks and Lyapunov profiles}
\end{center}

\begin{center}

{\sc Jan M.~Baetens}\\
{\rm KERMIT}\\
{\rm Department of Applied Mathematics, Biometrics and Process Control}\\ 
{\rm Ghent University}\\
{\rm Coupure links 653, Gent, Belgium}\\
{\tt jan.baetens{@}ugent.be}\\
\end{center}

\begin{center}
{\sc Janko Gravner}\\
{\rm Mathematics Department}\\
{\rm University of California}\\
{\rm Davis, CA 95616, USA}\\
{\rm \tt gravner{@}math.ucdavis.edu}
\end{center}

\begin{center}
{
April 11, 2016}
\end{center}

\vspace{0.3cm}

\paragraph{Abstract.}
We study non-equilibrium defect accumulation 
dynamics on a cellular automaton trajectory: a branching walk
process in which a defect creates a successor on any neighborhood site 
whose update it affects. On an infinite lattice, defects 
accumulate at different exponential rates in different 
directions, giving rise to the Lyapunov profile. 
This profile quantifies instability
of a cellular automaton evolution and is connected to the theory of
large deviations. We rigorously and empirically study Lyapunov profiles 
generated from random initial states. We also introduce explicit and computationally 
feasible variational methods to compute the Lyapunov profiles
for periodic configurations, thus developing an analogue of
Floquet theory for cellular automata.

\vskip0.3cm

\noindent 2010 {\it Mathematics Subject Classification\/}: 60K35, 37B15.

\vskip0.3cm

\noindent {\it Key words and phrases\/}: asymptotic shape,  branching walk,
cellular automaton, doubly periodic configuration, large deviations,
Lyapunov exponent, percolation, stability.

\newpage

\addtocounter{page}{-1}

\pagestyle{headings}
\thispagestyle{empty}


\section{Introduction}\label{intro}

{
Quantifying instability in physical systems and in mathematical models 
is a long-standing problem in nonlinear science, beginning with Lyapunov's pioneering work at the 
end of 19th century. Lyapunov discovered that the basic quantities are exponential 
rates which, when positive, measure divergence from an unstable trajectory. In this paper, we
elaborate on the well-known fact that instabilities often do not 
affect all components of a system to the same extent; more precisely, we  
study how fast defects may spread among these components, which we assume are spatially 
distributed. In the process, we establish connections 
with large deviation theory, a branch of probability theory that studies exponentially small
probabilities of ``rare'' events that do not conform to the ``typical'' scenario. 
In our models, the defects accumulate in space as a system of random walks, whose 
large deviation rates then determine Lyapunov instability. 
This point of view is not only useful 
when the dynamics starts at a random initial state, but also in periodic states with no 
randomness at all. 

While our approach could work for other many-component systems, we chose 
{\it cellular automata\/} ({\it CA\/}) as our platform. These deterministic dynamical systems are 
spatially and temporally discrete, with a fixed local update rule that
mandates that a new
state at the next tick of a clock depends only on a finite number of neighboring states. In addition, 
each spatial location (playing the role of a component or a degree of freedom) 
can be occupied with one of only finitely many states --- for simplicity, 
we will only consider {\it binary\/} CA, in which a site either takes the state $0$ or the state $1$. 
This setting minimizes technical considerations, which however remain a considerable 
challenge. It also facilitates the development of a computational approach, which has become
an indispensable element of stability research in many fields, but is particularly well-suited 
for CA. Consequently, our conclusions are based both on large-scale calculations and on rigorous mathematical arguments, 
the latter largely probabilistic. 

Let us consider a binary CA that is evolved from an initial state $\ca_0$, 
generating a trajectory $\ca_t$, $t=0,1,\ldots$. For instance, the CA known as 
{\it Rule 22\/} or {\it Exactly 1\/} \cite{GG4}, whose sites are integers in $\bZ$, 
is governed by the update rule dictating that
the state at $x\in \bZ$ is $1$ at time $t\ge 1$ if and only if exactly one of states at its three 
neighboring sites $x-1$, $x$ and $x+1$ was $1$ at the previous time step. 
How stable is a CA trajectory?} By analogy with continuous dynamical systems, 
the idea is to measure the effect of a small perturbation of $\ca_0$ on the evolution 
at later times. In his classic work, Wolfram \cite{Wol1} considered damage spreading, 
that is, growth of the set of affected sites in $\ca_t$ when a few sites in $\ca_0$ are flipped.  
In one dimension, there are two directions of propagation; when the maximum 
extent of damage progresses linearly the two slopes are called {\it Lyapunov exponents\/} 
as they measure the exponential divergence in distance between the original and 
perturbed states in the appropriate metric. This concept was developed further 
from computational and theoretical perspectives in
\cite{Gra1, Gra2, She, CK, FMM, Tis1, Tis2}.  

Damage spreading is possibly the simplest approach but it gives no indication on the 
rate of divergence within a bounded region; in particular it has nothing to say 
on the CA evolution on finite sets. Thus a different tool was introduced 
by Bagnoli et al.~\cite{BRR}, based on the fact that Lyapunov exponents in continuous 
dynamical systems can also be given
locally through the eigenvalues of the governing Jacobian. The Boolean derivative introduced in 
\cite{Vic1} is used in \cite{BRR} as the analogue for the Jacobian, which leads to the {\it branching
walk\/} dynamics of defects that we now informally describe. Recall that a trajectory $\ca_t$ of 
a CA is fixed. Assume a defect is present 
at a site $y$ at time $t$. That defect looks into each of its neighborhood sites $x$ to check whether
flipping the state $\ca_t$ at $y$ would produce a different state at $x$ than assigned by $\ca_{t+1}$; 
if so, the defect produces a successor at $x$. Each defect may produce more than one successor
(hence the term ``branching'') and acts independently of other defects. The exponential rate 
of accumulation of such defects is called the {\it maximal Lyapunov exponent\/} ({\it MLE\/}). 
The authors of \cite{BRR} envision this as an {\it equilibrium\/} theory: they 
measure the accumulation on a finite circle of sites after a long time (much larger than 
the length of the circle) has elapsed. Due to the resulting spatial translation invariance, 
there is only one rate of accumulation, and the meaning of the word {\it maximal\/} is unclear, 
except to distinguish the notion from the one arising from damage spreading; however, the
present setting provides an  {\it ex post facto\/} justification of this term. 

In this paper we continue the study initiated in 
\cite{BG} of the {\it non-equilibrium\/} version of defect branching walk 
dynamics. As the defects spread on an {\it infinite\/} lattice, there is substantial 
spatial variation in their accumulation; the exponential rates of spread in all space-time 
directions are collected into a function we call the {\it Lyapunov profile\/}.
{
For example, a one-dimensional Lyapunov profile 
$L=L(\alpha)$ roughly gives, for a real number $\alpha$, the exponential 
rate of accumulation on the line $x=\alpha t$ (see Fig.~\ref{intro-figure} 
for a few examples, including {\it Rule 22\/}). There is some conceptual similarity between this 
object and} the Lyapunov spectrum in multidimensional smooth dynamical systems, 
whereby the spectrum of the Jacobian accounts for 
perturbations in all directions in both the input and the output. In the discrete 
CA configuration space there is essentially one way to make an infinitesimal perturbation 
in the input (assuming irreducibility), but the
effect can be quite different in different directions of the output. Moreover, we 
empirically observe that typically 
the direction with the maximal effect has the profile height 
that is close to the MLE of \cite{BRR}.  

We emphasize that the dynamics of branching defects does not alter the 
trajectory $\ca_t$ but instead uses it as an environment for its evolution. It 
is thus a kind of second-class dynamics akin to the ones that percolate 
and create periodic structures in \cite{GH}, and to the ``slave'' 
synchronization rules of \cite{BER}. In fact, the set of sites that contain
at least one defect 
evolves as a four-state CA, which we refer to as the {\it defect percolation\/} CA, and which 
is conceptually very similar to the rules studied in \cite{GH}.
One property that substantially facilitates 
the analysis is that our dynamics are monotone --- adding 
defects only results in more of them later on --- a property that fails to 
hold for Wolfram's damage spreading. We call the asymptotic rate of defect spread,  
typically equal to the set on which the Lyapunov profile 
differs from $-\infty$, the {\it defect shape\/}. The Lyapunov profiles thus 
simultaneously provide information on the spatial reach
and local accumulation resulting from a defect perturbation. The defect shape does 
not have an a priori relation to the (appropriately scaled) damaged set; 
as we will see in Section~\ref{prelim-defdam}, it can be larger or smaller.  

The most important initial state $\ca_0$ for the CA analysis is the uniform 
product measure, that is, one in which the probability of a $0$ or a $1$ at 
any site is independently $1/2$. Indeed, this random configuration is, in a way, 
one in which all configurations are equally likely. The 
trajectory $\ca_t$ then determines a space-time 
random field for defect dynamics, resulting in a branching {\it random\/} walk process. 
The study of such processes in an {\it independent\/} space-time random 
environment (e.g., \cite{Big, BNT}) is a well-established subfield of the large deviations theory
\cite{DZ, RS}. The main idea is that the resulting profiles are given 
by a variational method: the process seeks the most advantageous option for accumulation 
at a spatial location; in general, the search space can have a
very high dimension. Our defect accumulation dynamics evolves in 
a highly correlated random field, even when 
the uniform product measure is invariant \cite{GH},
and extending large deviation techniques is an extraordinary challenge. 
We thus rely mostly on empirical methods to analyze nontrivial cases
with random initialization. Notably, we observe that detectable dependence 
of MLE on the initial CA density is connected to the dramatic advantage 
of the defect percolation as compared to the damage spreading.

The other extreme are spatially periodic initial states, which after 
a transient ``burn-in'' time interval must become also temporally periodic. 
Study of the stability 
of periodic solutions of dynamical systems also has a long history, and 
is known as {\it Floquet theory\/}; see e.g.~\cite{Moo} for a recent perspective.
We are able to develop a fairly complete analogue for CA dynamics, based 
on large deviations for finite Markov chains \cite{DZ, RS}. These methods 
work particularly well under the irreducibility assumption, in which 
case the Lyapunov profile is given by a one-dimensional variational 
problem. We give several examples in Section~\ref{periodic-profiles}, 
including the profile for the {\it Rule 110\/} ether \cite{Coo}.  
We also introduce direct methods to determine the defect 
shape, related to convex transforms 
that originate from crystallography. 

Lyapunov profiles encapsulate a lot of information on the stability 
of CA trajectory, but not all of it. For example, many rules, such as {\it Rule 22\/}, 
develop holes in the set of defect sites (see Fig.~\ref{intro-figure})
due to stable updates, that is, configurations whose updates are 
insensitive to perturbations at a single site. Thus the {\it defect density profile\/}, 
a function that gives 
the density of defect sites in a given space-time direction, is of interest. Although there is no
known a priori reason, density 
profiles of CA trajectories are typically constant on their support \cite{GG4, GG5}; 
we observe the same here (see Fig.~\ref{damage-defect-figure}), 
although non-constant density profiles do
occur, for example, due to reducibility (e.g., {\it Tot 7\/} example in 
Fig.~\ref{ex1shapes-examples}).

\begin{figure}[!ht]
\begin{center}
\includegraphics[trim=0cm 0cm 0cm 0cm, clip, width=9cm]{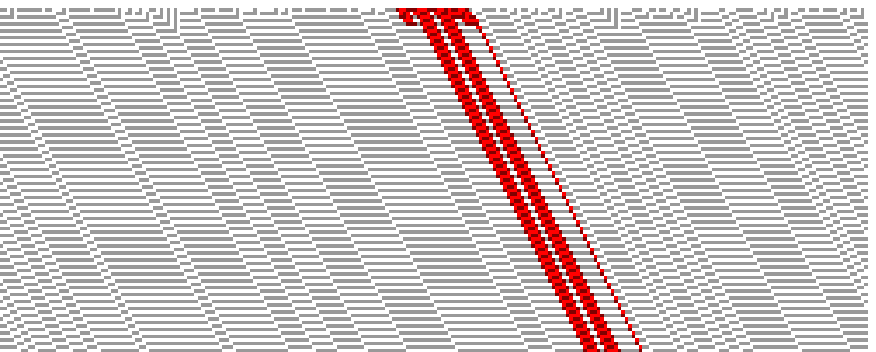}
\hspace{0.2cm}
\includegraphics[trim=0cm 0.2cm 0cm 0cm, clip, width=6cm]{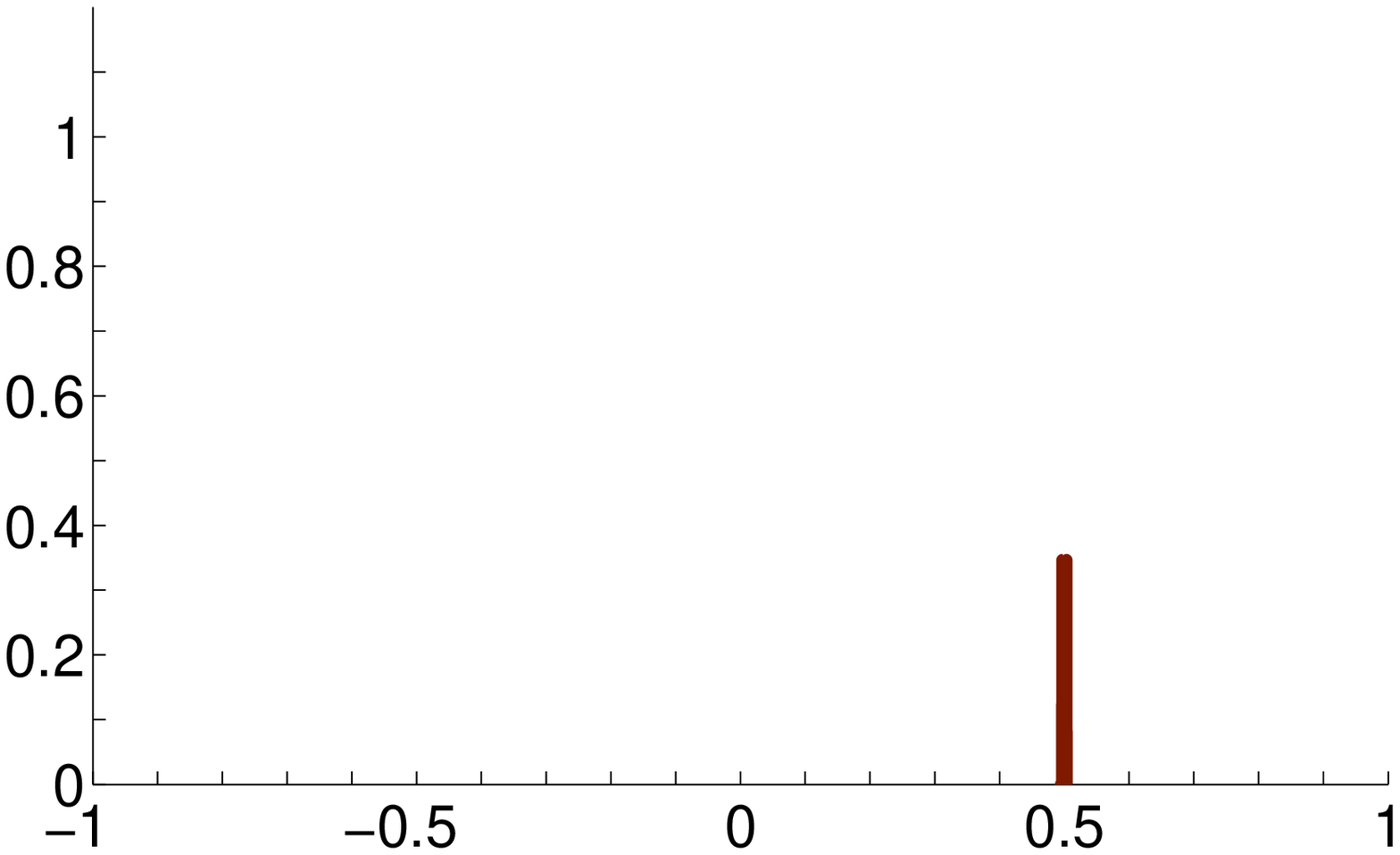}\\
\vspace{0.5cm}
\includegraphics[trim=0cm 0cm 0cm 0cm, clip, width=9cm]{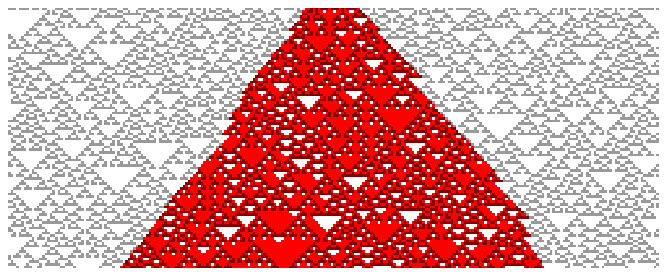}
\hspace{0.2cm}
\includegraphics[trim=0cm 0.2cm 0cm 0cm, clip, width=6cm]{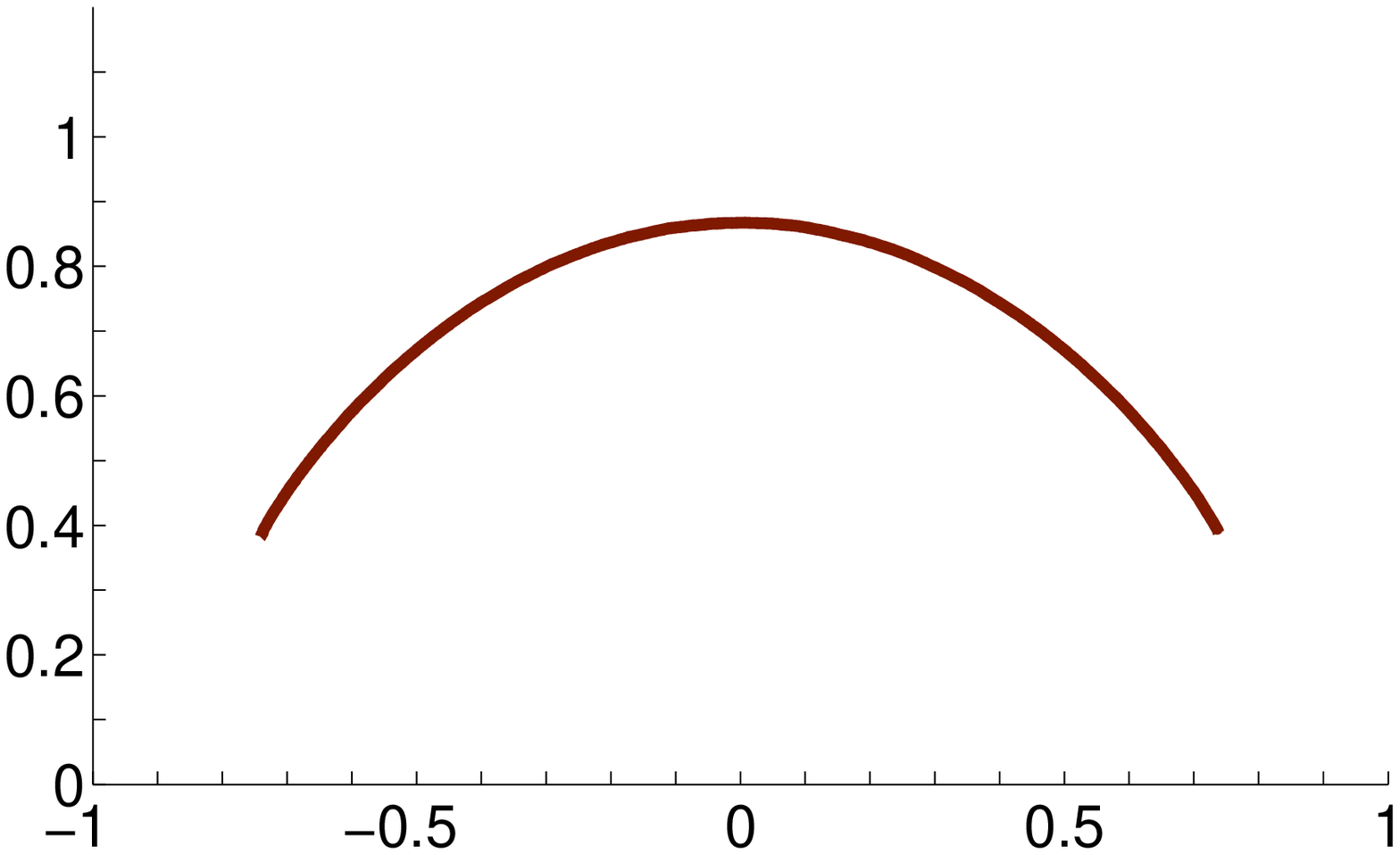}\\
\vspace{0.5cm}
\includegraphics[trim=0cm 0cm 0cm 0cm, clip, width=9cm]{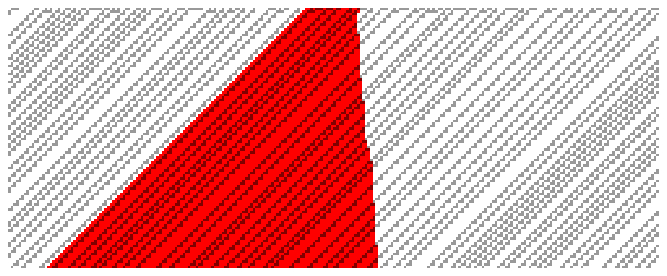}
\hspace{0.2cm}
\includegraphics[trim=0cm 0.2cm 0cm 0cm, clip, width=6cm]{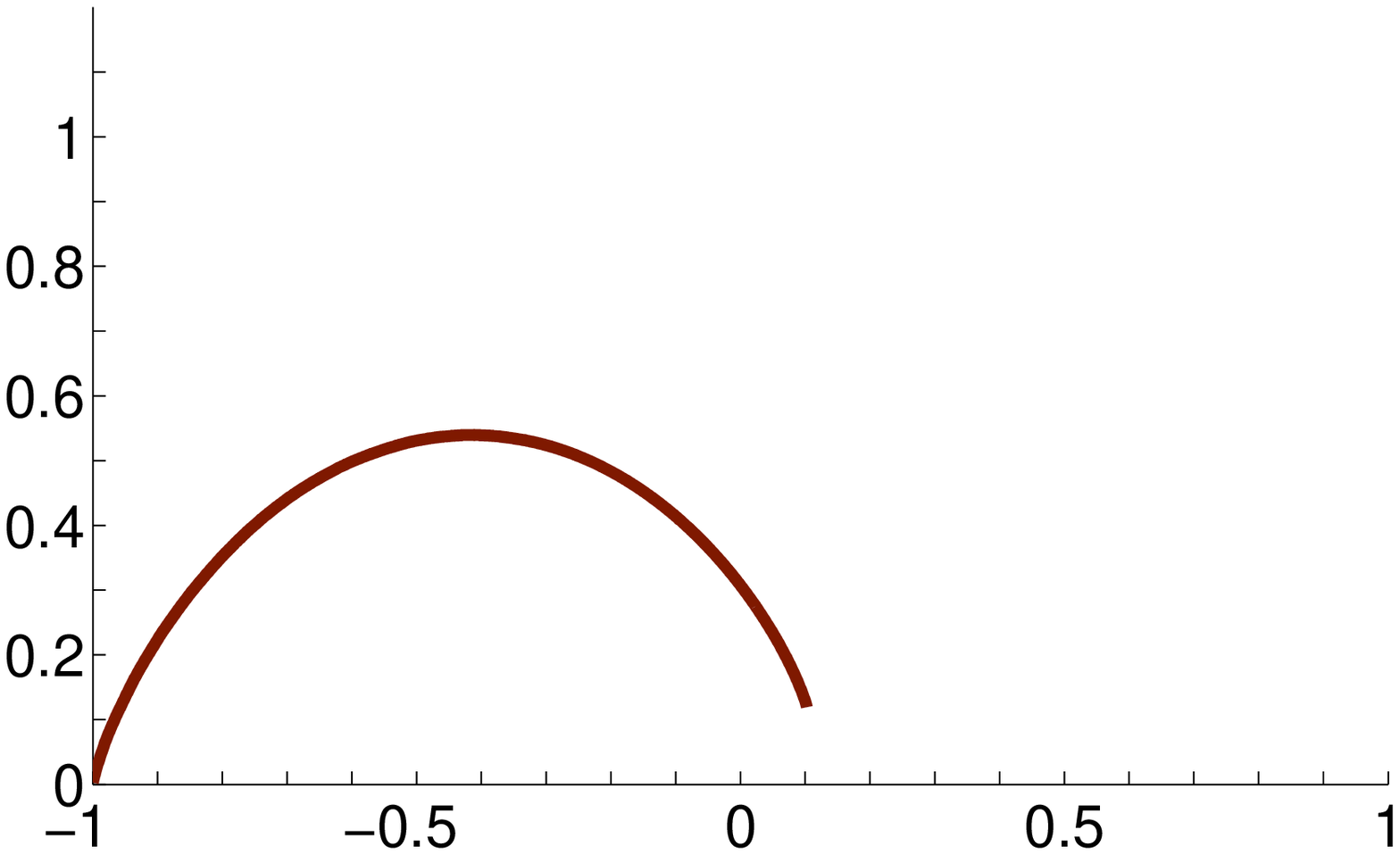}\\
\vspace{0.5cm}
\includegraphics[trim=0cm 0cm 0cm 0cm, clip, width=9cm]{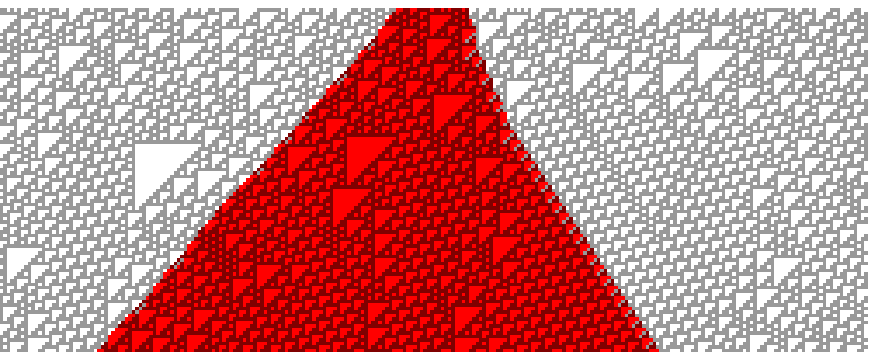}
\hspace{0.2cm}
\includegraphics[trim=0cm 0.2cm 0cm 0cm, clip, width=6cm]{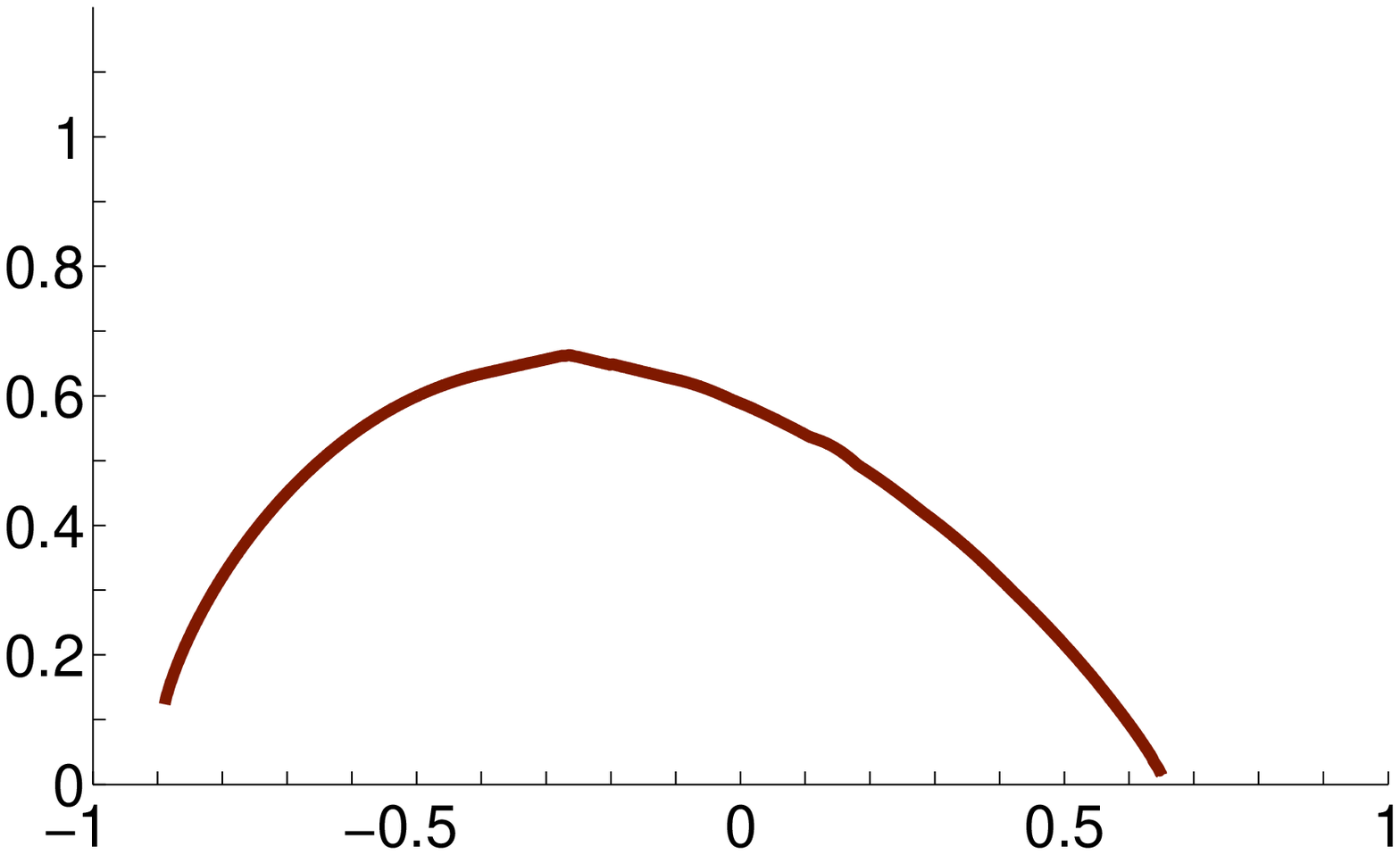}
\end{center}
\caption{ Evolution of defect percolation CA up to time $100$ {
(red sites are those 
that contain at least one defect)} and empirical 
Lyapunov profiles
at time $10^5$ 
for rules {\it 7\/},
{\it 22\/}, {\it 38\/}, and {\it 110\/}.} {
These profiles encapsulate the exponential accumulation rate vs.~space-time direction.}\label{intro-figure}
\end{figure}

We give formal definitions and some general results in Section~\ref{prelim}, after which
we focus on {\it elementary\/} CA \cite{Wol1}, the $256$ one-dimensional, three-neighbor
rules that have long been considered the primary testing ground for 
any CA theory, including stability analysis (e.g., \cite{BRR}). 
Due to their wide acceptance, we use the Wolfram's serial 
numbers \cite{Wol1} as nomenclature; the web site
\cite{Wol2} is particularly useful for a quick reference. 
Section~\ref{ECA-section} includes a comprehensive discussion on elementary CA
defect dynamics from the uniform product initial state. We also consider two-dimensional 
rules (Sections~\ref{2d-section} and~\ref{periodic-shapes}), where we restrict to 
{\it totalistic\/} rules whose update only depends on the neighborhood count. 


We conclude this section with a few illustrative examples and a brief discussion 
on how our approach relates to other complexity measures of CA rules. 
Fig.~\ref{intro-figure} depicts 
a sample defect percolation evolution for 
four elementary CA, together 
with approximate Lyapunov profiles\footnote{All of our many graphs 
of Lyapunov profiles are plots of the exponential accumulation rate vs.~space-time 
direction (see (\ref{lyapunov-profile-def})
for the formal definition); thus we omit the axis labels.}.
In all cases 
$\ca_0$ is the uniform product measure and the initial set of 
defects is an interval of $21$ sites. 
The first example is {\it Rule 7\/}, one of many rules
with degenerate profiles that 
are typically caused by persistent moving obstacles that defects cannot cross
(see Table~\ref{ECA-marginal}). 
Next is {\it Rule 22\/},
a classic chaotic rule for which it appears, at first glance, that 
the defect percolation has the same asymptotics as damage spreading \cite{Gra2}, but we 
will present evidence that this is not the case (see Fig.~\ref{damage-defect-figure}).
Next, {\it Rule 38\/} is the simplest
{\it stripes\/} rule (see Section~\ref{prelim-defdam})
in which the initial state creates a quenched random environment
for the branching walks, and thus the dynamics is conceptually similar
to one on a random tessellation \cite{BD}. Much about the resulting dynamics can be 
proved (see Prop.~\ref{rule-38}). The final example is 
{\it Rule 110\/}, which famously creates a periodic ether for interaction 
between various types of gliders \cite{Coo}. Whether the density of gliders
approaches zero is unknown (see \cite{LN} for positive evidence), 
and thus it is even less clear whether the Lyapunov 
profile approaches the one obtained by {\it starting\/} from the ether. The 
latter profile can be characterized by an explicit variational formula 
(see Section~\ref{periodic-profiles-rule110}).  

There have been many attempts to classify CA through complexity, 
going back to \cite{Wol1}; see \cite{MSZ, Mar, ZV} for recent reviews. 
The original Wolfram classification into four classes, {\it uniform\/}, 
{\it periodic\/}, {\it chaotic\/} and {\it complex\/}, often simply referred 
to by numerals 1--4, is still 
in wide use \cite{Mar}, despite considerable ambiguity in many interesting cases
\cite{MSZ}; for example, the intriguing {\it Rule 106\/}
(or its edge version {\it EEED\/} \cite{GG5})
could be called chaotic or complex. 
Much of the literature has attempted
to condense the complexity properties of a CA rule into a single 
number, although it is unclear if a linear ordering of 
CA by complexity provides the most insight; see \cite{ZV} and \cite{Mar} 
for some ``competing'' measures and resulting classifications. Our paper 
underscores this point by instead assigning a function 
to every CA rule, as in the right panels of Figure~\ref{intro-figure}. 
This leads to a natural division of rules into three classes, which 
can in the case of elementary CA be described as follows: 
{\it collapsing\/} rules for which the defects die out (e.g., rules {\it 0\/} and {\it 40\/}); 
{\it marginal\/} rules whose 
Lyapunov profile is a 
single ``stick,'' as is for {\it Rule 7\/} in Figure~\ref{intro-figure}; 
and {\it expansive\/} rules which generate exponential accumulation of defects  
on a linearly growing set, as in the other three cases 
of Figure~\ref{intro-figure}.  

One may intuitively expect that 
rules that have, by some measure, large complexity are the
expansive ones. As tends to be the case, this rule of thumb is 
useful, but is not a perfect predictor, as we now briefly 
illustrate on elementary CA.  
To be definite, we use Table 2 in \cite{Mar} 
as Wolfram's classification. The 8 uniform (class 1) rules 
are exactly the 8 rules in Table~\ref{ECA-collapsing} 
and are therefore all collapsing. At the 
other extreme, the 11 chaotic (class 3) and 4 complex (class 4) rules are all expansive (see 
Table~\ref{ECA-expansive}). However, there are expansive elementary CA which 
are classified as periodic (class 2). These are the stripes rules (of which 
{\it Rule 38\/} from Figure~\ref{intro-figure} is an example), which are 
in a sense in their own class: expansive but simple enough to be at least partly amenable to 
mathematical analysis. On the other side, there are two marginal rules, {\it 73\/} and
{\it 94\/}, that are sometimes classified as periodic \cite{Mar} and sometimes 
as complex \cite{ZV}, and stand out 
in our analysis as well in that the height of their profile is unusually difficult to 
estimate. Finally, the three additional collapsing rules identified in Table~\ref{ECA-mystery} 
are just barely such, as discussed in Section~\ref{ECA-subsection-dendep}. 
The exceptional rules mentioned in this paragraph --- among which the remaining eight
glider rules in Table~\ref{ECA-mystery} can also be counted ---  are all worth further study. 
   

\section{Definitions and basic results} \label{prelim}

In this paper we only consider binary CA, leaving the discussion of 
larger state spaces to our subsequent work. Thus, our object of study is 
a {\it cellular automaton\/} on the $d$-dimensional integer lattice 
$\bZ^d$ with state space $\{0,1\}$ that is
given by the finite ordered {\it neighborhood\/} $\cN\subset \bZ^d$
and the ({\it local\/}) {\it update function\/} of $|\cN|$ variables: 
$\phi:\{0,1\}^{|\cN|}\to \{0,1\}$. 
For a string $\vec s\in \{0,1\}^{|\cN|}$ we also write $\vec s\mapsto s'$ instead 
of $\phi(\vec s)=s'$. 
We call an update $\vec s\mapsto s'$ {\it stable\/} if $\vec s_1\mapsto s'$ 
for every $\vec s_1$ that differs from $\vec s$ in only one state.

\newcommand{\change}{\text{\tt change}}
 
The neighborhood of a point 
$x\in \bZ^d$ is the translation $\cN_x=x+\cN$, {
ordered 
the same way as $\cN$}. 
The {\it global function\/} $\Phi:\{0,1\}^{\bZ^d}\to \{0,1\}^{\bZ^d}$ is given
{
as follows. For arbitrary $\eta\in \{0,1\}^{\bZ^d}$, and $x\in \bZ^d$, 
let $\eta|_{\cN_x}$ be the 
vector of $|\cN|$ entries given by values of $\eta$ on $\cN_x$, listed in the 
order of sites in $\cN_x$. The function $\phi$ applied to this vector
provides the value of $\Phi(\eta)$ at $x$; in symbols,  
 $$\Phi(\eta)(x)=\phi(\eta|_{\cN_x}).$$ }

We denote by $\ca_t(x)=\ca(x,t)$, $x\in \bZ^d$, $t\in \bZ_+$, 
a trajectory of the CA, starting from a fixed 
initial state $\ca_0$, which can be deterministic or random.
That is, $\ca_t$ is defined recursively by iteration of $\Phi$: $\ca_{t+1}=\Phi(\ca_t)$ for $t\ge 0$.

\subsection{Lyapunov profiles}\label{prelim-lp}

We begin by defining the
branching walk dynamics that measures propagation of perturbations; see 
e.g.~\cite{Big, BNT} for probabilistic analysis of branching random walk. 
The defect configuration
$\De_t(x)=\De(x,t)\in\bZ_+$ describes the distribution of defects. Informally, 
for every $x\in \bZ^d$, $y\in \cN_x$, and every defect counted into $\De_t(y)$,  
$\De_{t+1}(x)$ is increased by 1 if applying the CA rule on the configuration $\ca_t$
that is perturbed at $y$ results in a perturbation at $x$.
 
More formally, for a configuration $\eta\in \{0,1\}^{\bZ^d}$, and
$y\in \bZ^d$, the perturbation of $\eta$ at $y$ is
the configuration $\eta^{(y)}$ defined by 
$$
\eta^{(y)}(x)=
\begin{cases}
1-\eta(y) &x=y\\
\eta(x) &\text{otherwise}
\end{cases}
$$
Further, $\change_t$ collects the information about effects of perturbations 
at time {
 $t$, and is essentially the Boolean derivative \cite{Vic1},}
$$
\change_t(y,x)=\ind(\Phi(\ca_t^{(y)})(x)\ne \ca_{t+1}(x)).
$$
{
(Here, $\ind$ is the indicator function, which gives the value 
$1$ or $0$ whenever its logical argument is true or false, respectively.)} Then,  
$$
\De_{t+1}(x)=\sum_{y\in \cN_x} \change_t(y,x)\De_t(y).
$$
Again, $\De_0$ is a fixed configuration, which we will always assume is 
nonzero with (possibly large) finite support. We call $(\ca_t, \De_t)$ the {\it 
defect accumulation dynamics\/}, {
matching the definition in \cite{BRR}}.

The configuration $\de_t$
given by $\de_t(x)=\ind(\De_t(x)>0)$ induces a CA evolution $(\ca_t, \de_t)$, 
which we call the {\it defect percolation\/} CA. In this four-state rule, 
a defect at $y$
spreads into a neighboring site $x$ if a change of the state of $\ca_t$ 
at $y$ affects the state at $x$ at the next time step.
Therefore, $\de_t$ is an oriented percolation dynamics on the original space-time CA configuration 
$\ca_t$; it is affected by the original CA evolution, leaving it unaffected in return.
Thus it plays a similar role to the percolation process in \cite{GH} that 
governs disorder-resistance. Another example are the ``second-class'' or ``slave'' 
processes that control synchronization in \cite{BER}. As convenient, we often interpret
$\de_t$ as subset of $\bZ^d$, determined by its support. 

We define the {\it 
Lyapunov profile\/} to be the function $L:\bR^d\to \{-\infty\}\cup[0, \infty)$ given for 
$\alpha \in \bR^d$ by 
\begin{equation}\label{lyapunov-profile-def}
L(\alpha)=\lim_{\epsilon\downarrow 0} 
\limsup_{t\to\infty} \frac 1t\log\left(\sum_{x:||x/t-\alpha||<\epsilon}\De(x,t)\right),
\end{equation}
where the norm is Euclidean (or, equivalently, any other). Informally, 
$$
\De(t \alpha , t)\approx e^{L(\alpha)t}, 
$$
so that in the space-time direction $\alpha$ the defects accumulate at the exponential 
rate $L(\alpha)$. We call the Lyapunov profile $L$ {\it proper\/} if replacing 
$\limsup$ with $\liminf$ in (\ref{lyapunov-profile-def}) results in the same limit $L(\alpha)$ 
for all $\alpha$. 

It is easy to see that the limit in (\ref{lyapunov-profile-def})
exists as either a nonnegative finite number or $-\infty$, 
and that one may replace the sum 
with maximum. It is also clear that $L(\alpha)=-\infty$ when {
$\alpha$ is outside $\text{co}(\cN)$, the convex hull of $\cN$. Further,
$L(\alpha)\le \log |\cN|$ for all $\alpha$ and $L$ is upper semicontinuous.}  
The {\it maximal Lyapunov exponent\/} ({\it MLE\/}) is then defined to be
\begin{equation}\label{MLE-def}
\lambda=\max_\alpha L(\alpha).
\end{equation}
An $\alpha$ at which the maximum in (\ref{MLE-def}) is achieved
is called a {\it MLE direction\/}, and is a space-time direction with the 
fastest growth of the number of defects. 
See \cite{BRR} for a different definition of the MLE, and \cite{BD} for 
a discussion in a more general context. 
We empirically observe that our concept of MLE is close to that of \cite{BRR} when 
the initial state is uniform product measure.

A binary CA is {\it additive\/} when the local map $\phi$ adds all its arguments 
modulo $2$.  In this case the Lyapunov profile is proper and independent of 
$\ca_0$ and $\De_0$. 
As it depends only on the neighborhood $\cN$, we denote the resulting Lyapunov profile by $L_\cN$. 
By elementary large deviations 
\cite{DZ, RS}, we can give it as a variational formula. For $y\in \bR^d$, let
$$
\Lambda(y)=\sum_{x\in \cN}\exp(\langle y, x\rangle).
$$
Then $L_\cN$ is given by the Legendre transform 
$$
L_\cN(\alpha)=\inf_{y}\left(-\langle y, \alpha\rangle + \Lambda(y)\right).
$$
Furthermore, $\lambda(\alpha)=\log|\cN|$ with the unique MLE direction given by the average of $\cN$:
$
|\cN|^{-1}\sum_{x\in \cN} x. 
$
For example, {\it Rule 150\/} is the one-dimensional additive CA with $\cN=\{0, \pm 1\}$ 
and has 
$$
L(\alpha)=L_\cN(\alpha)=\log\left(1+\alpha_0+\frac{1}{\alpha_0}\right)-\alpha\log\alpha_0, \text{ where }
\alpha_0=\frac{\alpha+\sqrt{4-3\alpha^2}}{2(1-\alpha)},
$$
with MLE $\lambda=\log 3$ and MLE direction $0$. Clearly, for any CA with neighborhood $\cN$, 
and any initialization $\ca_0$ and $\De_0$, 
$$
L(\alpha)\le L_\cN(\alpha).
$$
In this sense, the additive CA are the most unstable.

We also remark that, for additive rules, the theorem due to Badahur and Rao (see Section 
3.7 of \cite{DZ}) implies that for a fixed $\epsilon$ the $t$-limit in (\ref{lyapunov-profile-def})
exists and the convergence rate is $\mathcal O(t^{-1}\log t)$. Periodic cases 
(Section~\ref{periodic-profiles}) and chaotic rules with strong mixing 
properties (e.g., rules {\it 22\/}, {\it 30\/},
and {\it 106\/} among elementary CA)
appear to exhibit 
similarly fast convergence, while many other cases progress more slowly 
due to the fact that $\xi_t$ itself does so. In our empirical Lyapunov 
profile plots from random initial states, we choose $t=10^5$ and $\epsilon=4\cdot 10^{-3}$; 
we do not add the huge numbers of defects using exact integer arithmetic but 
instead use double precision to compute their logarithms using this formula for $0<B\le A$: 
$$
\log(A+B)=\log A+\log(1+\exp(\log B-\log A)).
$$

\subsection{Density profiles and defect shapes}\label{prelim-density}

Due to stable updates, the set of defect sites often has holes that are 
invisible in the Lyapunov profile $L$. To capture this information,
we introduce the function $\rho=\rho(\alpha)$ that 
gives the proportion of defect sites in the direction $\alpha\in \bR^d$, that is, on the rays $x=\alpha t$. 
Formally, we call $\rho$ the {\it defect density profile\/} if,
as $T\to\infty$, the measures given by properly scaled point-masses 
at $x/t$, for $(x,t)$ with $t\le T$ and $\de(x,t)=1$, 
converge to $\rho$ in the following sense:
\begin{equation}\label{density-profile-def}
\frac{2^d}{T^{d+1}}\sum_{(x,t):t\le T, \de(x,t)=1}\psi(x/t)\xrightarrow[T\to\infty]{} \int_{\bR^d} \rho(\alpha)\psi(\alpha)\,d\alpha,
\end{equation}
for any test function $\psi\in \cC_c(\bR^d)$. (Note that this convergence is in the weak$^*$-topology
used in functional analysis.) The scaling is chosen so that, 
when $\de\equiv 1$, $\rho\equiv 1$.  
See \cite{GG2, GG4, GG5} for 
other examples of density profiles. 

Furthermore, we define the {\it defect shape\/} $W$ to be the closed subset $\bR^d$ obtained 
by the following limit
in the Hausdorff sense, 
\begin{equation}\label{defect-shape-def}
W=\lim_{t\to\infty} \frac 1t\{x: \de(t,x)=1\}
\end{equation}
provided the limit exists. If $\de_t=\emptyset$ for some $t$,
then we let $W=\emptyset$. Observe that the support of the measure $\rho\,d\alpha$ 
with density $\rho$ is included in $W$, but does not necessarily equal $W$.
For example, $\{x: \de(t,x)=1\}$ 
could be the singleton $\{0\}$ (e.g., for the identity CA), resulting in $W=\{0\}$ but $\rho\equiv 0$.
On the other hand, the following result is easy to prove. 

\begin{prop} \label{L-W}
If $W$ exists, then 
$$
W=\{\alpha: L(\alpha)\ge 0\}.
$$
\end{prop}

\begin{proof}
Observe that the set on the right is closed as $L$ is upper semicontinuous. If 
we take any $\gamma>0$, then $\De_t\equiv 0$ on the complement of the fattening 
$W^\gamma$ for large enough $t$; therefore $L|_{(W^\gamma)^c}\equiv-\infty$, 
and then $L|_{W^c}\equiv-\infty$.
On the other hand, for any $\alpha\in W$, there exists a sequence 
of space-time points $(x_n,t_n)$ so that $\de_{t_n}(x_n)=1$ and 
$x_n/t_n\to\alpha$. Then for any $\epsilon>0$, 
$\sum_{||x/t-\alpha||<\epsilon}\De(x,t)\ge 1$ for large enough $t$, 
thus $L(\alpha)\ge 0$. 
\end{proof}

\subsection{Dependence of the initialization, 
and classification of CA trajectories}\label{prelim-initial}

In general, $L$ depends on both the CA initial state $\ca_0=\eta$ and the defect 
initial state $\De_0=A$. We make this dependence explicit by the notation 
$L_A^\eta$. It is clear that $L_{A_1}^\eta\le L_{A_2}^\eta$ whenever 
$A_1\subset A_2$, therefore the limit 
$$
L_\infty=L_\infty^\eta=\lim_{n\to\infty} L^\eta_{[-n,n]^d}
$$
exists. The importance of this object is explained in our next result. 

\begin{theorem}\label{L-infty-existence}
Assume $\eta$ is sampled from an ergodic measure on $\bZ^d$. Then 
there exists a deterministic upper semicontinuous function $\overline L$ so that
$$
L_\infty^\eta=\overline L
$$
almost surely. 
\end{theorem}

\begin{proof} All our functions will be defined on a large enough closed ball
within $\bR^d$, as the density profile is (deterministically) 
$-\infty$ outside the convex hull $\text{co}(\cN)$. 
Choose a countable set $\cF$ of continuous functions so that 
$G=\inf\{f\in \cF: f\ge G\}$ for every upper semicontinuous function $G$. 

The main observation is that the set $\{\eta: L_\infty^\eta\le y\}$ is 
translation invariant, that is, contains together with any $\eta$ all
its translations. By ergodicity, the probability of any such 
set is $0$ or $1$. For an $f\in \cF$, let  
$$\Omega_f=\{\eta:L_\infty^\eta(\alpha)\le f(\alpha)\text{ for every }
\alpha\}.
$$
Then 
$$
\P(\Omega_f)\in \{0,1\}
$$ 
for every $f\in \cF$. Let $\cF_0, \cF_1\subset \cF$  be the sets 
of functions with respective probabilities $0$ and $1$. The set 
$$
\Omega'=\left(\bigcap_{f\in \cF_1}\Omega_f\right)\cap\left( \bigcap_{f\in \cF_0}\Omega_f^c\right) 
$$
has $\P(\Omega')=1$. For $\eta\in \Omega'$, $\{f\in \cF:L_\infty^\eta\le f\}=\cF_1$. 
Thus, if we define 
$$
\overline L=\inf\{f:f\in \cF_1\}, 
$$
then $\overline L$ is upper semicontinuous and
$
\P(
L_\infty^\eta = \overline L)=1.
$
\end{proof}

As is the convention, we will therefore assume that $L_\infty^\eta$ is a determinstic 
function, by redefining it on the set of measure 0. 
In this fashion, we also define the deterministic closed set $W_\infty^\eta$ and the MLE $\lambda_\infty^\eta$. Again, we drop the superscript when the initial measure 
is understood from the context. 

For a given pair $\ca_0=\eta$ and $\De_0=A$, we call the defect accumulation 
dynamics: 
\begin{itemize}
\item  {\it expansive\/} if $L_A^\eta>0$ on a nonempty open set;
\item  {\it collapsing\/} if $L_A^\eta\equiv -\infty$; and 
\item  {\it marginal\/} otherwise. 
\end{itemize}
When $\ca_0$ is a product measure with a fixed density $p$, 
the above characterizations will refer to $L_\infty$. 
When not explicitly stated otherwise, the 
initialization is the uniform product measure, which has density $p=1/2$. 
With this default initial data, the above classification only 
depends on the rule, and in this context we refer to the CA itself as 
expansive, collapsing, or marginal, often
by the respective initial E, C, or M. We consider other densities 
$p\in (0,1)$ in Sections~\ref{prelim-defdam} and~\ref{ECA-subsection-dendep}.

\section{Defect dynamics vs.~damage spreading}\label{prelim-defdam}
The impetus to consider the defect shape $W$ comes from Wolfram's 
original concept of damage spreading \cite{Wol1, Gra2}, discussed in 
Section~\ref{intro}. We now provide a formal definition and briefly contrast the two notions. 
The {\it damage CA\/} is yet another ``second class'' dynamics on the 
trajectory $\ca_t$, given by the set of damaged sites $\damage_t\in \{0,1\}^{\bZ^d}$
and the recursive rule (in which addition and reduction $\mmod 2$ are sitewise)
$$
\damage_{t+1}=\left(\Phi((\ca_t+\damage_t)\mmod 2)+\ca_t\right) \mmod 2
$$
that records which updates are affected by the currently damaged sites.
We define the corresponding {\it damage shape\/} $W_\damage$ and {\it damage density
profile\/} $\rho_\damage$ analogously to (\ref{defect-shape-def}) and 
(\ref{density-profile-def}), respectively.

\vspace{-0cm}

\begin{figure}[ht]
\begin{center}
\includegraphics[trim=0cm 0cm 0cm 1cm, clip, width=5cm]{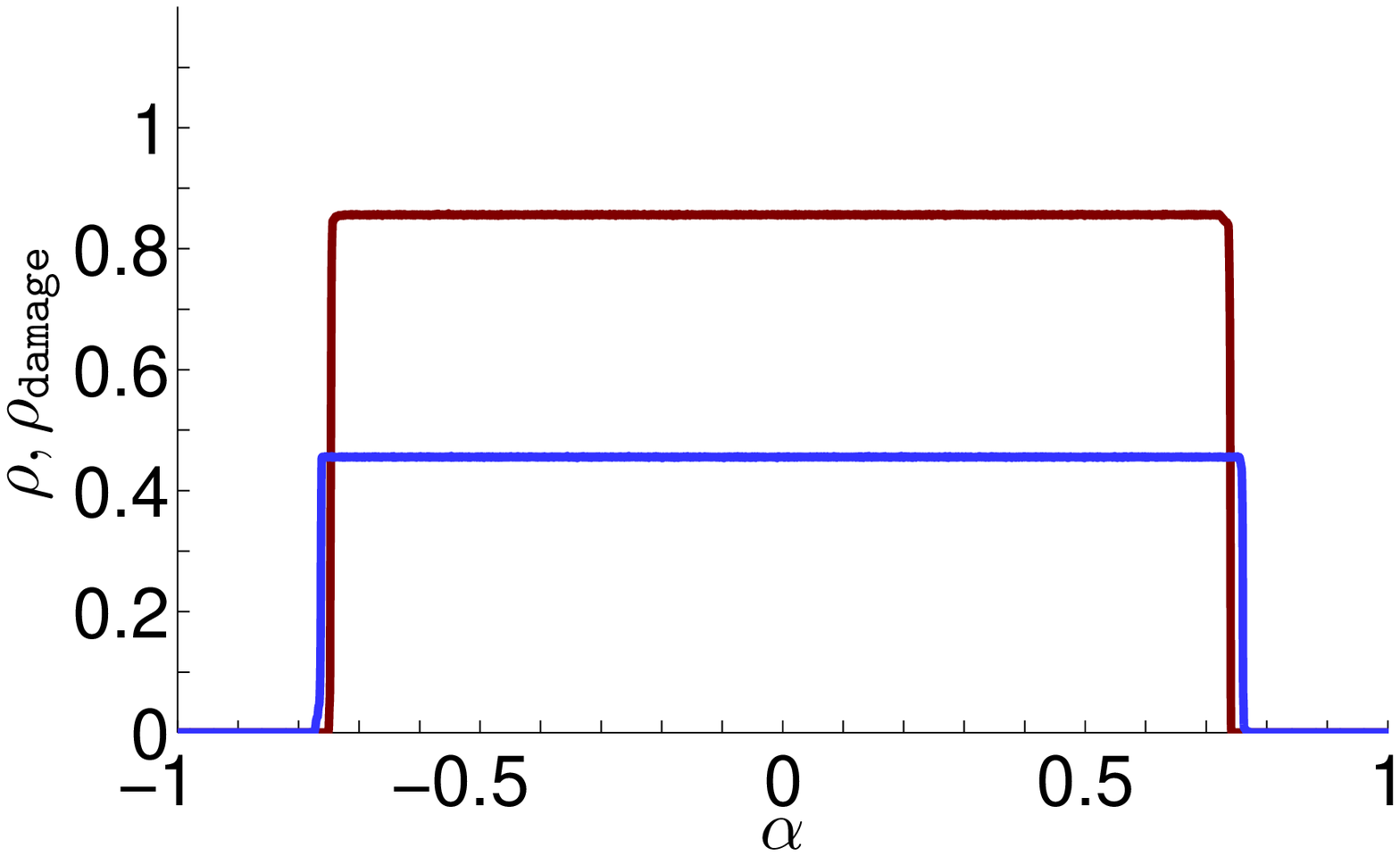}
\hspace{-0cm}
\includegraphics[trim=0cm 0cm 0cm 1cm, clip, width=5cm]{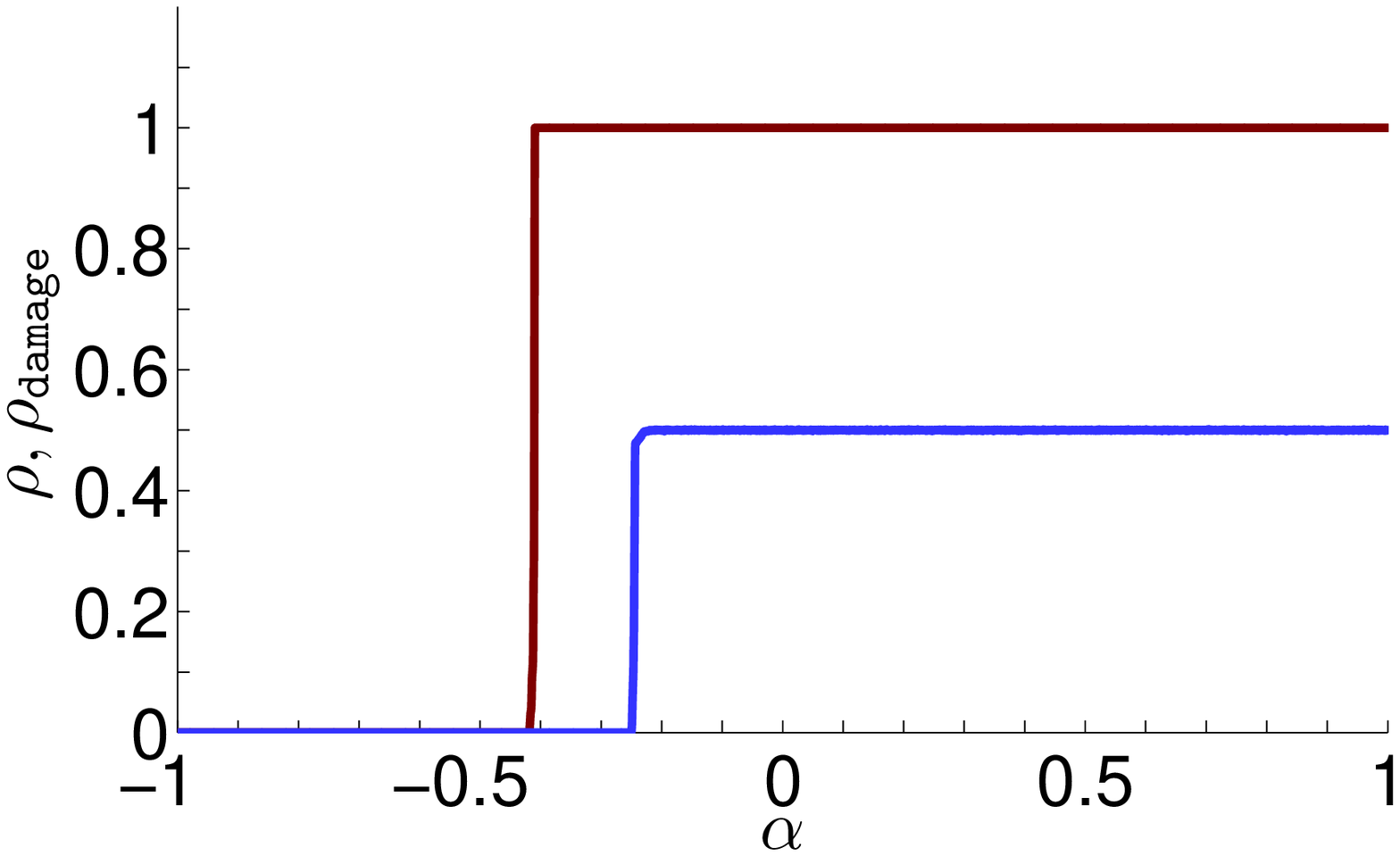}
\hspace{-0cm}
\includegraphics[trim=0cm 0cm 0cm 1cm, clip, width=5cm]{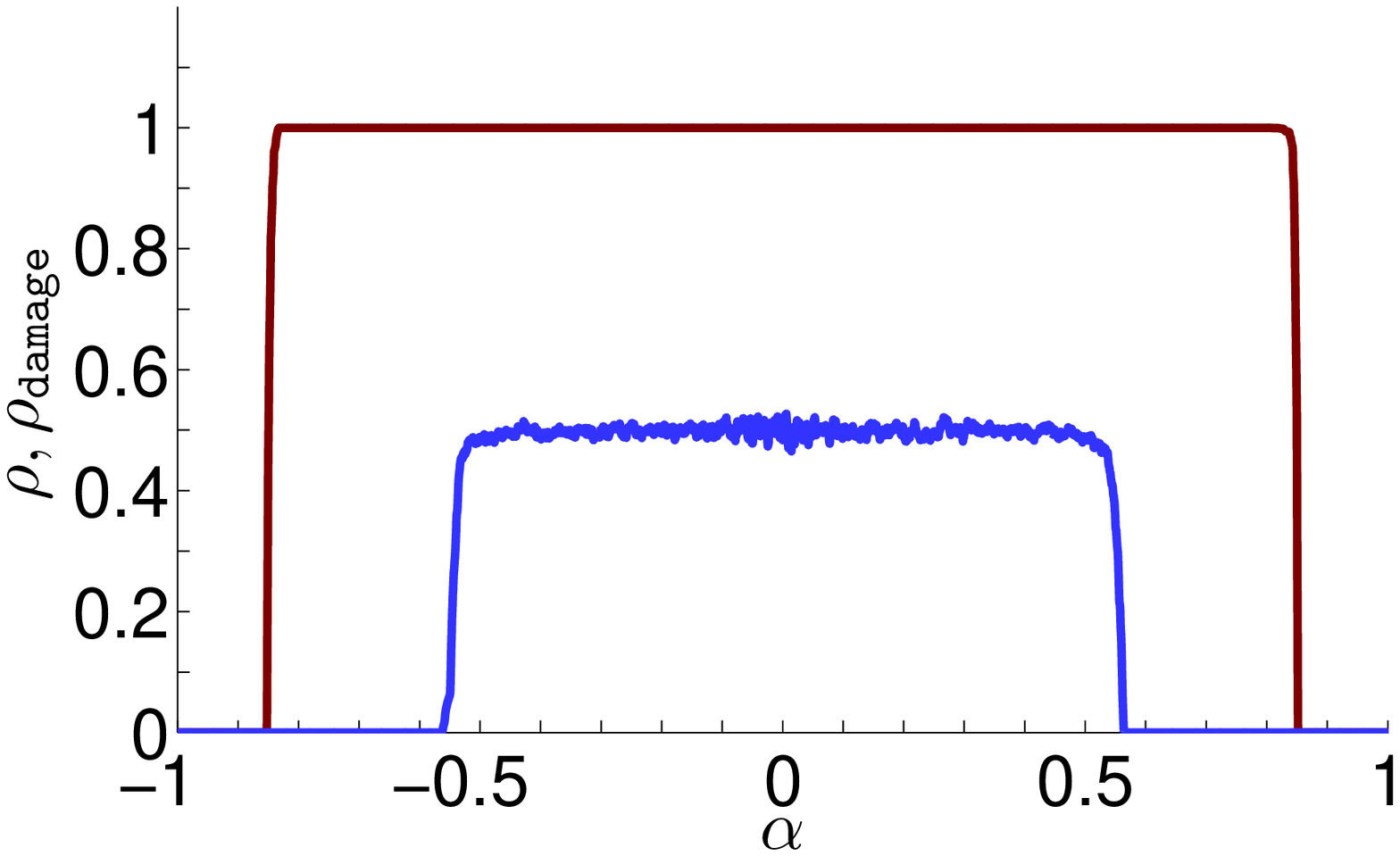}
\end{center}
\vspace{-0.5cm}
\caption{Empirical defect (dark red) and damage (light blue) density profiles at time $10^5$ 
for rules {\it 22\/}, {\it 30\/}, and {\it 54\/}.}\label{damage-defect-figure}
\end{figure}

To compare the damage and defect dynamics, we will assume they initially 
agree, i.e., that $\damage_0=\de_0=\De_0$ is a finite set. 
The dynamics $\de_t$ of defect sites only tracks one-site 
perturbations of $\ca_t$, while $\damage_t$ performs simultaneous 
changes at all perturbed sites, so
there might be significant difference between the 
two. Three examples of density and damage profiles
started from a uniform product measure are in 
Fig.~\ref{damage-defect-figure}. Observe that for {\it Rule 22\/} 
$\rho_\damage<\rho$ but $W\subsetneqq W_\damage$; in fact $W_\damage$ has
edges at about $\pm 0.77$ \cite{Gra2}, while those of $W$ lag behind by about $0.025$.
Another CA for which $\de_t$ similarly lags behind
$\damage_t$ is {\it Rule 122\/}, but in this instance  
the empirical evidence indicates that the difference disappears in 
the limit, as  $W= W_\damage=[-1,1]$.
On the other hand, two chaotic examples for which $W_\damage\subsetneqq W$ are also
included in Fig.~\ref{damage-defect-figure}. We also remark that, 
for additive rules such as {\it Rule 150\/}, 
$W_\damage$ does not exist due to the fractal evolution of $\damage_t$,  which is, for the same reason, much smaller than $\de_t$ for most
(but not all) times $t$.

Assume now that the initial state is more general, 
a product measure with density $p$.  For elementary CA, 
we address the dependence of defect accumulation dynamics on $p$ in 
Section~\ref{ECA-subsection-dendep}.  In this setting, rules
with significant variation in $p$ coincide with rules in which $W$ is an interval of positive length 
while $W_\damage$ is at most a singleton for all $p\in (0,1)$.  
(See Proposition~\ref{rule-38} for a formal proof in case of {\it Rule 38\/}.) 
This equivalence is interesting enough 
for a thorough theoretical development, which we do not attempt here. Instead, we 
provide a definition and a non-rigorous explanation next.

We call a one-dimensional CA trajectory $\ca_t$ {\it striped\/} (resp., {\it degenerate\/})
if there exist a translation 
number $v_0\in \bZ$, a delay time $t_0\ge 1$, an initial time $t_i\ge 0$, and an  
$\epsilon>0$ so that $\ca_{t+t_0}(x)=\xi_t(x-v_0)$ (resp., $\ca_{t}(x+1)=\ca_t(x)$)  
for $t\ge t_i$ and $x\in[(\inf\cN-\epsilon) t, 
(\sup\cN+\epsilon) t]$.  
 
A {\it stripes\/} CA is one whose trajectory is almost surely striped 
and non-degenerate for any initial product measure with 
density $p\in (p_1,p_2)$. Here, $(p_1,p_2)$ is a nonempty interval of densities
which is, when unspecified, assumed to be $(0,1)$.
For such CA,
the statistical properties of the invariant striped state typically  
depend on $p$. Consequently, 
if a stripes CA is expansive, then 
we expect that the Lyapunov profile also varies with $p$. 
On the other hand, it is easy to see that if
$\ca_t$ and its perturbation $(\ca_t+\damage_t)\mmod 2$ are both striped, 
$\damage_t$ remains bounded. For product measures, a striped trajectory typically 
results from transient structures that are eroded away at exponential rate, and this 
property cannot be changed by a finite perturbation. For such trajectories, 
$W_\damage$ is at most a singleton. Therefore, the equivalence discussed
above is a consequence of the fact that 
all expansive elementary CA started from product measures 
are either attracted to a chaotic or complex state for any density $p\in (0,1)$, 
or are stripes CA. We now discuss two examples 
with $\cN=\{0,\pm 1,\pm 2\}$ that show that there are other possibilities for general CA.

The first CA is simple: the update rule is
$abcde\mapsto 1$ if and only if $abcde$ includes $010$ as a substring. The resulting
global rule $\Phi$ satisfies $\Phi^2=0$, as for any $\ca_0$ 
there are no isolated $1$s at time $t=1$ and then no
$1$s at all at time $t=2$. This is a degenerate case, and indeed $W_\damage=\emptyset$, but $W=[-1,1]$
and $\lambda_\infty=\log 3$ for all initial states (as the defect dynamics
coincides with that for {\it Rule 150\/} from time $1$ on). In particular, there is 
no dependence on $p$ but very large discrepancy between $W$ and $W_\damage$. 

Our second counterexample is a ``particle''
CA that conserves the density of $1$s. A $1$  at $x$ makes a jump to $x+2$ if 
the states in $[x,x+2]$ are $100$ and it makes a jump to $x-1$ if the states at $[x-2,x+1]$ 
are $1011$. 
Simulations make it clear that trajectories from random initializations are not
striped, and that this rule is marginal for small $p$ (with $W=\{2\}$) and 
expansive for large $p$, with a phase transition somewhere between $0.2$ and $0.3$. 
Moreover, $W_\damage=W$ at all $p\in (0,1)$, by
contrast to the dramatic dependence on $p$.

\section{Elementary cellular automata}\label{ECA-section}

In this section we investigate the defect accumulation dynamics 
for the elementary CA, the one-dimensional rules with $\cN=\{-1,0,1\}$. The initial configuration 
$\xi_0$ will be the default uniform product measure, except in 
Section~\ref{ECA-subsection-dendep}, where 
we discuss product measures with other constant densities. 
In these circumstances, the defect dynamics 
remains essentially equivalent if the roles of the two states are switched, or if the rule is replaced 
by its left-right reflection. This leaves us with 88 equivalence classes represented 
by 88 ``minimal'' CA \cite{Vic2}, which we proceed to analyze. The update functions 
for rules featured in our rigorous arguments (here or in 
Section~\ref{periodic-profiles}) are given in Table~\ref{ECA-update}. 
 
\begin{table}[ht!]
\caption{Update functions for some elementary CA.} 
\vspace{0.0cm}
\label{ECA-update}
\centering
\begin{tabular}{| c | c | c| c | c | c |c|c|c|}
			\hline
 \begin{tabular}{@{}c@{}}Rule\end{tabular}
& \begin{tabular}{@{}c@{}}$000$
\end{tabular}
& \begin{tabular}{@{}c@{}}$001$ 
\end{tabular}
& \begin{tabular}{@{}c@{}}$010$
\end{tabular}
& \begin{tabular}{@{}c@{}}$011$
\end{tabular}
& \begin{tabular}{@{}c@{}}$100$
\end{tabular}
& \begin{tabular}{@{}c@{}}$101$ 
\end{tabular}
& \begin{tabular}{@{}c@{}}$110$ 
\end{tabular}
& \begin{tabular}{@{}c@{}}$111$ 
\end{tabular}
  \\
			\hline
			\hline
{\it 22\/} & $0$ & $1$& $1$& $0$& $1$& $0$& $0$& $0$\\ \hline
{\it 27\/}  & $1$ & $1$& $0$& $1$& $1$& $0$& $0$& $0$\\ \hline
{\it 38\/}  & $0$ & $1$& $1$& $0$& $0$& $1$& $0$& $0$\\ \hline
{\it 110\/}  & $0$ & $1$& $1$& $1$& $0$& $1$& $1$& $0$\\ \hline
{\it 152\/}  & $0$ & $0$& $0$& $1$& $1$& $0$& $0$& $1$\\ \hline
\end{tabular}

\end{table}

Many of the $88$ rules are quite transparent and a simple worst case analysis 
as elucidated in our next two theorems yields a rigorous result. The first theorem 
gives the condition under which defect growth is restricted. 

\begin{theorem} \label{wc-ub} Assume that there exist a string $B\in \{0,1\}^b$, $b>0$,
a time $t_B$, and a number $v_B$ with the following property. 
Any pair $(\ca_0,\de_0)$, such that $\ca_0$ equals $B$ on $[0,b-1]$
and  $\de_0$ is 1 exactly on the complement $[0,b-1]^c$, yields $\ca_{t_B}|_{[v_B, v_B+b-1]}=B$ 
and $\de_{t_B}|_{[v_B, v_B+b-1]}\equiv 0$. Then, 
if $\xi_0$ is any translation invariant product measure with $\P(\xi_0(x)=1)\in(0,1)$,  
$L_\infty$ equals $-\infty$ off $\{v_B/t_B\}$. In particular, with such an initialization,
the defect 
accumulation dynamics is not expansive.
\end{theorem}

\begin{proof}
Assume a finite $\de_0$. A translate of $B'$ consisting of 
$t_B$ contiguous copies of 
$B$ (almost surely) exists somewhere to the 
right of the support of $\de_0$. Suppose that, at some time $t$, an interval 
$[x, x+b\cdot t_B-1]$ has the following two properties: 
all defects are to its left; and it
is occupied by a translate of $B'$. As defects cannot advance faster than by distance $1$
at each time step, and by the hypotheses, the interval  
$[x+v_B, x+b\cdot t_B-1+v_B]$ has the same properties at time $t+t_B$. It follows 
that $\de_t \subset(-\infty, N+ t\cdot v_B/t_B]$ for all $t\ge 0$ and 
some a.s.~finite random 
variable $N$. Consequently, $W\subset (-\infty, v_B/t_B]$ a.s. As this is true 
for any finite $\de_0$, $L_\infty\equiv-\infty$ on $(v_B/t_B, \infty)$. 
An analogous argument shows that the same holds for $(-\infty, v_B/t_B)$
\end{proof}

If $B$ and $t_B$ are fixed, the property required by Theorem~\ref{wc-ub}, 
can be checked by a finite verification. Namely, to look for 
all possible $v_B$, all $2^{4t_B}$ possible initial configurations 
in $2t_B$ sites both to the left and to the right of $B$ are generated 
and then the dynamics is run to the time $t_B$. If it happens that 
$B$ occurs at {\it two\/} (or more) distinct intervals of $b$ sites at time $t_B$, 
then Theorem~\ref{wc-ub} implies the rule is collapsing. 

We now state a general result in the opposite direction, i.e., we give a condition 
that guarantees defect expansion. Recall that $L_{\cM}$ is the Lyapunov profile for the additive dynamics 
with neighborhood $\cM$.

\begin{theorem} \label{wc-lb} Assume that there exist a set $\cM\subset \bZ$ 
with at least two points
and a time $t_\cM$ with 
the following property: for $\de_0=\ind(0)$ and arbitrary $\ca_0$, 
$\de_{t_{\cM}}\equiv 1$ on $\cM$. Then 
$$
L_\infty(\alpha)\ge \frac 1{t_\cM}L_{\cM}(t_\cM\alpha).
$$
In particular, the defect accumulation 
dynamics is expansive. 
\end{theorem}

\begin{proof} This follows from a simple induction argument. 
\end{proof}

\subsection{Elementary CA with provably collapsing defect dynamics}\label{ECA-subsection-collapsing}

Theorem~\ref{wc-ub} implies defect collapse for the $8$ rules listed in 
Table~\ref{ECA-collapsing}.  
 
\begin{table}[ht!]
\caption{The $8$ provably collapsing rules.} 
\vspace{0.0cm}
\label{ECA-collapsing}
\centering
\begin{tabular}{| c | c | c|}
			\hline
 \begin{tabular}{@{}c@{}}Rule\end{tabular}
& \begin{tabular}{@{}c@{}}class 
\end{tabular}
& \begin{tabular}{@{}c@{}}proof 
\end{tabular}
  \\
			\hline
			\hline
{\it 0\/} & C & trivial\\ \hline
{\it 8\/} & C & $B=0$, $t_B=1$, $v_B=-1,0$ \\ \hline
{\it 32\/} & C &$B=0$, $t_B=1$, $v_B=\pm 1$ \\ \hline
{\it 40\/} & C &$B=00$, $t_B=1$, $v_B=-1,0$\\ \hline
{\it 128\/} & C & $B=0$, $t_B=1$, $v_B=-1,0,1$ \\ \hline
{\it 136\/} & C & $B=0$, $t_B=1$, $v_B=-1,0$\\ \hline
{\it 160\/} & C & $B=0$, $t_B=1$, $v_B=\pm1$ \\ \hline
{\it 168\/} & C & $B=00$, $t_B=1$, $v_B=-1,0$\\ \hline
\end{tabular}

\end{table}

\subsection{Elementary CA with provably marginal defect dynamics}\label{ECA-subsection-marginal}
 
The rules for which we are able to verify the hypotheses of 
Theorem~\ref{wc-ub} to prove marginal defect dynamics are listed in the Table~\ref{ECA-marginal}. 
We do not provide the arguments that these cases are indeed not collapsing; 
these can be obtained at a glimpse from examples generated by random
initial states (e.g., see Fig.~\ref{intro-figure} for {\it Rule 7\/}).
The MLE directions 
are given by application of Theorem~\ref{wc-ub}, while approximate MLE values are 
based on empirical evidence: we ran a random configuration with an interval of
$10^3$ defects for $10^5$ time steps. However, as
we have not attempted a rigorous determination, 
it is possible that rare favorable configurations result in values higher 
than we obtained. For example, {\it Rule 73\/} seems a good candidate for this 
to occur.

\begin{center}
\begin{longtable}{|c|c|c|c|c|}
\caption{The $46$ rules with provably marginal defect accumulation dynamics.}
\label{ECA-marginal} \\

\hline \multicolumn{1}{|c|}{Rule} & \multicolumn{1}{c|}{class} & \multicolumn{1}{c|}{proof} 
& \multicolumn{1}{c|}{MLE dir.} & \multicolumn{1}{c|}{MLE}
\\ \hline \hline
\endfirsthead

\multicolumn{3}{c}%
{{
\tablename\ \thetable{} --- {\it continued from previous page}}} \\
\hline \multicolumn{1}{|c|}{Rule} & \multicolumn{1}{c|}{class} & \multicolumn{1}{c|}{proof} 
& \multicolumn{1}{c|}{MLE dir.} & \multicolumn{1}{c|}{MLE}
\\ \hline \hline
\endhead

 \multicolumn{3}{r}{{\it Continued on next page}} \\ 
\endfoot

\hline 
\endlastfoot
{\it 1\/} & M & $B=1$, $t_B=2$, $v_B=0$& $0$ & $0.55$\\ \hline
{\it 2\/}  & M & $B=0$, $t_B=1$, $v_B=-1$& $-1$ & $0$\\ \hline
{\it 3\/}  & M & $B=00$, $t_B=2$, $v_B=1$& $1/2$ & $0.35$ \\ \hline
{\it 4\/} & M & $B=0$, $t_B=1$, $v_B=0$& $0$ & $0$ \\ \hline
{\it 5\/} &  M & $B=1$, $t_B=2$, $v_B=0$& $0$ & $0.35$\\ \hline
{\it 7\/} &  M & $B=11$, $t_B=2$, $v_B=1$& $1/2$ & $0.35$ \\ \hline
{\it 10\/} &  M & $B=0$, $t_B=1$, $v_B=-1$& $-1$ & $0$\\ \hline
{\it 12\/} &  M & $B=0$, $t_B=1$, $v_B=0$& $0$ & $0$\\ \hline
{\it 13\/} &  M & $B=01$, $t_B=1$, $v_B=0$& $0$ & $0.48$\\ \hline
{\it 15\/} &  M & $B=0$, $t_B=2$, $v_B=2$ (right shift w.~toggle)& $1$ & $0$\\ \hline
{\it 19\/} &  M & $B=00$, $t_B=2$, $v_B=0$& $0$ & $0.35$\\ \hline
{\it 23\/} &  M & $B=00$, $t_B=2$, $v_B=0$& $0$ & $0.69$\\ \hline
{\it 24\/} &  M & Prop.~\ref{equivalent-rules}& $1$ & $0$\\ \hline
{\it 27\/} & M & Prop.~\ref{rule27-marginal}& $1/2$ & $0$\\ \hline
{\it 28\/} &  M & $B=01$, $t_B=1$, $v_B=0$& $0$ & $0.48$\\ \hline
{\it 29\/} &  M & $B=01$, $t_B=1$, $v_B=0$& $0$ & $0.35$\\ \hline
{\it 33\/} &  M & Prop.~\ref{equivalent-rules} & $0$ & $0.66$\\ \hline
{\it 34\/} &  M & $B=0$, $t_B=1$, $v_B=-1$& $-1$ & $0$\\ \hline
{\it 36\/} &  M & $B=00$, $t_B=1$, $v_B=0$& $0$ & $0$\\ \hline
{\it 42\/} &  M & $B=0$, $t_B=1$, $v_B=-1$& $-1$ & $0$\\ \hline
{\it 44\/} &  M & $B=00$, $t_B=1$, $v_B=0$& $0$ & $0.48$\\ \hline
{\it 46\/} &  M & Prop.~\ref{equivalent-rules}         & $-1$ & $0$\\ \hline
{\it 50\/} &  M & $B=01$, $t_B=2$, $v_B=0$& $0$ & $0.48$\\ \hline
{\it 51\/} &  M & $B=0$, $t_B=2$, $v_B=0$ (toggle)& $0$ & $0$\\ \hline
{\it 72\/} &  M & $B=0$, $t_B=1$, $v_B=0$& $0$ & $0.69$\\ \hline
{\it 73\/} &  M & $B=0110$, $t_B=1$, $v_B=0$& $0$ & $0.91$\\ \hline
{\it 76\/} &  M & $B=0$, $t_B=1$, $v_B=0$& $0$ & $0$\\ \hline
{\it 77\/} &  M & $B=01$, $t_B=1$, $v_B=0$& $0$ & $0.69$\\ \hline
{\it 78\/} &  M & $B=10$, $t_B=1$, $v_B=0$& $0$ & $0.48$\\ \hline
{\it 94\/} &  M & $B=101$, $t_B=1$, $v_B=0$& $0$ & $0.61$\\ \hline
{\it 104\/} &  M & $B=00$, $t_B=1$, $v_B=0$& $0$ & $0.69$\\ \hline
{\it 108\/} &  M & $B=00$, $t_B=1$, $v_B=0$& $0$ & $0.86$\\ \hline
{\it 130\/} &  M & $B=0$, $t_B=1$, $v_B=-1$& $-1$ & $0$\\ \hline
{\it 132\/} &  M & $B=0$, $t_B=1$, $v_B=0$& $0$ & $0$\\ \hline
{\it 138\/} &  M & $B=0$, $t_B=1$, $v_B=-1$& $-1$ & $0$\\ \hline
{\it 140\/} &  M & $B=0$, $t_B=1$, $v_B=0$& $0$ & $0$\\ \hline
{\it 152\/} &  M & Prop.~\ref{rule152-marginal} & $1$ & $0$\\ \hline
{\it 156\/} &  M & $B=01$, $t_B=1$, $v_B=0$& $0$ & $0.69$\\ \hline
{\it 162\/} &  M & $B=0$, $t_B=1$, $v_B=-1$& $-1$ & $0$\\ \hline
{\it 164\/} &  M & $B=00$, $t_B=1$, $v_B=0$& $0$ & $0$\\ \hline
{\it 170\/} &  M & $B=0$, $t_B=1$, $v_B=-1$ (left shift) & $-1$ & $0$\\ \hline
{\it 172\/} &  M & $B=00$, $t_B=1$, $v_B=0$& $0$ & $0.48$\\ \hline
{\it 178\/} &  M & $B=01$, $t_B=2$, $v_B=0$& $0$ & $0.69$\\ \hline
{\it 200\/} &  M & $B=0$, $t_B=1$, $v_B=0$& $0$ & $0.69$\\ \hline
{\it 204\/} &  M & $B=0$, $t_B=1$, $v_B=0$ (identity)& $0$ & $0$\\ \hline
{\it 232\/} &  M & $B=00$, $t_B=1$, $v_B=0$& $0$ & $0.69$\\ 
\end{longtable}
\end{center}

For some rules, Theorem~\ref{wc-ub} does not apply directly but only after a transient 
period; we collect the necessary properties in our next three results. We remark that agreement 
of the dynamics of two CA after a transient time does not necessarily imply that their defect 
accumulation dynamics agree. 

\begin{prop}\label{equivalent-rules}
The following hold for arbitrary initial states:
\begin{enumerate}
\item {\it Rule 24\/}: All $1$s are isolated at time $t=1$; thereafter, the 
CA evolves as {\it Rule 2\/}. 
\item {\it Rule 33\/}: Every isolated $0$ at $(x,t)$, $t\ge 1$, requires two 
isolated $0$s at $(x\pm 1, t-1)$. If a configuration has no isolated $0$s, the 
CA evolves as {\it Rule 1\/}. 
\item  {\it Rule 46\/}:  There is no isolated $1$ at time $t=1$; thereafter, the 
CA evolves as {\it Rule 42\/}.
\end{enumerate}
\end{prop}

\begin{proof}
These are all straightforward verifications.
\end{proof}

\begin{prop}\label{rule152-marginal}
Assume the CA is {\it Rule 152\/}.  States $11$ at $(x,t)$, $(x+1,t)$, $t\ge 1$ require $111$ 
at $(x,t)$, $(x+1,t)$, $(x+2,t)$; if a configuration has only isolated $1$s the CA evolves 
as {\it Rule 16\/}, which is equivalent, via a left-right reflection, to {\it Rule 2\/}. 
Furthermore, if $\ca_0$ is the 
uniform product measure, then almost surely there exists an $x$ such that there is no $11$ 
in $[x-2+t, x+t+2]$ for all $t$. Consequently, {\it Rule 152\/} is marginal.
\end{prop}

\begin{proof}
These are simple checks, other than the last statement. To prove the latter, 
let $A_x$ be the event that the initial configuration is $00000$ in $[x,x+4]$ and that, 
for every $n\ge 0$, the interval $[x+5+n, x+5+2n]$ contains
at least one $0$. 
It suffices to show that 
\begin{equation}\label{equivalent-rules-eq1}
\P(A_x\text{ happens i.o.~for }x\ge 0)=\P(A_x\text{ happens i.o.~for }x\le 0)=1.
\end{equation}
Let $B_x$ be the event that $[x,x+4]$ contains only $0$s 
and that the following holds for any interval 
$I_{x,k}=[x+5+2^k, x+5+2^{k+1}-1]$ of length $2^k$: if $0\le k\le 4$, the
entire $I_{x,k}$ is covered by $0$s; and if $k>4$, each of the four disjoint subintervals 
of $I_{x,k}$ of length $2^{k-4}$ contains at least one $0$. We claim that 
$B_x\subset A_x$. 
Indeed, if $2^k\le n< 2^{k+1}$, then the interval $[x+5+n, x+5+2n]$ has its left 
endpoint in $I_{x,k}$ and length at least $2^k+1$. Then it either covers the right half
of $I_{x,k}$ or the left quarter of $I_{x,k+1}$. 

Now, let
$$
a=\P(B_0)=2^{-20}\prod_{k=5}^{\infty} \left(1-2^{-2^{k-4}}\right)^4>0.
$$
Then $\P(B_x)=a$ for every $x$. Moreover, for a large $r$, chose 
the largest $\ell$ so that $r\ge 5+2^\ell$; then 
$$
a\le \P(B_{x}|B_{x+r})\le \frac{a}{\prod_{k\ge \ell}  \left(1-2^{-2^{k-4}}\right)^4}\le a (1+c2^{-r}), 
$$
for some constant $c>0$. The second moment method now easily proves (\ref{equivalent-rules-eq1}) 
with $B_x$ in place of $A_x$ and ends the proof. 
\end{proof}

\begin{prop}\label{rule27-marginal} Assume the CA is 
{\it Rule 27\/}. Assume  that $\ca_0$ and $\de_0$ both vanish 
on $[a,b]$, where $b-a\ge 5$. 
Then for all even $t$, $\ca_t$ and $\de_t$ both vanish on 
$[a+t/2, b+t/2-4]$. Consequently, this rule is marginal.
\end{prop}

\begin{proof} We begin with a few observations. Assume that $t\ge 1$ and that the pair configuration 
$10$, underlined in (\ref{rule27-pred0}), appears in $\xi_t$. Then 
there are two possibilities for the nearby states in $\xi_{t-1}$ (represented by the 
top line) and $\xi_t$, as depicted in
(\ref{rule27-pred0}). An analogous property, also given in (\ref{rule27-pred0}), holds for the pair $01$.
\begin{equation}\label{rule27-pred0}
\begin{aligned}
&011&\quad 0010         &\qquad\qquad&1011&\;\;\quad 100\\[-2mm]
&\phantom{0}\underline{10}&\quad 1\underline{10}&\qquad&0\underline{01}&\;\;\quad \underline{01}1
\end{aligned}
\end{equation}
It immediately follows that $1010$ is only possible in the initial state. 
Assume next that $0101$ occurs in $[1,4]$ in $\ca_t$. 
Then we claim that for any $k\ge 0$ and time $t-2k\ge 0$, the configuration 
in $[t-k, t+2k+4]$ is 
\renewcommand{\lsquare}{\text{\scalebox{3}[1]{$\square$}}}
\begin{equation}\label{rule27-pred}
\lsquare\ldots\lsquare00101
\end{equation}
where there are $k$ $\lsquare$ blocks of length $3$, each containing either $001$ 
or $011$. We also claim that at time $t-2k-1$ the configuration at $[t-k, t+2k+5]$ 
must be 
\begin{equation}\label{rule27-pred1}
\lsquare\ldots\lsquare101100
\end{equation}
where now each of the $k$ $\lsquare$ blocks of length $3$ contains either $100$ 
or $101$. Our induction hypothesis is that both (\ref{rule27-pred}--\ref{rule27-pred1})
are satisfied at each $k\ge 0$. For $k=0$, this is an easy verification 
using (\ref{rule27-pred0}).
The induction step is also straightforward using the fact that the update rule satisfies
$00*\mapsto1$, $*10\mapsto 0$, 
and $1\!*\!1\mapsto 0$.

We now state four key facts. The first two are about the original CA and the next two 
about the defect percolation CA. The first fact follows from the claim above, 
while the remaining three are straightforward. 
\begin{itemize}
\item
As (\ref{rule27-pred}) does not contain $000$,  
if $\ca_0$ vanishes on $[x,x+2]$ in $\ca_0$, then the state of $\ca_t$ cannot 
contain $0101$ on
the interval $[x-t, x+t/2+3]$ for any even $t\ge 0$. 
\item
Suppose
that $\ca_0(0)=0$ and the five state configuration of $\ca_0$ in $[-1,3]$ 
contains neither  $0101$ nor $1010$. Then $\ca_2(1)=0$. 
\item
If $\ca_0$ vanishes on $[0,1]$ and $\de_0(0)=0$, then $\ca_2(1)=0$ and 
$\de_2(1)=0$. 

\item If $\ca_0$ and $\de_0$ both vanish on $[0,1]$, 
and $\ca_0$ is not $101$ on $[2,4]$, then $\ca_2$ vanishes on 
$[1,2]$ and $\de_2(1)=0$.  
\end{itemize}
The above four facts establish the claimed ``non-invasion'' of the interval of 
$0$s in the statement, and marginality easily follows. 
\end{proof}

\subsection{Elementary CA with expansive defect dynamics}\label{ECA-subsection-expansive}

There is overwhelming empirical evidence that the 22 rules in Table~\ref{ECA-expansive}
are expansive. For nine of these cases we provide a proof: four are additive
or nearly additive (rules {\it 60\/}, {\it 90\/}, {\it 105\/}, and {\it 150\/}), four 
more are handled by Theorem~\ref{wc-lb}  (rules {\it 30\/}, {\it 45\/}, {\it 54\/}, and {\it 57\/}),
and {\it Rule 38\/} is the subject of 
our next result. This last rule is a stripes CA, as a 
disordered state self-organizes into a {\it random\/} configuration which is merely shifted 
(see Section~\ref{prelim-defdam} for a formal definition). With some confidence we conjecture (although we do not have a
proof) that rules {\it 6\/}, {\it 25\/}, 
{\it 26\/},  {\it 41\/}, {\it 57\/}, {\it 62\/}, 
{\it 134\/} and {\it 154\/} are also stripes CA. In
Section~\ref{ECA-subsection-dendep}, we will see that 
these rules are also characterized by the dependence of MLE on 
the initial density of $1$s in $\ca_0$, as expected from the discussion 
in Section~\ref{prelim-defdam}.

Table~\ref{ECA-expansive} gives (in most cases empirical) estimates of the MLE, its direction, 
defect shape $W$, and the defect density $\rho$ on $W$, which appears constant in all cases. 
Fig.~\ref{figure30-106} depicts Lyapunov profiles for {\it Rule 30\/} and {\it Rule 106\/}, two rules 
that leave the uniform product measure invariant. See Section~\ref{ECA-subsection-dendep}  
for a discussion on {\it Rule 62\/}.

\begin{figure}[!ht]
\begin{center}
\includegraphics[trim=0cm 0.2cm 0cm 0cm, clip, width=6cm]{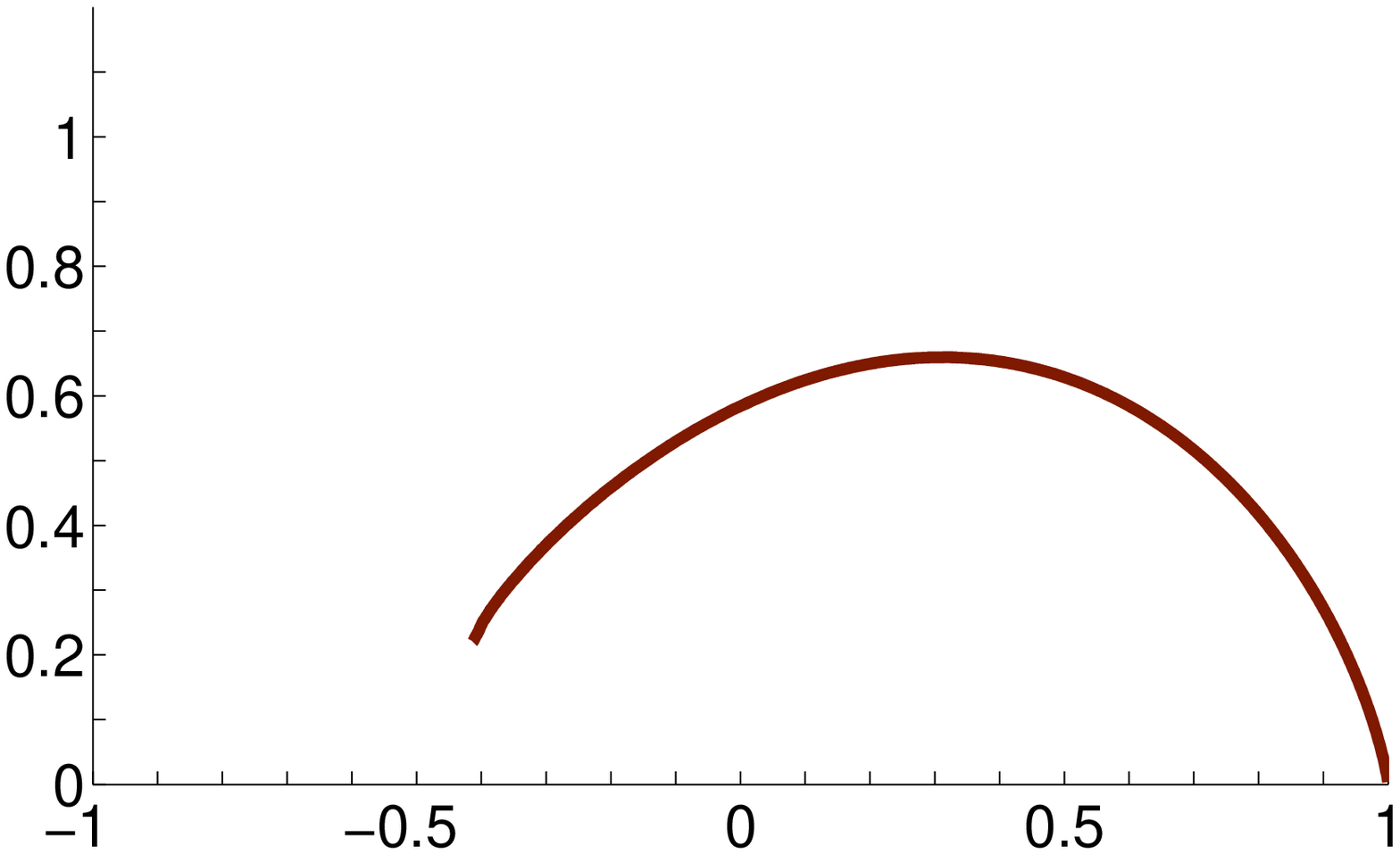}
\hspace{0.2cm}
\includegraphics[trim=0cm 0.2cm 0cm 0cm, clip, width=6cm]{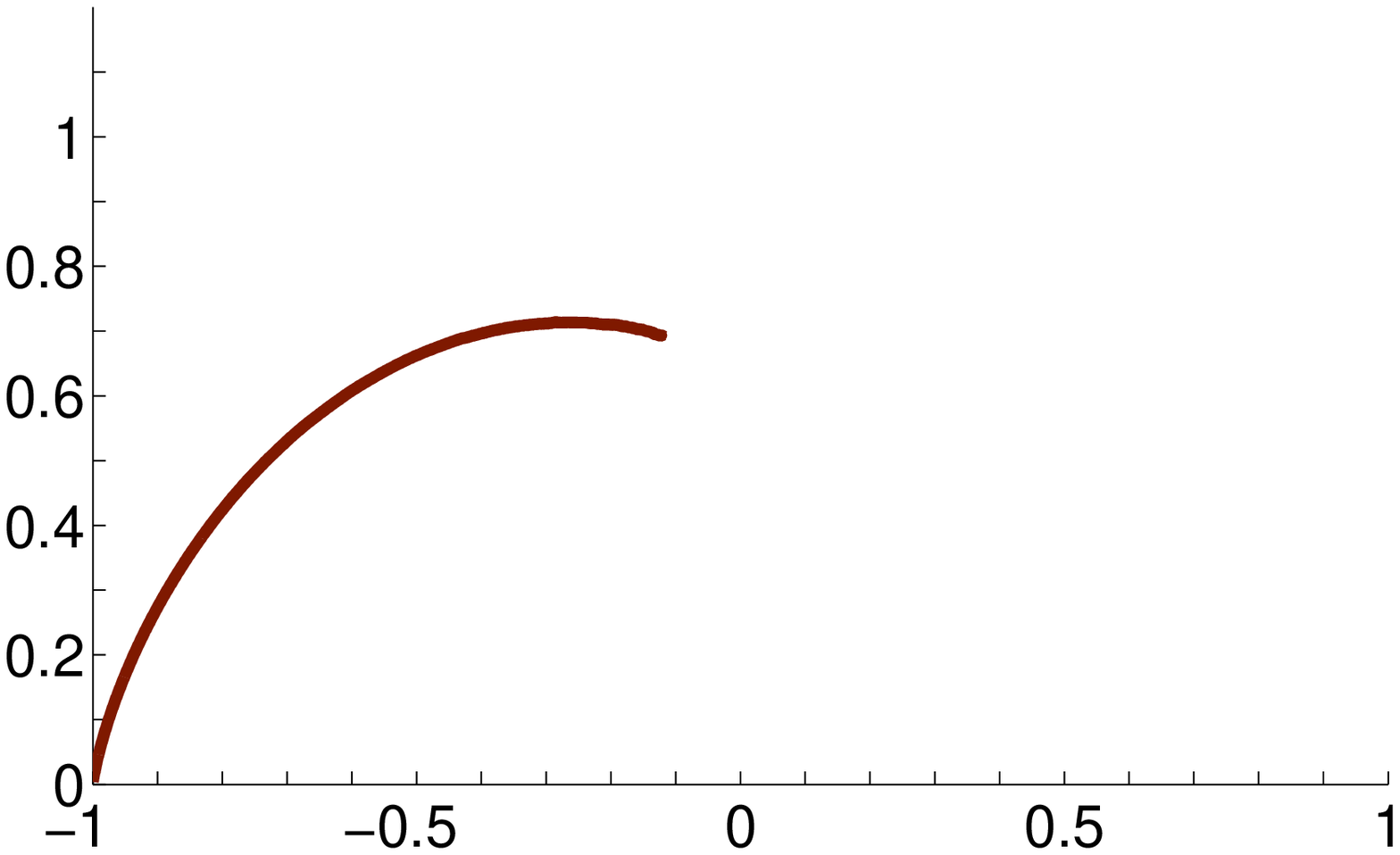}
\end{center}
\caption{Empirical Lyapunov profiles for rules {\it 30\/} and {\it 106\/} at $t=10^5$.} \label{figure30-106}
\end{figure}

\begin{table}[!ht]
\caption{The $23$ expansive rules.}\label{ECA-expansive}
\vspace{0.0cm}
\centering
\begin{tabular}{| c | c | c|c|c| c| c|}
			\hline
 \begin{tabular}{@{}c@{}}Rule\end{tabular}
& \begin{tabular}{@{}c@{}}class 
\end{tabular}
& \begin{tabular}{@{}c@{}}proof
\end{tabular}
& \begin{tabular}{@{}c@{}}MLE dir. 
\end{tabular}
& \begin{tabular}{@{}c@{}}MLE
\end{tabular}
&\begin{tabular}{@{}c@{}}$W$
\end{tabular}
&\begin{tabular}{@{}c@{}}$\rho$
\end{tabular}
  \\
			\hline
			\hline
{\it 6\/} & E & --- & $-0.29$ & $0.55$  &$[-1,0.36]$ & $0.84$\\ \hline
{\it 18\/} & E & --- & $0$ & $0.69$   &$[-1,1]$ & $0.5$\\ \hline
{\it 22\/} & E & --- & $0$ & $0.87$   &$[-0.74,0.74]$ & $0.86$\\ \hline
{\it 25\/} & E & --- & $-0.17$ & $0.52$   &$[-0.83,0.5]$ & $1$\\ \hline
{\it 26\/} & E & --- & $-0.32$ & $0.41$   &$[-1,0.23]$ & $1$\\ \hline
{\it 30\/} & E & $\cM=\{1,3\}$, $t_\cM=3$ & $0.31$ & $0.66$   &$[-0.41,1]$ & $1$\\ \hline
{\it 38\/} & E & Prop.~\ref{rule-38} & $-0.41$ & $0.54$   &$[-1,0.1]$ & $1$\\ \hline
{\it 41\/} & E & --- & $0.02$ & $0.86$   &$[-0.75,1]$ & $0.94$\\ \hline
{\it 45\/} & E & $\cM=\{0,2\}$, $t_\cM=2$ & $0.22$ & $0.72$   &$[-0.48,1]$ & $1$\\ \hline
{\it 54\/} & E & $\cM=\{0,\pm 1\}$, $t_\cM=3$ & $0$ & $0.74$   &$[-0.85,0.85]$ & $1$\\ \hline
{\it 57\/} & E & $\cM=\{0,\pm 1\}$, $t_\cM=3$ & $0$ & $0.69$   &$[-1,1]$ & $1$\\ \hline
{\it 60\/} & E & additive & $1/2$ & $0.69$ &$[0,1]$ & $1$  \\ \hline
{\it 62\/} & E & --- & $0$ & $0.44$ &$[0,0.53]$& $1$\\ \hline
{\it 90\/} & E & additive & $0$ & $0.69$   &$[-1,1]$ & $1$\\ \hline
{\it 105\/} & E & additive with toggle & $0$ & $1.1$   &$[-1,1]$ & $1$\\ \hline
{\it 106\/} & E & --- & $-0.26$ & $0.71$   &$[-1, -0.11]$ & $1$\\ \hline
{\it 110\/} & E & --- & $-0.25$ & $0.66$   &$[-0.88, 0.67]$ & $1$\\ \hline
{\it 122\/} & E & --- & $0$ & $0.65$   &$[-1, 1]$ & $1$\\ \hline
{\it 126\/} & E & --- & $0$ & $0.71$   &$[-1,1]$ & $1$\\ \hline
{\it 134\/} & E & --- & $-0.21$ & $0.51$   &$[-1,0.49]$ & $0.81$\\ \hline
{\it 146\/} & E & --- & $0$ & $0.69$   &$[-1,1]$ & $0.5$\\ \hline
{\it 150\/} & E & additive & $0$ & $1.1$   &$[-1,1]$ & $1$\\ \hline
{\it 154\/} & E & --- & $-0.42$ & $0.48$   &$[-1,0.11]$ & $1$\\ \hline
\end{tabular}
\end{table}

We should also mention that 
it is easy to check that {\it Rule 154\/} and 
{\it Rule 106\/} are right permutative \cite{GG5} and 
thus 
at least not collapsing, with $\cM=\{-1\}$, $t_\cM=1$. 
In fact, due to the $\limsup$ in the definition of $L$ (\ref{lyapunov-profile-def}), there are seven other 
rules that are provably not collapsing as a defect at the origin must generate at least one successor,
although its location varies with $\xi_0$. These rules are {\it 37\/}, 
{\it 41\/}, {\it 56\/},
{\it 62\/}, {\it 110\/}, {\it 134\/}, {\it 146\/}, and {\it 184\/}. 

Another remark is that the three quasi-additive rules studied by E.~Jen \cite{Jen}, 
{\it 18\/}, {\it 146\/} and {\it 126\/}, all feature annihilating dislocations 
that make the CA approach {\it Rule 90\/}. This apparently causes 
the Lyapunov profile to be 
indistinguishable from the one for {\it Rule 90\/} for the first 
two rules (thus the MLE is $\log 2$), but not for {\it Rule 126\/} whose defect 
dynamics differs from that for {\it Rule 90\/} even in the invariant state. 

\begin{prop}\label{rule-38}
Assume the CA is {\it Rule 38\/}. If $\De_0\subset[-r,r]$,  then 
$\damage_t\subset[-t-r, -t+r+3]$. 
On the other hand, the defect accumulation dynamics is expansive; in fact, 
$L_\infty$ is strictly positive on $(-1,\alpha_r]$, where $\alpha_r=308/2977$ and $W=[-1,\alpha_r]$. 
\end{prop}

\begin{proof}
First observe (by a simple verification) that there 
is no $0101$ in $\ca_t$, for $t\ge 1$, and then no $111$ for $t\ge 2$. We will 
assume $t\ge 2$ from now on. Any $0100$ (resp.~$0110$)
starting at $x$ 
at time $t\ge 2$ generates $0110$ (resp.~$0100$) starting at $x-1$ at time $t+1$. 
Thus the entire configuration
$\ca_{t+2}$ is obtained 
by shifting $\ca_t$ to the left by $2$. This proves the first claim.

As the rule has no stable update, a full interval of defects can only be eroded
at speed one from the edges. Assume (without loss of generality) 
that the left edge of an interval of defects of length at least 
$3$ is on an infinite diagonal 
(of slope $1$) of $1$s. Then the boundary arrangement (with a defect site $(x,t)$ underlined) 
is one of these four: $00\underline10$, $10\underline10$, $00\underline11$, $10\underline11$. In all 
cases the defect at $(x,t)$ branches into two defects, one at $(x,t+1)$ and one at $(x-1,t+1)$. 
Thus the left edge of the defect interval advances at light speed. 

There are six possible arrangements at the right edge at $(x,t)$ (underlined); we write 
$\downarrow$ when the edge stays at $x$ at time $t+1$ and $\searrow$ when it moves 
to $x+1$ (that is, when the defect branches into two):
$$
 0\underline 000\, \downarrow
\quad 0\underline 001\, \downarrow
\quad 0\underline 010\, \searrow
\quad  0\underline 011\, \downarrow
\quad  0\underline 100\, \downarrow
\quad  1\underline 100\,\downarrow
$$
Thus the right edge never retreats and advances when in contact with the diagonal in 
one of the two ``phases.''

To be more precise, we first provide a convenient Markovian description 
of $\ca_2$. Consider the set $H$ of $24$ pairs $(s,a)$, where 
$s$ is a binary strings of length $4$ that does not 
contain $111$ or $0101$, and $a$ is either $0$ or $1$. 
Call $x\in \bZ$ in a state $(s,a)$ if the string $s$ 
{\it ends\/} at $x$ and $x\in a+2\bZ$. As sites at distance $5$ or more have 
independent $\ca_2$-state, this is a Markov chain. Define the following subsets of $H$, 
\begin{equation}\label{H1-H2}
\begin{aligned}
&H_1=\{(0011,1), (1011,1), (0010,0), (1010,0)\}\\
&H_2=\{(0011,0), (1011,0), (0010,1), (1010,1)\}
\end{aligned}
\end{equation}
Start in (say) the state $(0011,0)$ at $x=0$, and consider 
the successive states of the chain given by positive integers. 
Define 
$\tau_0=0$ and then let $\tau_k$, $k=1,2,\ldots$ be the number of steps after 
$\tau_{k-1}$ needed to 
enter $H_1$ (even $k$) or $H_2$ (odd $k$). For example, if $\ca_2$ on $\bZ_+$ 
happens to be $101100110010\ldots$, then $\tau_1=3$ and $\tau_2=8$. 

By the preceding part of the proof, the right edge of $\de_t$ is at $n$ at 
time $\sum_{i=1}^n\tau_i-n$. By symmetry, almost surely,
$$
\lim_{n\to\infty}\frac 1n\sum_{i=1}^n\tau_i= 
\frac{\sum_{h\in H_1}\pi(h)\E T(h,H_2)}{\sum_{h\in H_1}\pi(h)}.
$$
Here, $\pi$ is the invariant measure and $\E T(h, H_2)$ is the expected time to reach $H_2$ from 
$h$, both readily computable by a matrix computation 
to get the limit $3285/308$. The right edge of $W$ 
then is 
$$
\lim_{n\to\infty}\frac{n}{\sum_{i=1}^n \tau_i-n}=\alpha_r.
$$

Finally, we prove the claim that $L_\infty>0$ on $(-1,\alpha_r]$. For this, it 
is sufficient to show that 
\begin{equation}\label{rule38-edgeheight}
L_\infty(\alpha_r)=\alpha_r\log 2
\end{equation}
as then, by just considering defects that accumulate on the path that first 
moves on the right edge and then on a leftward diagonal of $1$s, 
$$
L_\infty(\alpha)\ge \frac{\alpha_r\log 2}{\alpha_r+1}(\alpha+1)
$$
on $[-1,\alpha_c]$. To prove (\ref{rule38-edgeheight}), observe first 
that the only {\it Rule 38\/} update that is sensitive to a change of
both left and center input is $010\mapsto 1$. The number of paths
at the right edge thus goes up by a factor of $2$ precisely when the rightmost defect
is on the middle $1$ of $010$. The number of times this happens is
exactly the number of states in $H_1$ (resp.~in $H_2$) in $[\tau_k+1, \tau_{k+1}]$ 
for odd $k$ (resp.~even $k$). The expected number of such states is $1$, 
by elementary Markov chain theory, and so the number of paths at the right edge at time $\sum_{i=1}^n\tau_i-n$ is $2^{N_n}$ where $N_n/n\to 1$ a.s.~as $n\to\infty$. The 
claimed equality (\ref{rule38-edgeheight}) follows.
\end{proof}

While Proposition~\ref{rule-38} determines its support, a full 
characterization of the Lyapunov profile in cases
such as {\it Rule 38\/} is closely related to quenched large deviations for random walks 
in a random environment (see e.g.~\cite{Yil}).  
A computationally viable variational technique is beyond current methods
(which in particular require nondegeneracy conditions that {\it Rule 38\/} walks 
fail to satisfy)
and seems a very interesting open problem. 
 
\subsection{Classification of the remaining elementary CA}\label{ECA-subsection-remaining}

The remaining $11$ rules are gathered in Table~\ref{ECA-mystery}, with 
conjectured class and other empirical information. 

\begin{table}
\caption{The $11$ remaining rules.}\label{ECA-mystery}
\vspace{0.0cm}
\centering
\begin{tabular}{| c | c | c|c|c|}
			\hline
 \begin{tabular}{@{}c@{}}Rule\end{tabular}
& \begin{tabular}{@{}c@{}}class 
\end{tabular}
& \begin{tabular}{@{}c@{}}notes
\end{tabular}
& \begin{tabular}{@{}c@{}}MLE dir. 
\end{tabular}
& \begin{tabular}{@{}c@{}}MLE
\end{tabular}
  \\
			\hline
			\hline
{\it 9\/} & M & long transient period & $1$ & $0$\\ \hline
{\it 11\/} & M & medium transient period & $1$ & $0$\\ \hline
{\it 14\/} & C & gliders erode defects& --- &  ---\\ \hline
{\it 35\/} & M & medium transient period & $1/2$ & $0$\\ \hline
{\it 37\/} & M & medium transient period & $0$ & $0.35$\\ \hline
{\it 43\/} & C & gliders erode defects  & --- & ---\\ \hline
{\it 56\/} & M & medium transient period & $1$ & $0$\\ \hline
{\it 58\/} & M & long transient period & $-1$ & $0$\\ \hline
{\it 74\/} & M & medium transient period & $-1$ & $0$\\ \hline
{\it 142\/} & C & gliders erode defects   & --- & ---\\ \hline
{\it 184\/} &M  & defects percolate when gliders collide& $0$ & $0$ \\ \hline
\end{tabular}
\end{table}

All these 
dynamics feature a relatively simple invariant state, an {\it ether\/},
which supports a variety of annihilating gliders. A detailed 
quantitative analysis of the glider dynamics necessary for 
the proof may be possible in some cases (for some results in 
this direction, see \cite{BF} for {\it Rule 184\/}, 
and density computations of the three collapsing rules in Section~\ref{ECA-subsection-dendep}), but is
beyond the scope of this paper. However, we observe that 
the glider configuration for seven of these CA appears to 
stabilize at an exponential rate (hence the reference 
to the ``transient period''), while {\it Rule 184\/} 
and the three collapsing rules feature recurrent glider collisions 
that drive their density to zero much more slowly, at the rate $t^{-1/2}$ (by the 
argument in \cite{DS} for a similar dynamics). 

\subsection{Dependence of defect accumulation on initial density for elementary CA}\label{ECA-subsection-dendep}

We now turn our attention to how the defect accumulation depends on the density 
of $1$s in the initial CA configuration. We will assume that $\ca_0$ is the product measure with 
constant $p=\P(\ca_0(x)=1)\in (0,1)$, and mostly study how MLE varies with $p$.
Our next result greatly reduces the rules we need to 
consider.

\begin{theorem}\label{M-density-ind}
All rules in Table~\ref{ECA-marginal} are marginal for all $p\in(0,1)$, and their 
Lyapunov profile $L_\infty$ does not depend on $p$.
\end{theorem}

\begin{proof}
The proof of the equivalence property  
for {\it Rule 152\/} in Proposition~\ref{rule152-marginal} is easily adapted.
The remainder follows from the fact than any finite configuration occurs infinitely 
often in any nontrivial uniform product measure.  
\end{proof}

With one exception, we expect that Theorem~\ref{M-density-ind} holds 
also for the marginal rules in Table~\ref{ECA-mystery}. The special case is {\it Rule 184\/}, 
which does not have a transient defect dynamics when $p=1/2$ \cite{BF}, but the transience does 
hold for other $p$. Furthermore, the defect dynamics is marginal for all $p$,
and the MLE does not depend on $p$, but its direction does: it is $1$ for $p<1/2$, 
$0$ for $p=1/2$ and $-1$ for $p>1/2$. 

The three collapsing rules in Table~\ref{ECA-mystery} are at first quite mysterious 
and computer simulations do not offer conclusive evidence even 
on the classification of the defect dynamics near $p=1/2$. Therefore, we need to 
to take a closer look at gliders for these three CA.  As the analysis for {\it Rule 142\/}
is almost exactly the same as for {\it Rule 14\/}, we will only discuss the latter  
and {\it Rule 43\/} in detail. For both of these, 
the ether is the configuration $(0011)^\infty$, which gets 
translated to the left by $1$ every time step. There are two kinds of gliders, leftward-
and rightward-moving ones, at sites with local configurations as 
given in Table~\ref{glider-collapse} (with a glider site and direction indicated by the arrow). 
As we see from this table, one or the other type of gliders ``wins'' when $p\ne 1/2$. 
However, for the advantage to be detectable empirically, the array size would have
to be on the order of at least $1/(2p-1)^4$, 
too impractical when $p=0.51$, say. 
From simulations we conclude that glider imbalance leads to marginal 
dynamics with the MLE equal to $0$ in both cases and the 
MLE direction either $-1$ (for {\it Rule 14\/}) or $1$ (for {\it Rule 43\/}). 
When $p=1/2$, the glider dynamics has the same behavior as in {\it Rule 184\/}
(at the same $p$), 
but by contrast the defects are not able to percolate through all collisions, 
which causes the collapse in the case of a uniform product initialization. These three rules 
thus do exhibit dramatic variation with $p$, albeit of 
a rather degenerate kind, as $\lambda_\infty=0$ except at a single exceptional density
$p=1/2$ where $\lambda_\infty=-\infty$ .

\begin{table}
\begin{quote}
\caption{Gliders in rules {\it 14\/} and {\it 43\/} and information about their 
initial probabilities.}\label{glider-collapse}
\end{quote}
\vspace{-0.5cm}
\centering
\begin{tabular}{| c | c | c|c|c|}
			\hline
 \begin{tabular}{@{}c@{}}Rule\end{tabular}
& \begin{tabular}{@{}c@{}}leftward glider sites
\end{tabular}
& \begin{tabular}{@{}c@{}} rightward glider sites
\end{tabular}
& \begin{tabular}{@{}c@{}}$\P(\leftarrow)-\P(\rightarrow)$
\end{tabular}
& \begin{tabular}{@{}c@{}}$\P(\leftarrow)$ when $p=1/2$
\end{tabular}
  \\
			\hline
			\hline
\begin{tabular}{@{}c@{}}\vspace{0.3cm}{\it 14\/} 
\end{tabular}
& $0\!\overleftarrow0\!0$, $111\!\overleftarrow1$, $111\!\overleftarrow0\!0$ & 
$0\!\overrightarrow1\!0$, $01\!\overrightarrow0\!1$, $011\!\overrightarrow0\!1$ 
& $(2p-1)^2$ &  $7/32$ \\ \hline
\begin{tabular}{@{}c@{}}\vspace{0.3cm}{\it 43\/} 
\end{tabular}
& $0\!\overleftarrow1\!0$, $1\!\overleftarrow0\!1$ & $0\!\overrightarrow0\!0$, $1\!\overrightarrow1\!1$ & $-(2p-1)^2$ & $1/4$\\ \hline
\end{tabular}
\end{table}

 It remains to address the rules in
Table~\ref{ECA-expansive}.  The $14$
rules that are not stripes CA 
are attracted to the same invariant state independent of $p$; that state is chaotic 
except for {\it Rule 110\/} that possibly slowly converges \cite{LN} 
to the periodic state with the MLE 
around $0.65$
discussed in Section~\ref{periodic-profiles-rule110}. As a result, the Lyapunov profiles, 
and therefore the MLE, 
for these $14$ rules exhibit no significant variation with $p$. 
Next, we present evidence that the 
nine stripes rules, while they remain expansive, 
do have detectable dependence 
of the MLE $\lambda_\infty$ on $p$.
 
\begin{table}[ht!]
\caption{Dependence of the MLE on $p\in(0,1)$ for expansive stripes CA.} 
\vspace{0.0cm}
\label{ECA-stripes}
\centering
\begin{tabular}{| c | c | c|}
			\hline
 \begin{tabular}{@{}c@{}}Rule\end{tabular}
& \begin{tabular}{@{}c@{}}min.~MLE 
\end{tabular}
& \begin{tabular}{@{}c@{}}max.~MLE
\end{tabular}
  \\
			\hline
			\hline
{\it 6\/} & $0.54$ at $p=0.4$ & $0.69$ at $p=1-$\\ \hline
{\it 25\/} & $0.35$ at $p=0+, 1-$ & $0.52$ for $p\in(0.4.0.6)$ \\ \hline
{\it 26\/} & $0.41$ at $0.37$ & $0.59$ at $p=0+, 1-$\\ \hline
{\it 38\/} & $0.54$ at $0.5$ &$0.69$ at $p=0+,1-$ \\ \hline
{\it 41\/} & $0.86$ for $p\in (0.15,0.85)$ &$0.89$ at $p=0+,1-$\\ \hline
{\it 57\/} & $0.693$ for $p=0.5$ & $0.706$ at $p=0.25, 0.75$\\ \hline
{\it 62\/} & $0.44$ for $p\in (0.08,0.92)$ & $0.47$ at $p=0+,1-$\\ \hline
{\it 134\/} & $0.45$ at $p=1-$  & $0.68$ at $p=0+$ \\ \hline
{\it 154\/} & $0.43$ at $p=0.22$& $0.69$ at $p=1-$\\ \hline
\end{tabular}

\end{table}
 
The nature of this dependence differs significantly among the nine expansive stripes rules
and is summarized in Table~\ref{ECA-stripes}. Most approximations are based on computations up to 
time $t=2\cdot 10^4$ for $99$ equally spaced densities in $(0,1)$. We use $t=10^5$ 
for the more subtle rules {\it 57\/} and  {\it 62\/}, which are discussed in 
greater detail below. Except for these two rules, we observe 
a greater MLE variability than reported in 
\cite{BRR}, which restricts the range of $p$, and, as reviewed in the Introduction, 
has a related but different definition of MLE $\lambda_\infty$. However, in some cases 
$\lambda_\infty$ is indistinguishable from a constant on an interval, as
indicated in Table~\ref{ECA-stripes}. 
We illustrate the density dependence by giving more details for {\it Rule 134\/} (see Fig.~\ref{134-dendep}): this rule generates the profile that spreads out  with increasing $p$, as its peak decreases and its support widens.

\begin{figure}[!ht]
\begin{center}
\includegraphics[trim=0cm 0cm 0cm 0cm, clip, width=6cm]{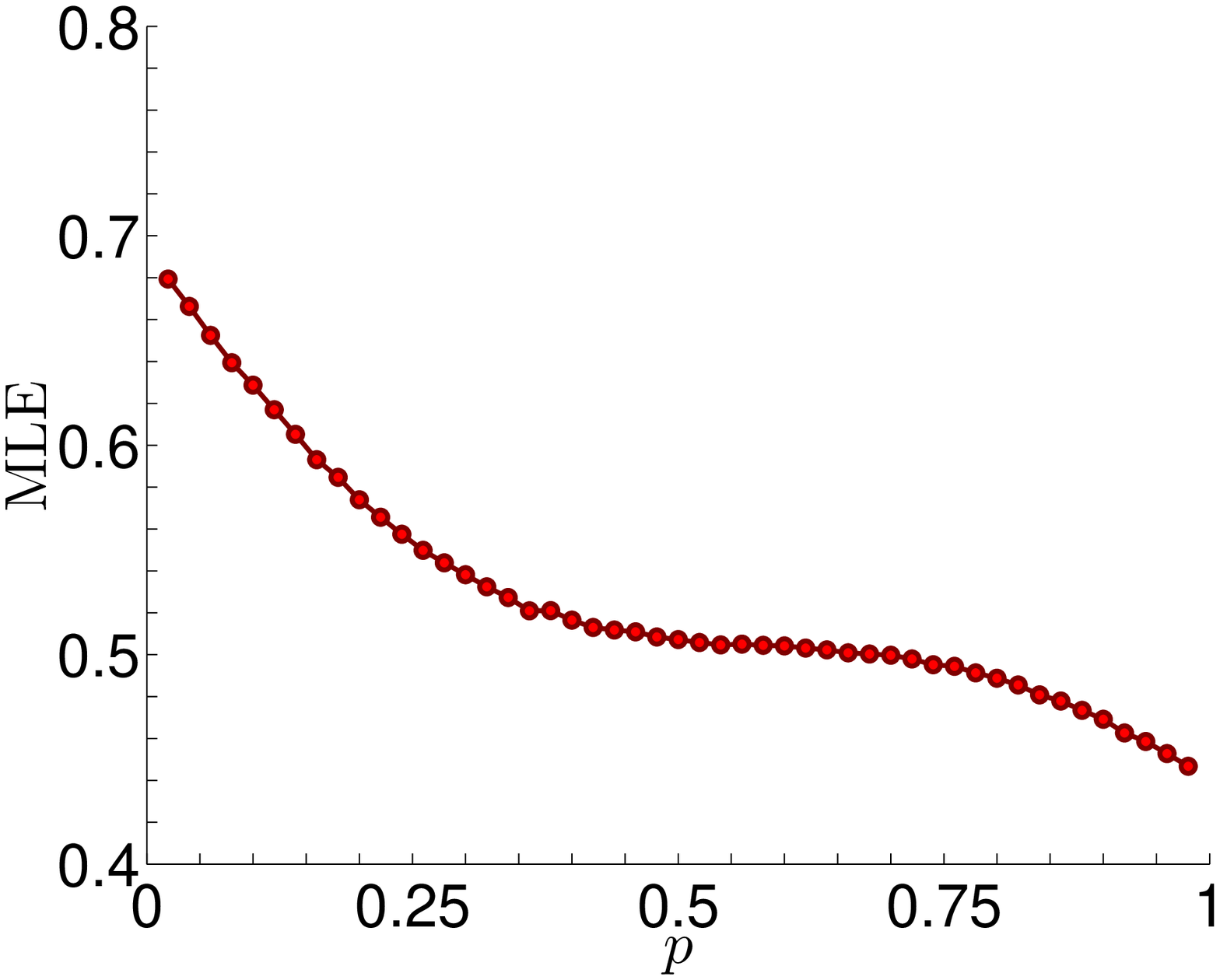}
\hspace{0.2cm}
\includegraphics[trim=0cm 0cm 0cm 0cm, clip, width=6cm]{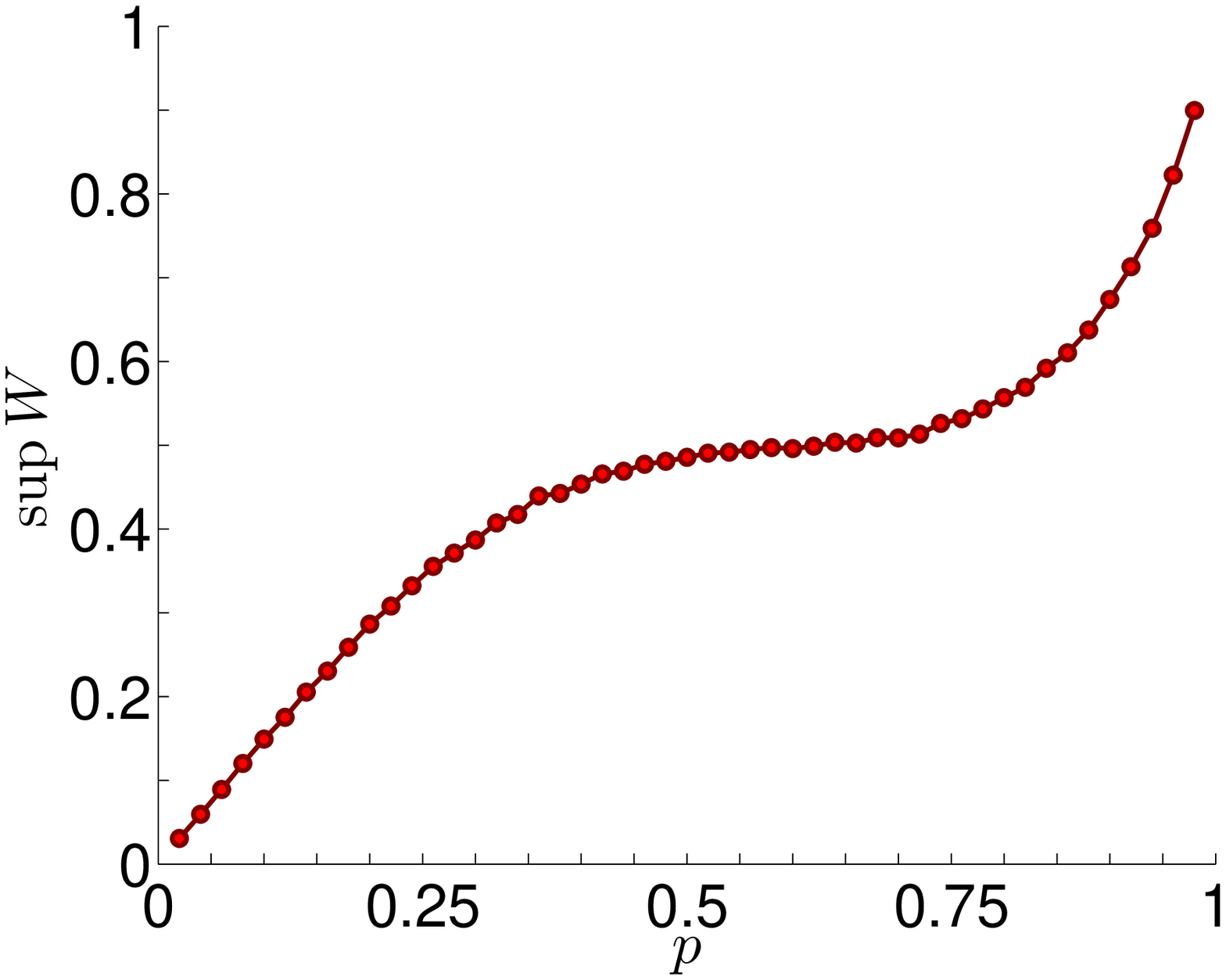}\\
\end{center}
\caption{Dependence on density $p$ for {\it Rule 134\/}: the
MLE (left) and the right edge of $W$ are graphed vs.~$p$. (The left edge of $W$ 
stays at $-1$.)} \label{134-dendep}
\end{figure}

We conclude this section with an empirical analysis of rules {\it 57\/} and {\it 62\/}.  
Like for the other seven stripes rules, 
it is (empirically) clear that for these two $W_\damage$ is (a.s.) at most a singleton for all $p$. 
Unlike the others, however, they at first appear to exhibit no density dependence of MLE on $p$. 
This necessitates a closer inspection, and we begin 
with {\it Rule 62\/}. 

As is common for stripes CA, {\it Rule 62\/} dynamics undergoes a transient phase until (in this case 
vertical) stripes dominate. This phase is quite long-lasting, 
and is characterized by the annihilation of diagonal 
gliders, which are temporarily able to block the expansion of defects.
See Fig.~\ref{62-profile-figure} for a sample evolution and 
the resulting Lyapunov profile. 

\begin{figure}[!ht]
\begin{center}
\includegraphics[trim=0cm 0cm 0cm 0cm, clip, width=9cm]{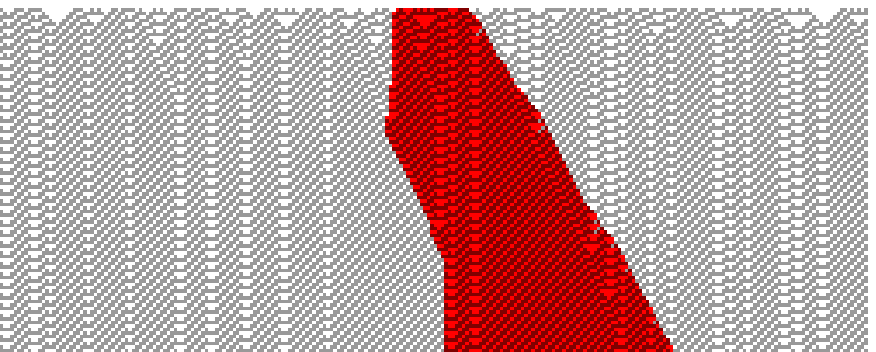}
\hspace{0.2cm}
\includegraphics[trim=0cm 0.2cm 0cm 0cm, clip, width=6cm]{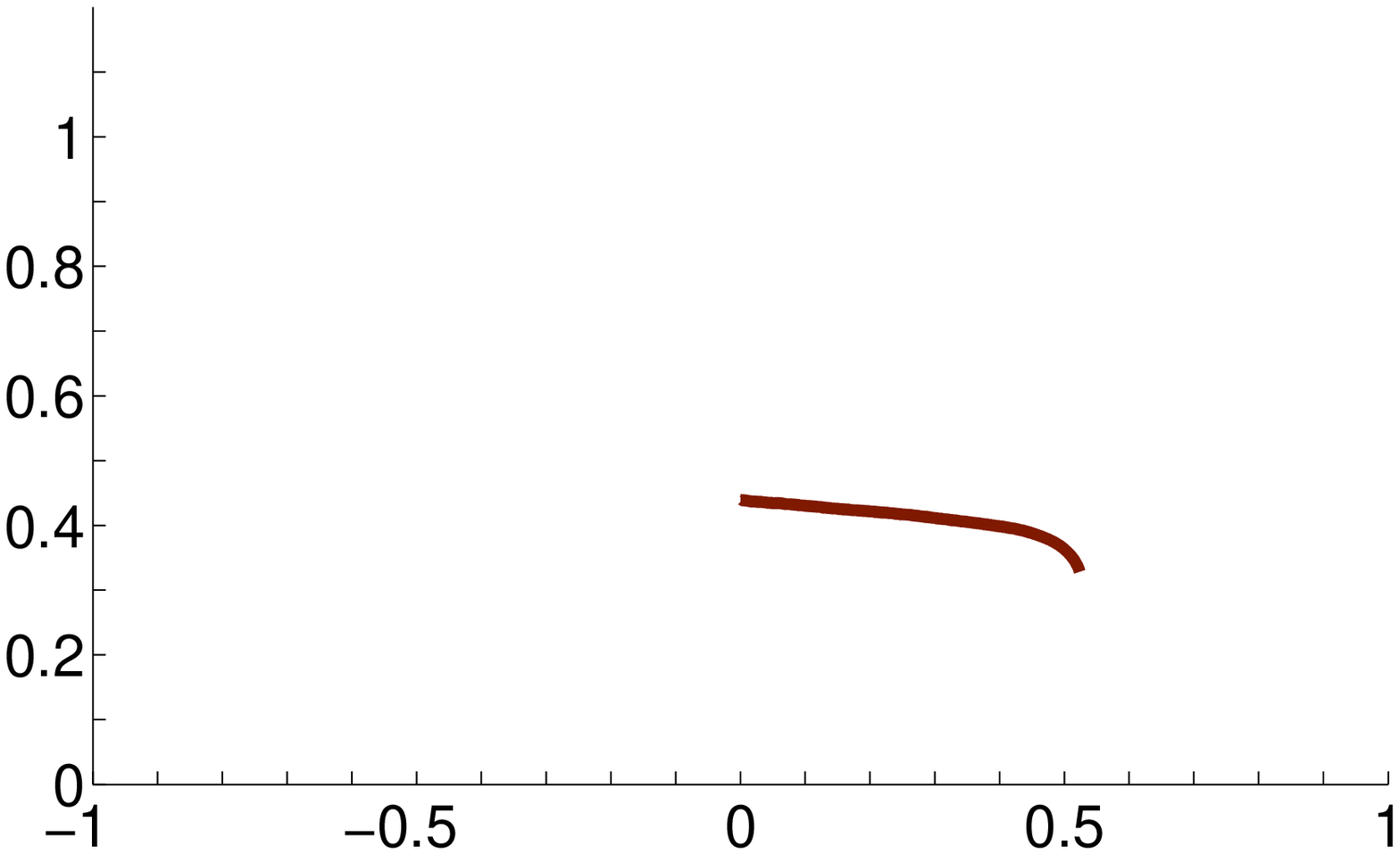}\\
\end{center}
\caption{Evolution of defect percolation CA up to time $100$ and the empirical Lyapunov profile
at time $10^5$ 
for {\it Rule 62\/}.} \label{62-profile-figure}
\end{figure}

It turns out that the only detectable variation of the MLE and its direction
occurs near $p=0$ and $p=1$.  
In fact, there seems to be an intriguing phase transition near $p=0.08$ that 
is marked by the sharp turn of MLE curve and the sudden passage of the MLE direction to $0$. 
See Fig.~\ref{62-dendep}.

\begin{figure}[!ht]
\begin{center}
\includegraphics[trim=0cm 0cm 0cm 0cm, clip, width=6cm]{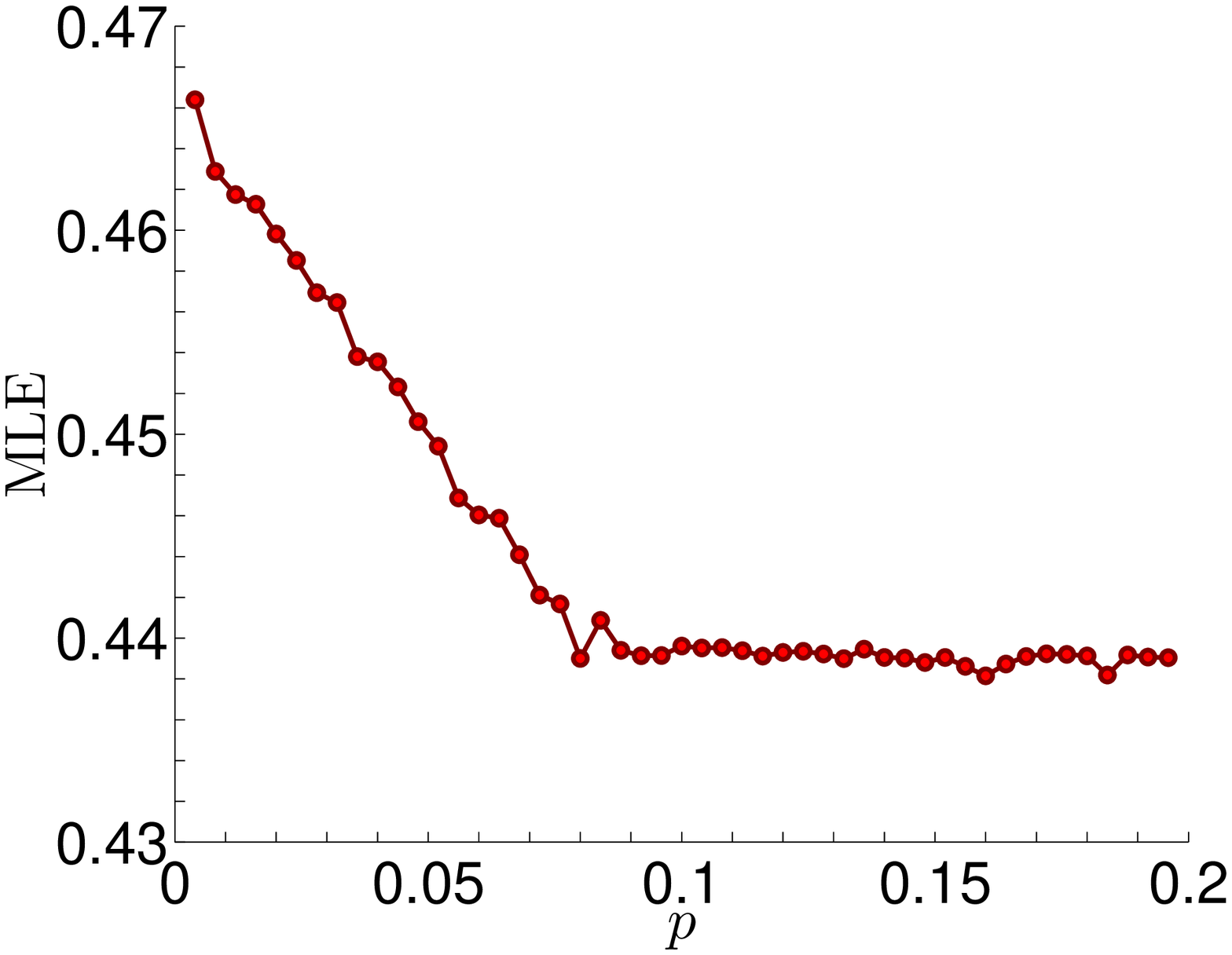}
\hspace{0.2cm}
\includegraphics[trim=0cm 0cm 0cm 0cm, clip, width=6cm]{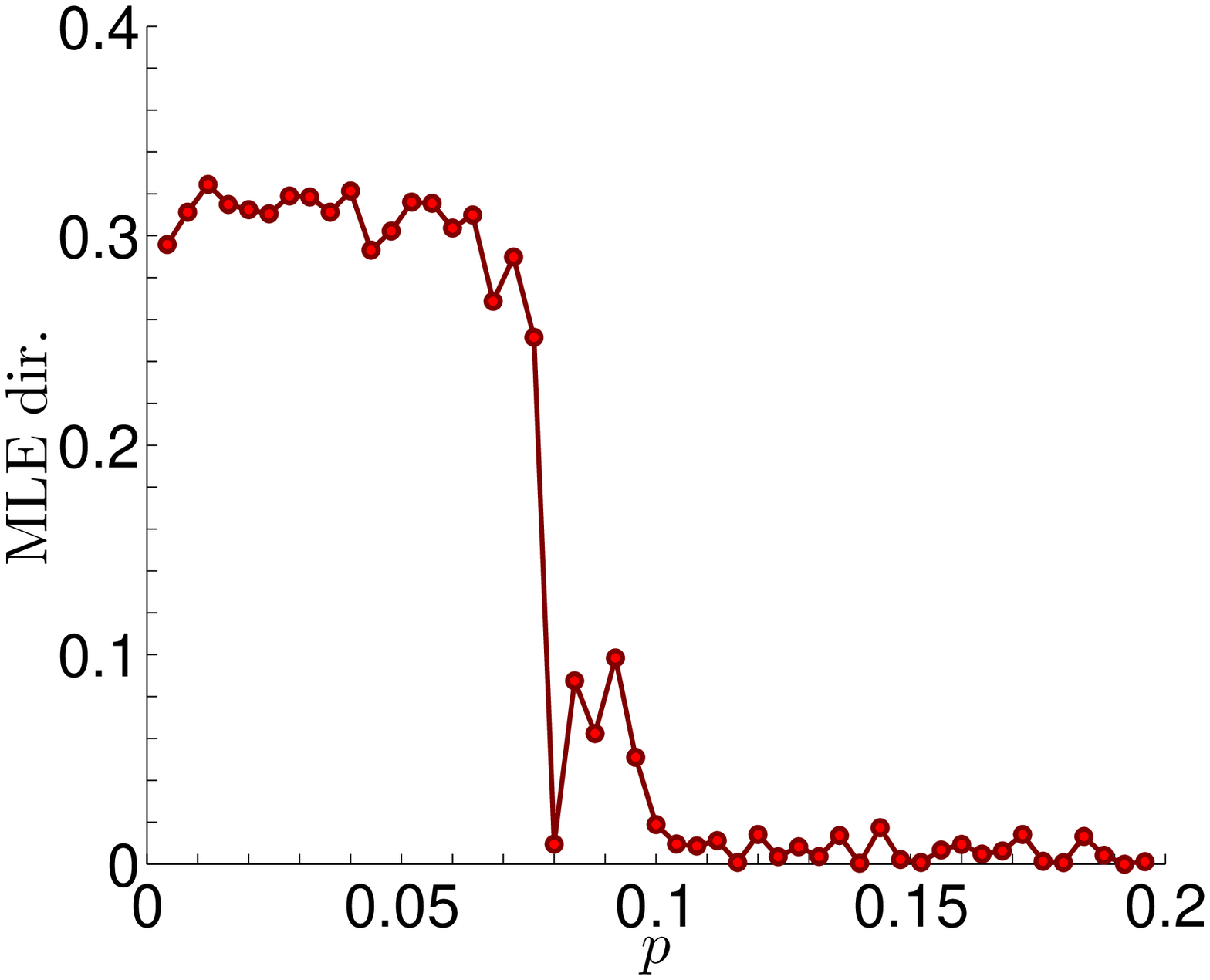}\\
\end{center}
\caption{Dependence on density $p$ near $p=0$ for {\it Rule 62\/}: the
MLE (left) and its direction are graphed vs.~$p$. These graphs are based on comptations up to $t=10^5$.} \label{62-dendep}
\end{figure}

Finally, {\it Rule 57\/} is another case with  
pairwise annihilating gliders, which are rightward-moving
pairs of $0$s and leftward-moving
pairs of $1$s on a checkerboard ether. This rule is invariant under a symmetry transformation:
if one switches the roles of two states, and then 
applies the left-right reflection, one obtains the same rule. As a consequence, 
temporarily using the superscript to indicate the dependence on $p$, 
$L^p_\infty(\alpha)=L^{1-p}_\infty(1-\alpha)$ and $\lambda^p_\infty=\lambda^{1-p}_\infty$. It is 
therefore enough to consider $p\in (0,1/2)$. On this interval, {\it Rule 57\/}
is a stripes rule, with the rightward gliders dominating. 
At $p=1/2$, this rule cannot be striped, as
$\xi_t$ equals its 
reflection in distribution and thus neither 
of the two gliders can win. See Fig.~\ref{57-dendep} for the 
empirical results.

\begin{figure}[!ht]
\begin{center}
\includegraphics[trim=0cm 0cm 0cm 0cm, clip, width=6cm]{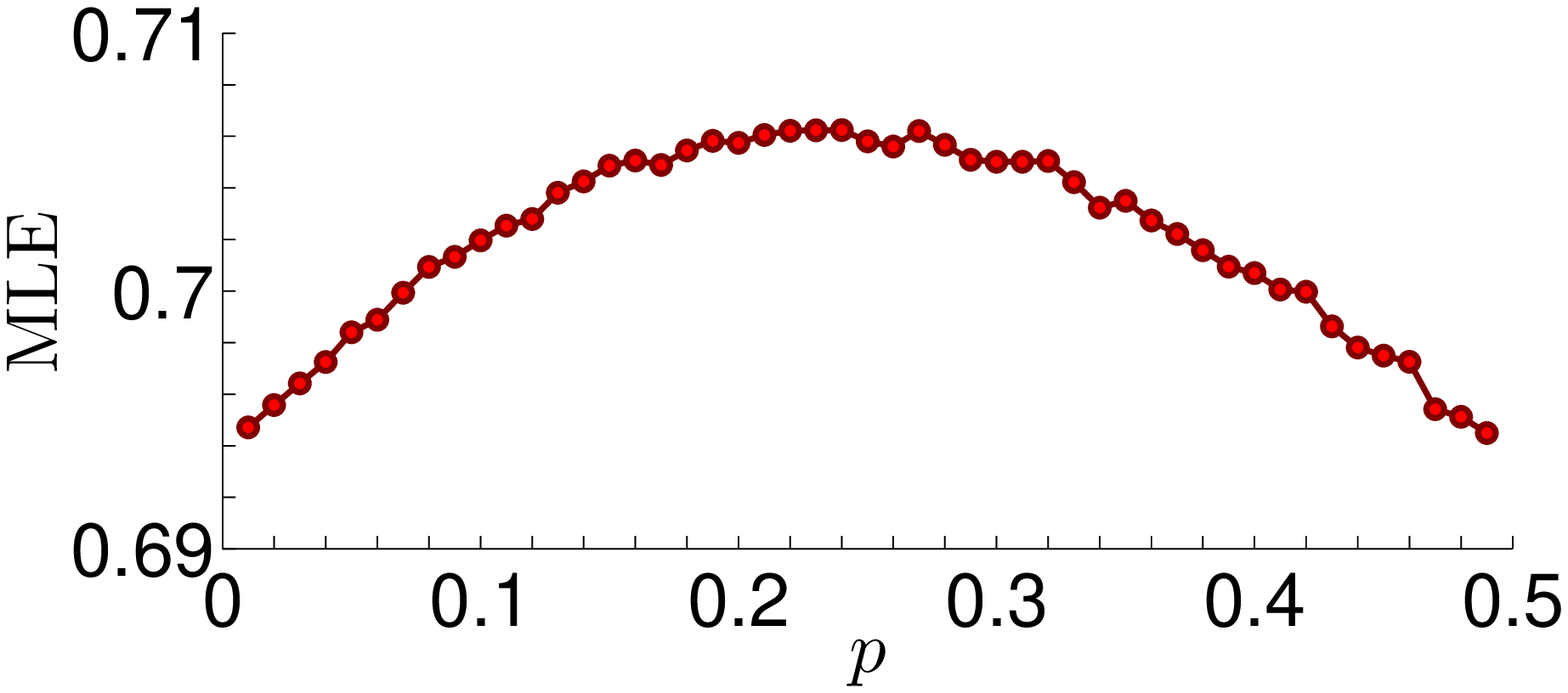}
\hspace{0.2cm}
\includegraphics[trim=0cm 0cm 0cm 0cm, clip, width=6cm]{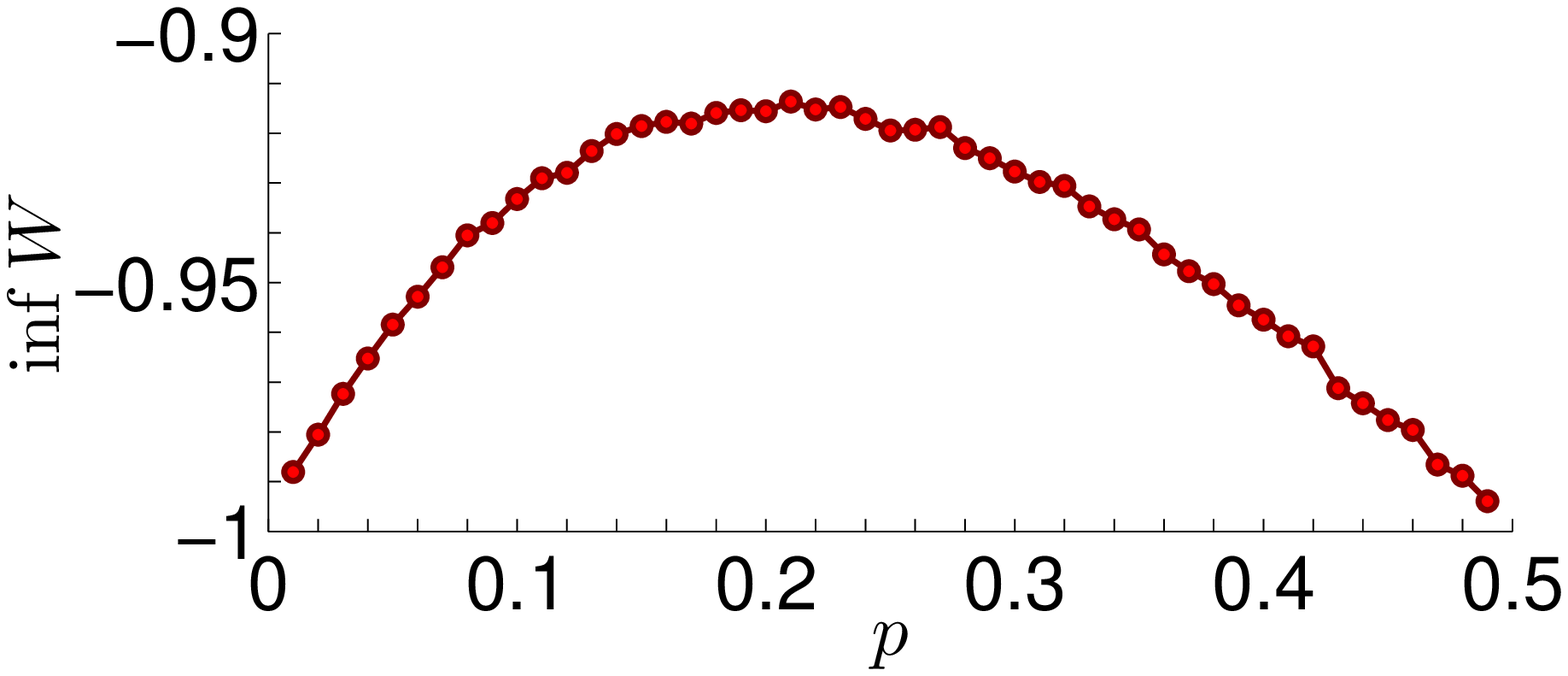}\\
\end{center}
\caption{Dependence on density $p\in (0,1/2)$ for {\it Rule 57\/}:
the MLE (left) and the left edge of $W$ are graphed vs.~$p$ (the right edge stays at $1$).
These are computed at $t=10^5$.} \label{57-dendep}
\end{figure}
  
\section{Two-dimensional cellular automata} \label{2d-section}

While the theoretical set-up is similar, a two-dimensional geometry is
much less restrictive than a one-dimensional one, making rigorous theory more 
demanding and in need of further development. We restrict our attention to
totalistic rules with a von Neumann 
or Moore neighborhood. The one simple rigorous result we provide next
identifies $8$ of the $2^6=64$ of the former rules, 
and $32$ of the  $2^{10}=1024$ of the latter rules, 
as collapsing. The nomenclature we use is similar to the one in \cite{Vic1}:
the rule is identified by the neighborhood, and the name {\it Tot\/} followed by the list 
of {\it occupation numbers\/}, that is, the neighborhood counts that update to $1$. 
For example, Moore neighborhood {\it Tot 1\/} updates $x$ to $1$ precisely when there is
a single $1$ among the $9$ neighbors of $x$. 

\begin{prop} \label{boot-collapse}
Assume that $\ca_0$ is a product measure with density $p\in (0,1)$. 
For Moore neighborhood, any totalistic rule for which $4,\ldots 9$ are all among the occupation
numbers is collapsing. The same holds for any von Neumann rule 
whose occupation numbers include all of $2,\ldots, 5$. Consequently,  Moore (resp.~von Neumann)
rules that have none of $0,\ldots 5$ (resp.~none of $0,\ldots, 3$) among occupation numbers   
are also collapsing.
\end{prop}

\begin{proof} 
Assume we have a von Neumann rule in which any site $x$ updates into state $1$ by contact 
with $2$ or more $1$s. The proof in the Moore case is similar, and the last 
two statements are proved by switching the roles of $0$s and $1$s.  
Call an $L\times L$ square {\it good\/} if 
the configuration within the square is such that no matter what the configuration 
outside the square is, the rule completely fills the square by $1$s in time $L^2$.
By the result in \cite{Sch},
for $L$ large enough (in fact, of size $\exp(cp^{-2})$, for some constant $c$), 
a fixed $L\times L$ square is good with probability at least $0.9$, 
Note also that once such a square is filled by $1$s and free of defects, no defect 
can ever enter it. 

Now tile $\bZ^2$ with $L\times L$ squares. As the critical site percolation probability 
on $\bZ^2$
is smaller than $0.9$, by time $L^2$ the good squares confine all defects into a 
finite set. Then that finite set is completely covered by $1$s in a finite (random) 
time and then the defects must all die as $11111\mapsto$ 1 is a stable update. 
\end{proof}

\begin{figure}[ht]
\begin{center}
\begin{subfigure}{.33\textwidth}
\centering
\includegraphics[trim=0cm 0cm 0cm 0cm, clip, width=4cm]{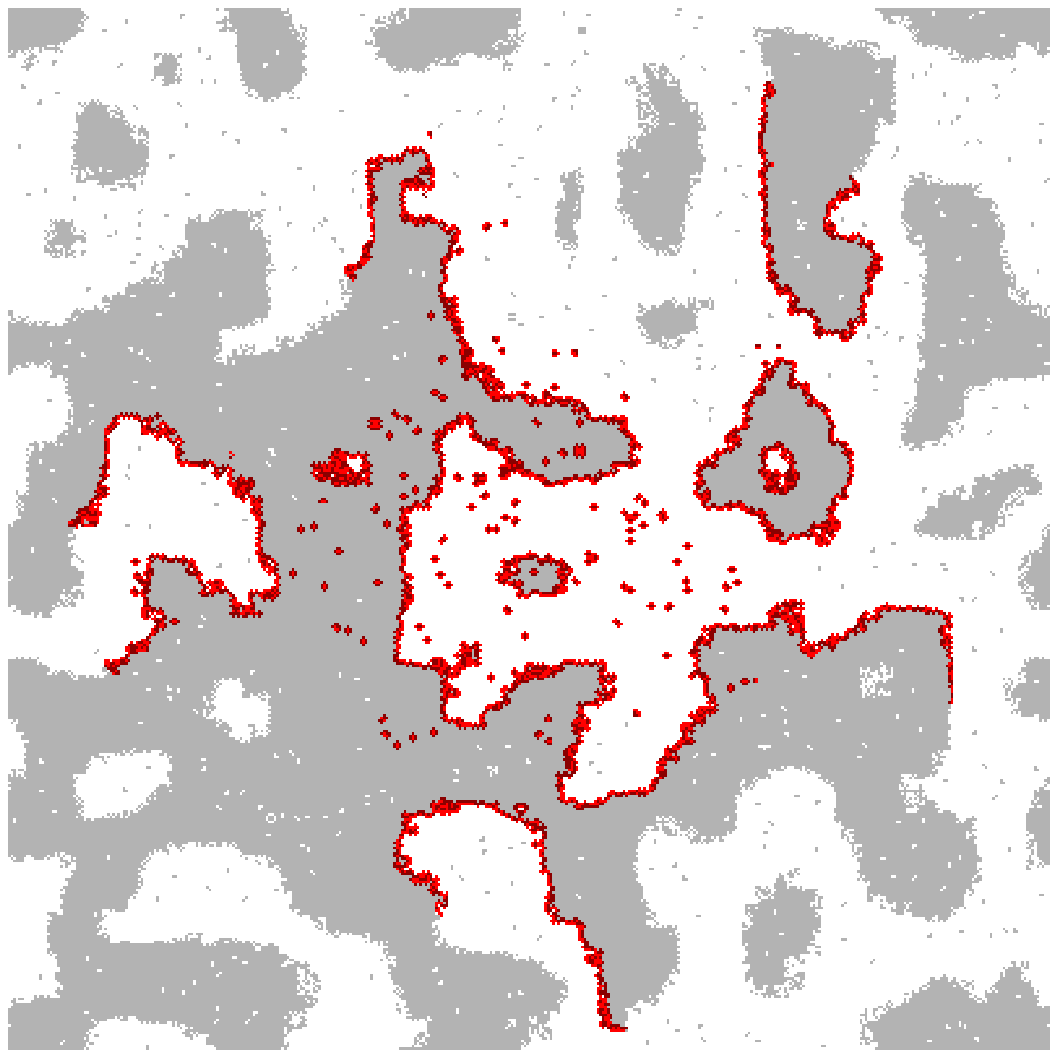} 
\caption{}
  \label{2d-figure:sub1}
\end{subfigure}
\hspace{-0.2cm}
\begin{subfigure}{.33\textwidth}
\centering
\includegraphics[trim=0cm 0cm 0cm 0cm, clip, width=4cm]{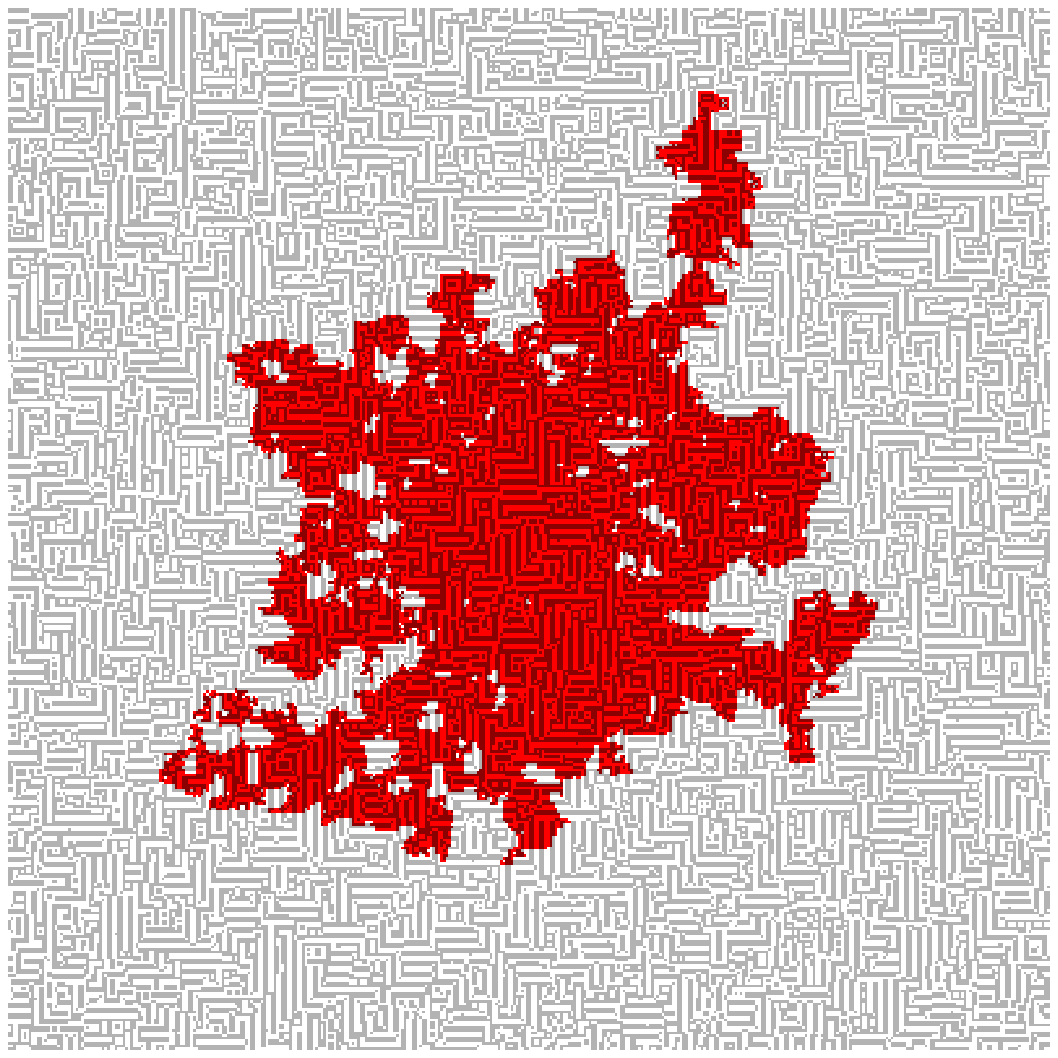}
\caption{}
  \label{2d-figure:sub2}
\end{subfigure}
\hspace{-0.2cm}
\begin{subfigure}{.33\textwidth}
\centering
\includegraphics[trim=0cm 0cm 0cm 0cm, clip, width=4cm]{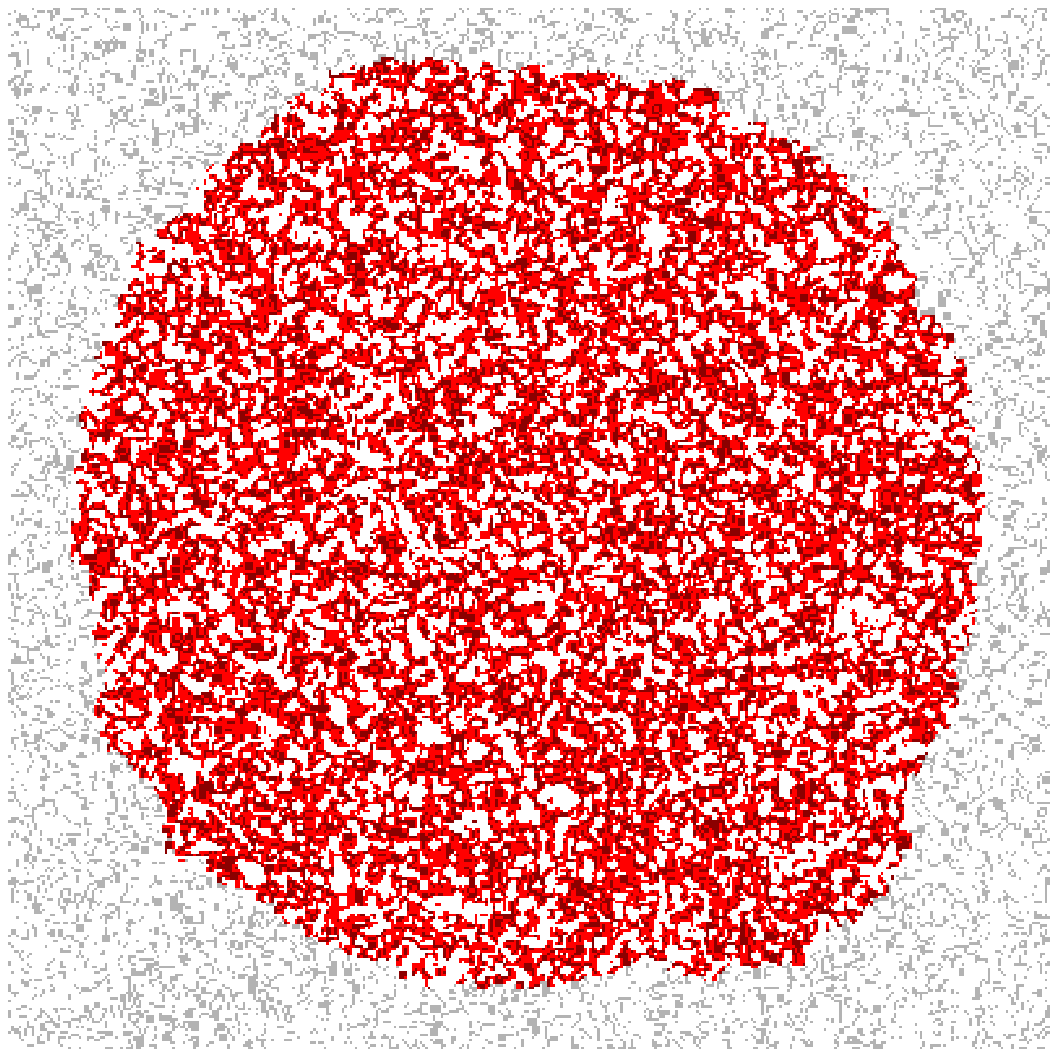}
\caption{}
  \label{2d-figure:sub3}
\end{subfigure}
\end{center}
\vspace{-0.5cm}
\caption{
Snapshots of two-dimensional evolution of defect percolation CA: (a) von Neumann {\it Tot 245\/},
(b) von  Neumann {\it Tot 125\/}, and (c) Moore {\it Tot 1\/}.
} \label{2d-figure}
\end{figure}

Needless to say, we have good empirical evidence that many more of these rules are 
collapsing. Possibly the most interesting cases are von Neumann 
{\it Tot 245\/} and Moore {\it Tot 46789\/} rules, both famously known as the 
{\it Vishniac twist\/} \cite{Vic1, TM}. In these rules, the defects can survive only 
on the border between $0$s and $1$s, and those borders anneal away, 
i.e., shrink and disappear due to a
process resembling surface tension. This is however a slow evolution during which the set of 
defect sites self-organizes into long ``noodles,'' as in Fig.~\ref{2d-figure:sub1}, which features 
the von Neumann case.

Among the notable apparently marginal rules, we mention the ``cauliflower'' von Neumann rule 
{\it Tot 125\/}, in which defect sites do spread while the state is close enough to the 
uniform product measure, but eventually the CA reaches a state that 
stops further defect growth; 
see Fig.~\ref{2d-figure:sub2}. 

Chaotic rules are very common among totalistic ones, thus expansive defects accumulation dynamics 
also abound. A typical example is Moore {\it Tot 1\/}, whose defect
percolation CA is illustrated in Fig.~\ref{2d-figure:sub3}, while its 
empirical Lyapunov profile (at time $t=300$) is depicted in Fig.~\ref{Tot1-figure}. 
We estimate its MLE to be about $1.24$. 
\vspace{-0.2cm}
\begin{figure}[ht]
\begin{center}
\includegraphics[trim=0cm 0cm 0cm 0.5cm, clip, width=7cm]{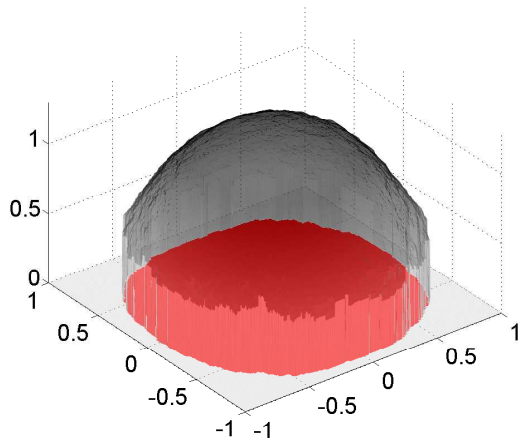}
\end{center}
\vspace{-0.5cm}
\caption{Approximation to Lyapunov profile $L$ for {\it Tot 1\/}. The ``chiseled'' 
boundary indicates drop to $-\infty$ and the red ``shadow'' approximates $\{L>-\infty\}=W$.} \label{Tot1-figure}
\end{figure}

\section{Periodic initial states}\label{section-periodic}

A configuration $\eta\in S^{\bZ^d}$ is {\it doubly periodic\/} for a 
CA with global map $\Phi$ if there exist
\begin{itemize}
\item a number $\pi\ge 1$ so that $\Phi^\pi(\eta)=\eta$; and
\item a number $\sigma\ge 1$ so that, for every $x$, $\eta(x)=\eta(x\mmod \sigma)$, where 
$x\mmod \sigma$ reduces every coordinate of $x$ modulo $\sigma$.
\end{itemize} 
We assume that $\pi$ and $\sigma$ are the smallest possible and refer to them respectively
as the {\it temporal\/} and {\it spatial period\/}.
When $d=1$, it is convenient to also introduce a {\it shift period\/} $\pi_0\ge 1$, the smallest
time at which there exists a {\it shift\/} $\sigma_0\in \{0,\ldots, \sigma-1\}$  such that 
the CA shifts $\eta$ to the right by $\sigma_0$ in $\pi_0$ steps:
$(\Phi^{\pi_0}\eta)(x)=\eta(x-\sigma_0)$ for all $x$. Note that $\pi_0$ divides $\pi$. 
In this section, we will assume that a doubly periodic configuration $\eta$ is the initial 
state $\xi_0$ for the CA dynamics. We often specify a periodic configuration $\eta$ by a {\it tile\/}, 
that is a configuration in $S^{[0,\sigma-1]^d}$ that gives $\eta$ on $[0,\sigma-1]^d$. 

One complication in the analysis of periodic orbits is caused by reducibility. 
To each $\eta$ we associate the {\it reduced kernel\/}
$$
K:\bZ_\sigma^d\times \bZ_\sigma^d \to \{0,1\}, 
$$
which has $K(a,b)=1$ exactly when the defect percolation dynamics starting 
from $\ind(a)$ results in $\de_\pi(x)=1$ for some $x=b\mmod \sigma$. We call 
$\eta$ {\it irreducible\/} if $K$ is irreducible. Clearly, we may check 
irreducibility at time $\pi_0$; more on this later. 

The doubly periodic configuration $\eta$ is {\it strongly irreducible\/} if 
there exists an $a\in \bZ^d$ such that, for every $x_0\in \bZ^d$, and $\de_0=\ind(x_0)$, 
$$
\cup_t\{x: \de(x+ta)=1\}=\bZ^d.
$$
If $\eta$ is irreducible but not strongly irreducible, then the set of points 
in $\{\de_t=1\}$ is included in a periodic space-time lattice.

\subsection{Defect shapes and density profiles}\label{periodic-shapes}

Without loss
of generality we assume in this section that $\eta$ is strongly irreducible and 
$a=0$. In our examples, we will commonly have  strong irreducibility 
if we neglect the sites with stable updates. We call such cases {\it essentially 
strongly irreducible\/}. We will assume that the initial set of defects $\de_0$ 
is a $\sigma\times \sigma$ square, to prevent their accidental death. In the 
essentially strongly irreducible cases, the density profile $\rho$ is constant 
on $W$ (and of course vanishes off $W$). 

For the next theorem, we let  $\cS^{d-1}\subset \bR^d$ be the set of unit 
vectors, that is, the set of directions in $d$ dimensions. The half space in direction 
$u$ is defined by 
$$
H_u^-=\{x\in \bR^d:\langle x, u\rangle\le 0\}.
$$

\begin{theorem}\label{2dshapes-thm}
For any unit vector $u\in \cS^{d-1}$,
there exists a number $w(u)\ge 0$ so
that, if $\de_0=H_u^-\cap \bZ^d$, 
$$
t^{-1}\de_t\to w(u)u+H_u^-
$$ as $t\to\infty$, in Hausdorff metric. 
Moreover, if we form the set
$$
K_{1/w}=\bigcup_{u\in \cS^{d-1}} \{\alpha u: 0\le \alpha\le 1/w(u)\},
$$
then the limiting shape is given by the polar transform of $K_{1/w}$, 
$$
W=K_{1/w}^*=\{x\in \bR^d:\langle x, u\rangle\le w(u)\}. 
$$
\end{theorem}

We refer to $w(u)$ as a {\it half-space velocity\/} \cite{GG1,GG3, Wil}.
In mathematical models of crystallography, $K_{1/w}$ is sometimes called 
the {\it Frank diagram\/} \cite{Gig}. In our case, as can be seen from 
the proof, $K_{1/w}$ is a convex polygon. In $d=2$ its vertices 
can only be in the directions orthogonal to lines through 
two points of the Minkowski sum of $\pi$ copies of $\cN$, 
i.e., $\{x_1+\ldots +x_\pi: x_1,\ldots, x_\pi\in \cN\}$. 

\begin{proof}
For simplicity of notation, we assume $\pi=1$; the proof is easily adapted to 
general $\pi$. 

Interpret a subset of $S\subset\bZ^d$ as a
$\sigma^d$-tuple of subsets $(S_a: a\in \bZ_\sigma^d)$, where $S_a=S\cap(a+\sigma\bZ^d)$. Denote the 
set of these tuples by $\Sigma$. Using one of these tuples as the $\de_0$, 
$\de_\pi$ may be interpreted as a map $\Psi:\Sigma\to\Sigma$. Let $\widetilde \Sigma$ be 
the set of all $\sigma^d$-tuples of subsets of $\bR^d$.  We define 
the map $\widetilde\Psi:\widetilde \Sigma\to\widetilde \Sigma$ as follows. 
The image of $(\widetilde S_a: a\in \bZ_\sigma^d)$ is the vector of sets $(\widetilde T_b: b\in \bZ_\sigma^d)$
such that 
\begin{equation}\label{2dshapes-eq1}
\widetilde T_b=\{x\in \bR^d:0\in\Psi((\widetilde S_a-x)\cap (a+\sigma\bZ^d):a\in \bZ_\sigma^d)_b\}.
\end{equation}
In words, at each $x$, the occupation of the set at coordinate $b$ 
is decided by translating $\bZ^d$ so that the $b$th lattice covers $x$, 
intersecting all sets with this translation, and then applying the 
discrete rule. It immediately follows from (\ref{2dshapes-eq1}) that 
the discrete and continuous rules are conjugate: 
\begin{equation}\label{2dshapes-eq2}
\Psi(\widetilde S_a \cap (a+\sigma\bZ^d):a\in \bZ_\sigma^d)=
(\widetilde \Psi(\widetilde S_a :a\in \bZ_\sigma^d)_b\cap (b+\sigma\bZ^d): b\in \bZ_\sigma^d).
\end{equation}
The continuous rule is useful because of its translation invariance when applied 
to half-spaces. To formulate this property, 
fix a direction $u\in \cS^{d-1}$ and a vector $(\alpha_a^0: a\in \bZ_\sigma^d)$.
Then, there exists a vector $(\alpha_a^1: a\in \bZ_\sigma^d)$ so that 
\begin{equation}\label{2dshapes-eq3}
\widetilde \Psi(\alpha_a^0 u+H_u^-:a\in \bZ_\sigma^d)=(\alpha_a^1 u+H_u^-: a\in \bZ_\sigma^d).
\end{equation}
Now iterate $\widetilde\Psi$ to get a sequence of vectors $(\alpha_a^t:a\in \bZ_\sigma^d)$, $t=0,1,\ldots$
Due to strong irreducibility and the discrete nature of the dynamics, 
there exist a number $w(u)\ge 0$ and an integer $k\ge 1$ so that, for a large enough $t$,
$$
\alpha_a^{t+k}-\alpha_a^t=kw(u), 
$$
for every $a$. Due to monotonicity, 
$w(u)$ is independent of the initial vector $(\alpha_a^0)$.
This proves the existence of the half-space velocities.
 Now the theorem follows from methods from \cite{Wil, GG1, GG3}.
Observe also that $\tilde \Psi$ is set-additive, that is, for any $\widetilde S_a, 
\widetilde S_a'\subset\bR^d$, 
$$
\widetilde \Psi(\widetilde S_a\cup \widetilde S_a':a\in \bZ_\sigma^d)=
\widetilde \Psi(\widetilde S_a:a\in \bZ_\sigma^d)\cup
\widetilde \Psi(\widetilde S_a':a\in \bZ_\sigma^d), 
$$
where the second union is coordinate-wise. Writing a half-space as a union of 
its points, this implies that $K_{1/w}=L^*$ and thus $K_{1/w}$ is convex. 
\end{proof}

We now turn to examples. We will restrict ourselves to two-dimensional Moore neighborhood
{\it Tot $\theta$} rules (see Section~\ref{2d-section}). 
We start with the observation that it is quite possible that 
$W=\emptyset$. For example,  $\eta\equiv 0$ is a fixed state (with $\sigma=\pi=1$) 
for {\it Tot $\theta$} when $\theta\ge 1$ and has $W=\emptyset$ when $\theta\ge 2$. 

We start with $\theta=1$. We have generated all possible doubly periodic states 
with $\sigma\le 4$. There are 12 of them (modulo symmetries of the lattice $\bZ^2$) 
and none have $W=\emptyset$, although in four cases the interior of $W$ is empty. 
We provide two examples: 
\begin{itemize}
\item tile {\scriptsize
$\begin{matrix} 
0000\\
0000\\
0011\\
0011
\end{matrix}$}, $\pi=2$, first quarter vertices of $W$  $(2/3,0)$, $(2/5, 2/5)$, $(0,2/3)$, and the defect
density profile $\rho|_W\equiv 3/4$ on $W$ (Fig.~\ref{ex1shapes-examples:sub1});
\item tile {\scriptsize
$\begin{matrix} 
0000\\
0011\\
0000\\
1100
\end{matrix}$}, $W=[-2/3,2/3]\times \{0\}$, which has empty interior, thus $\rho\equiv 0$
(Fig.~\ref{ex1shapes-examples:sub2}).
\end{itemize}
For the first of these, Fig.~\ref{ex1frank} illustrates the relationship between the 
Frank diagram (the larger outline with first quarter vertices $(1,2/3)$, $(2/3,1)$, and the shape described in Theorem~\ref{2dshapes-thm}.

\begin{figure}[ht]
\hskip1cm
\begin{minipage}[t]{0.4\linewidth}
\vspace{0pt}
\vskip0.3cm
\includegraphics[trim=0cm 1cm 0cm 1cm, clip, width=5cm]{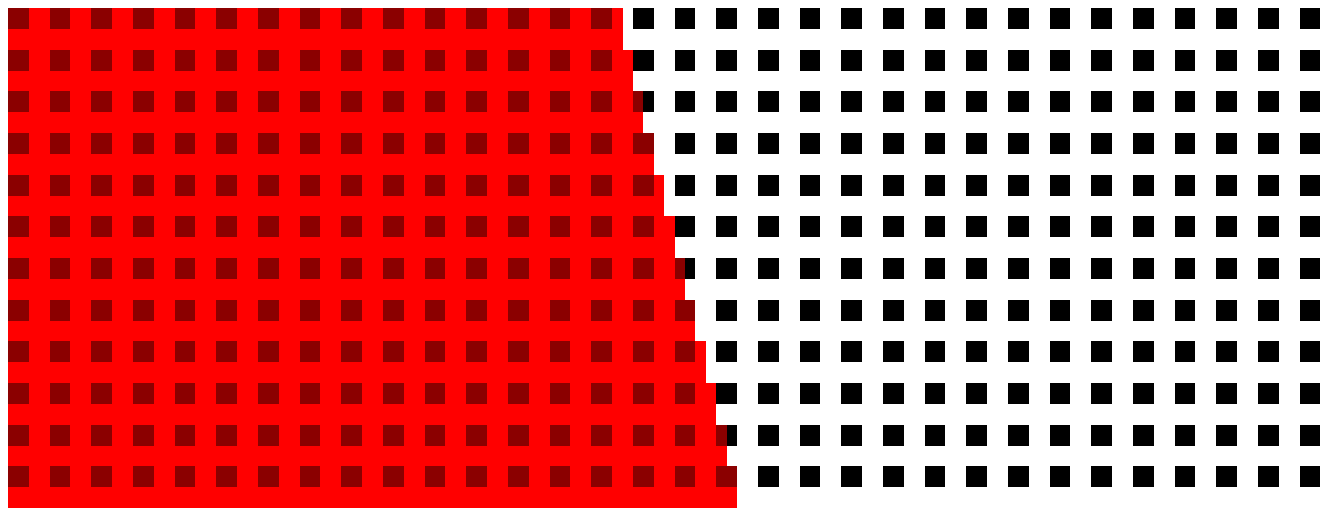}

\vskip-1cm
\includegraphics[trim=0cm 1cm 0cm 1cm, clip, width=5cm]{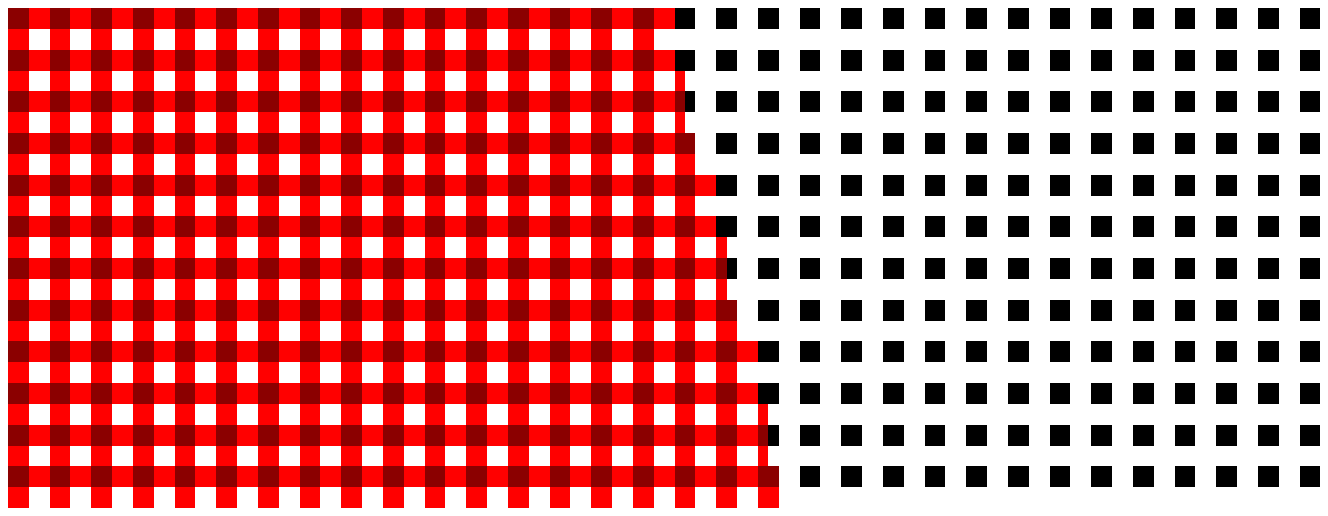}

\vskip-1cm
\includegraphics[trim=0cm 1cm 0cm 1cm, clip, width=5cm]{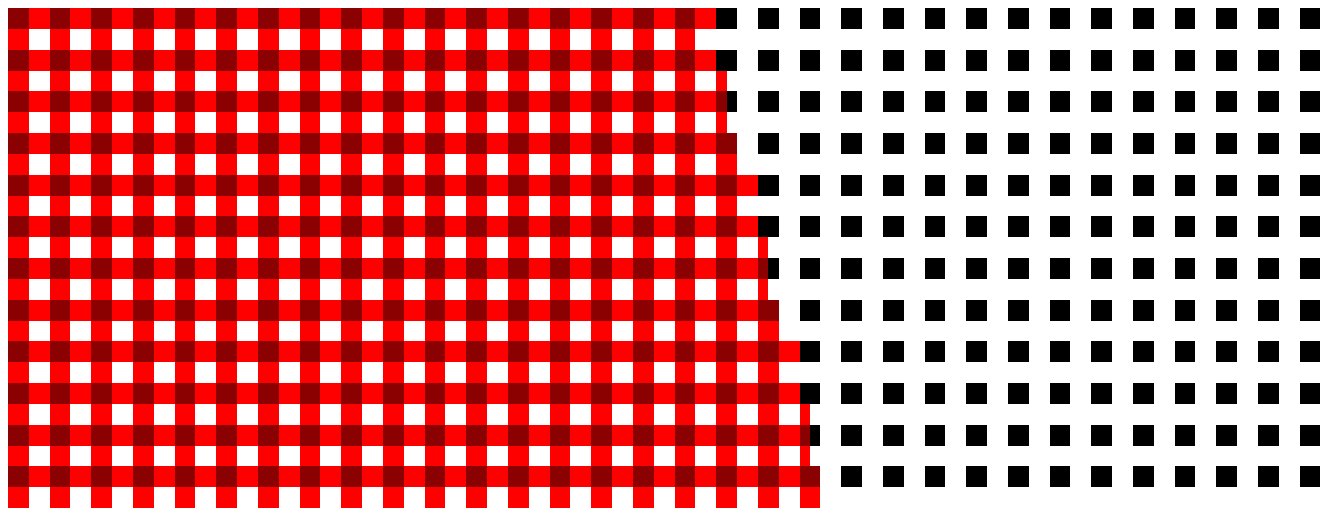}
\vfill\eject
\end{minipage}
\hskip-1cm
\begin{minipage}[t]{0.4\linewidth}
\vspace{0pt}
\includegraphics[trim=0cm 0cm 0cm 0cm, clip, width=8cm]{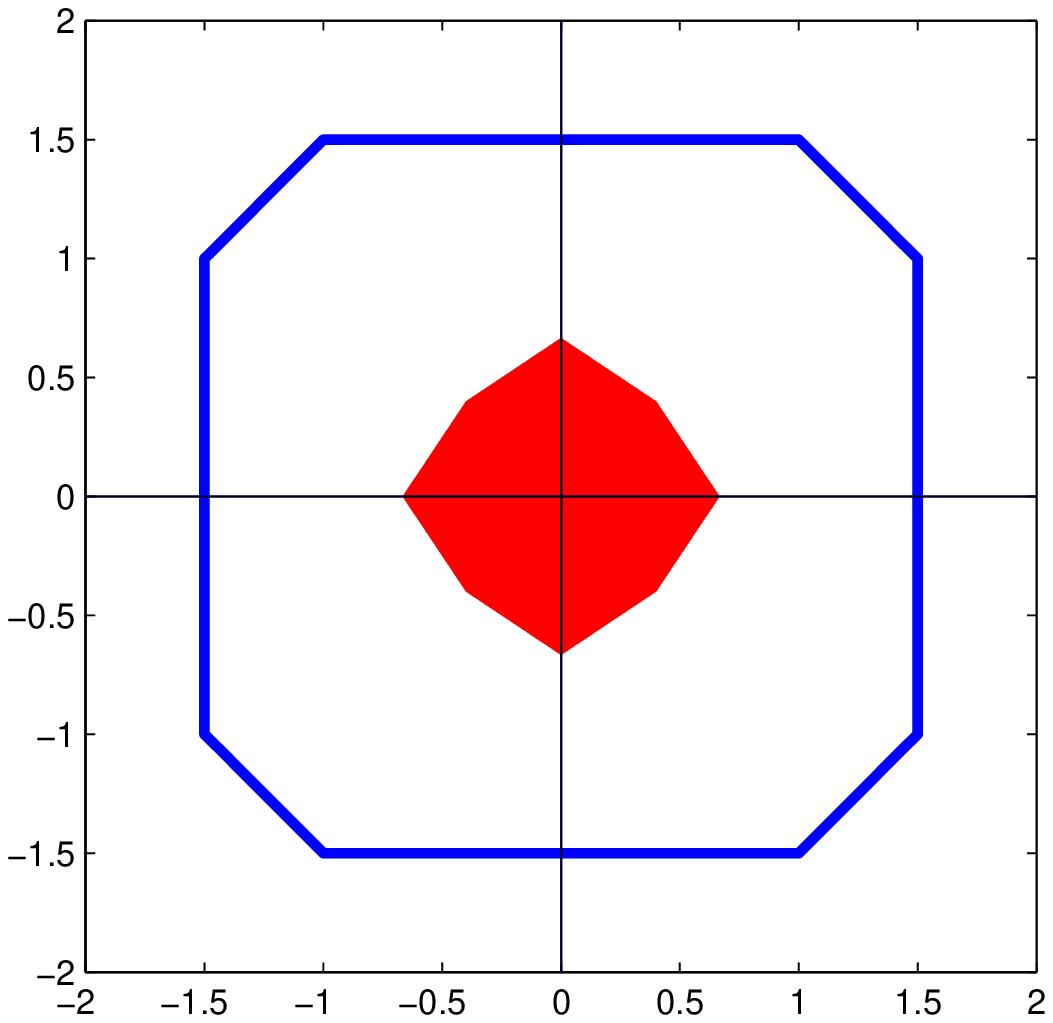}
\end{minipage}
\vspace{-0.5cm}
\caption{
The shape characterization for our first {\it Tot 1\/} example in the text. Left: 
propagation of the half-plane with boundary slope $-4$, depicted at times  
$t=0,6,12$.  The configuration at $t=12$ is a horizontal translation by $4$
of the one at $t=6$, which results in the advancement by $4$ every $6$ time steps.
Right: defect shape from the Frank diagram.}
\label{ex1frank} 
\end{figure}

When $\theta=3$ there are 24 doubly periodic states with $\sigma=4$, of which we selected a 
nonsymmetric shape:
\begin{itemize}
\item 
 tile {\scriptsize
$\begin{matrix} 
0000\\
0001\\
0010\\
1001
\end{matrix}$}, 
with $\pi=6$, eleven vertices 
$(\pm2/3,-1)$,
$(8/9,-8/9)$,
$(1,\pm2/3)$,
$(\pm 2/3,1)$,
$(-8/9,8/9)$,
$(-1,2/3)$,
$(-1,-1/3)$,
$(-8/9,-2/3)$, and $\rho|_W\equiv 5/6$ (Fig.~\ref{ex1shapes-examples:sub3}).
\end{itemize}
Our final example has $\theta=7$, 
\begin{itemize}
\item
tile 
 {\scriptsize
$\begin{matrix}
0111\\
1011\\
1110\\
1101\\
\end{matrix}$}
and $\pi=2$ (Fig.~\ref{ex1shapes-examples:sub4}). 
This case is clearly not essentially strongly 
irreducible. In fact, it is easy to check that defects 
on $0$s and $1$s do not communicate.  On $1$s the defects spread as fast as the light 
cone, resulting in the defect shape $[-1,1]^2$. However, the spread on $0$s is considerably 
slower, resulting in the inner symmetric octagon with two of its vertices $(1,0)$, $(2/3,2/3)$.
This octagon is not visible in the defect shape, but clearly shows up 
in the defect density profile $\rho$, 
which is $1$ on the octagon and $3/4$ on the region between the square and the octagon.
\end{itemize}
 
\begin{figure}[ht]
\begin{center}
\begin{subfigure}{.25\textwidth}
\centering
\includegraphics[trim=0cm 0cm 0cm 0cm, clip, width=3.5cm]{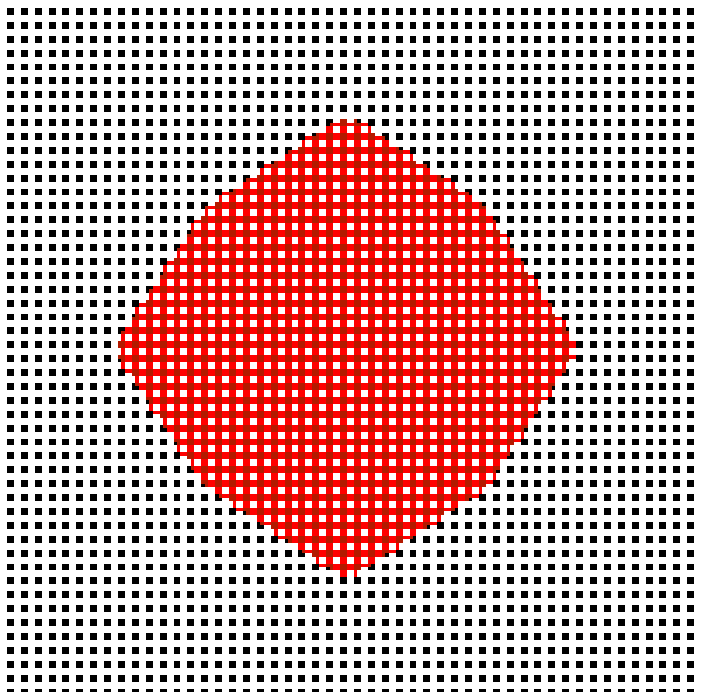}
\caption{}
  \label{ex1shapes-examples:sub1}
\end{subfigure}
\hskip-0.2cm
\begin{subfigure}{.25\textwidth}
\centering
\includegraphics[trim=0cm 0cm 0cm 0cm, clip, width=3.5cm]{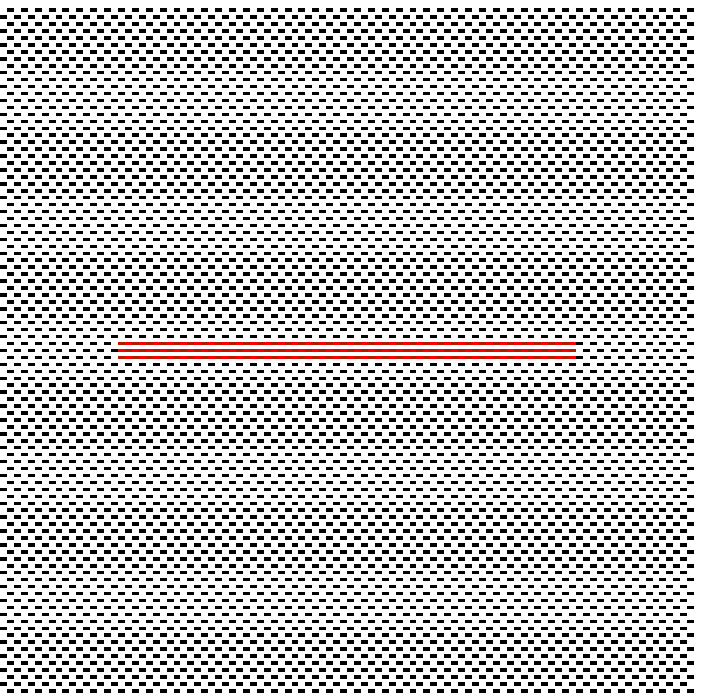}
\caption{}
  \label{ex1shapes-examples:sub2}
\end{subfigure}
\hskip-0.2cm
\begin{subfigure}{.25\textwidth}
\centering
\includegraphics[trim=0cm 0cm 0cm 0cm, clip, width=3.5cm]{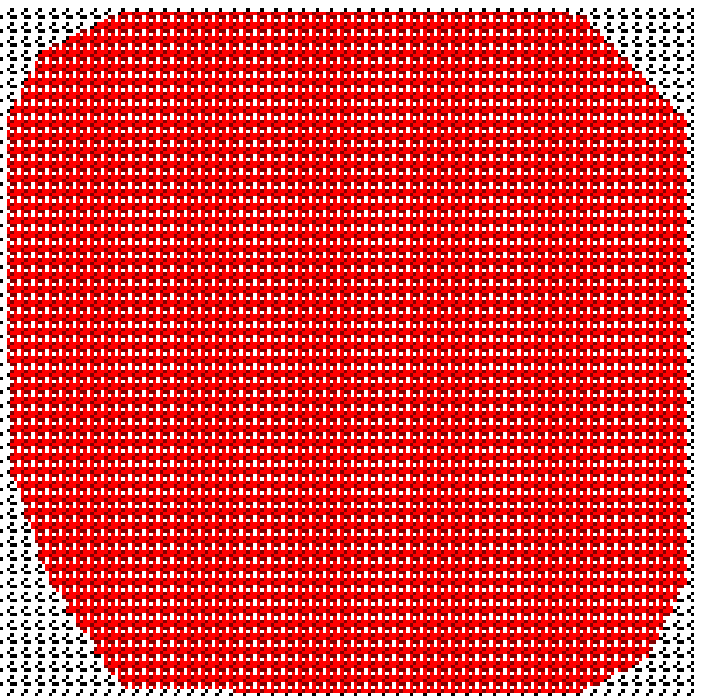}
\caption{}
  \label{ex1shapes-examples:sub3}
\end{subfigure}
\hskip-0.2cm
\begin{subfigure}{.25\textwidth}
\centering
\includegraphics[trim=0cm 0cm 0cm 0cm, clip, width=3.5cm]{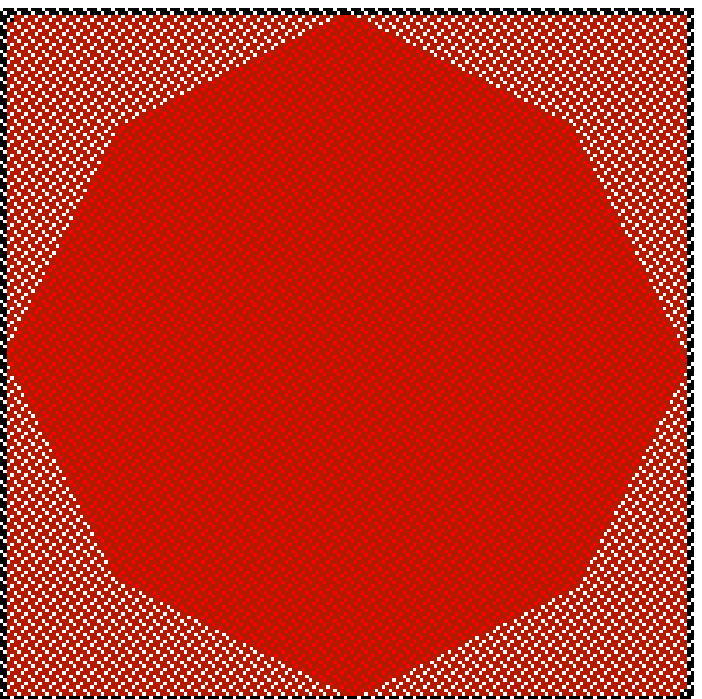}
\caption{}
  \label{ex1shapes-examples:sub4}
\end{subfigure}
\end{center}
\caption{Examples of defect shapes for  {\it Tot 1\/} ((a) and (b)), 
{\it Tot 3\/} (c), and {\it Tot 7\/} (d) rules.}
\label{ex1shapes-examples} 
\end{figure}

\subsection{Lyapunov profiles in one dimension}\label{periodic-profiles}

Our discussion on Lyapunov profiles of doubly periodic configurations will be 
limited to $d=1$ for simplicity. Most of our techniques extend readily to higher 
dimensions. 

\subsubsection{Variational principle in the irreducible case}\label{periodic-profiles-irreducible}

The input for the Lyapunov profile computation is the {\it expansion graph\/} $\cE$
that we define first. The vertices of this directed graph are numbers in 
$\bZ_\sigma=\{0,\ldots, \sigma-1\}$ and we attach to each edge $e$ of $\cE$ 
a {\it displacement 
label\/} $D(e)$ and a {\it size label\/} $N(e)$. 
For an $i\in \{0,\ldots,\sigma-1\}$, assume $\De_0$ is $1$ at $i$ and $0$ otherwise.
Suppose $\De_{\pi_0}$ has $n_i\ge 0$ nonzero values $N_j^{(i)}$ at
$i+D_j^{(i)}$, $j=1,\ldots, n_i$; each of these values generates an edge 
$(i,i+D_j^{(i)}-\sigma_0\mod \sigma)$ of $\cE$ emanating from $i$, 
with the displacement label $D_j^{(i)}$ and  size label $N_j^{(i)}$.
Note that an oriented pair of vertices $(i_1, i_2)$ may be joined 
by multiple edges with distinct displacement 
labels. Let $d_T=\sum_i n_i$ be the number of edges of $\cE$.
Assume the edges of $\cE$ are ordered $e_1,\ldots, e_{d_T}$, 
say lexicographically
among the oriented pairs of vertices and by increasing displacement label within the same oriented pair.

Construct the $d_T\times d_T$ matrix $T$ as follows. 
If edges $e_k$ and $e_{\ell}$ connect ordered
pairs $(i_1,j_1)$ and $(i_2,j_2)$ then 
$$
T_{k, \ell}=
\begin{cases}
N(e_k)  &\text{if } j_1=i_2\\
0 &\text{otherwise}
\end{cases}
$$
The {\it weight matrix\/} $W_y$, which depends on a real parameter $y$, 
is a diagonal matrix of the same size as $T$ given (using the order of edges) by
\begin{equation}\label{diagonal-matrix}
W_y=\text{diag}(
\exp(yD(e_k)), k=1,\ldots, d_T).
\end{equation}
The much simpler matrix is $T'$ is a $\sigma\times \sigma$ matrix indexed 
by vertices of $\cE$
with entries
$$
T_{i_1, i_2}'=\sum_{e\in \cE\text{ connects }(i_1,i_2)} N(e).
$$
Thus the matrix $T'$ counts defect paths that connect the $\sigma$ phases, while 
$T$ keeps track of their displacements as well. 

The large deviation principles that determine $L$ have particularly simple variational 
form when $\eta$ is irreducible, and therefore both $T$ and $T'$ are 
irreducible. This is the setting in the next theorem. 
We use the notation $\spr$ for the spectral radius of a matrix.

\begin{theorem}\label{LP-irreducible-thm}
Assume that $\eta$ is irreducible. Then $L$ is proper,
independent of $\De_0$,
and is given as follows. Let
\begin{equation}\label{log-spr}
\Lambda(y)=\log\spr(T\cdot W_y).
\end{equation}
Then the Lyapunov profile $L$ is given by the Legendre transform of $\Lambda$
that is, by 
\begin{equation}\label{Lp-eq}
L(\alpha)=\inf_{y\in \bR}(-y\alpha+\Lambda(y)).
\end{equation}
Furthermore, let $\lambda_1=\spr(T)$ be the 
largest eigenvalue of $T$. Then the MLE is given by 
\begin{equation}\label{MLE-eq}
\lambda=\log\spr(T')=\log\lambda_1.
\end{equation} 
For $k=1\ldots, d_T$ 
define constants $c_k$ so that the $k$th diagonal element $T^n_{kk}$ of $T^n$ satisfies
\begin{equation}\label{MLE-ck}
T^n_{kk}\sim c_k\lambda_1^n
\end{equation}
as $n\to\infty$. Then 
the unique MLE direction equals 
\begin{equation}\label{MLE-direction-eq}
\sum_{k=1}^{d_T} c_kD(e_k).
\end{equation} 
\end{theorem}

\begin{proof}
Apart from (\ref{MLE-direction-eq}), the claims follows from
standard large deviation theory and Perron-Frobenius theory 
(see Section 3.1 in \cite{DZ}).

To verify (\ref{MLE-direction-eq}), we use further results 
on asymptotics of nonnegative matrices. By Section 5 of \cite{FS}, 
there exist a diagonal matrix 
$\Gamma=\text{diag}(\gamma_1,\ldots, \gamma_{d_T})$ 
with all $\gamma_i>0$ and a stochastic matrix $P$ so that 
$T=\lambda_1\Gamma^{-1}P\Gamma$. Then
$$
T_{ij}^n=\lambda_1^n\gamma_i^{-1}\gamma_j P_{ij}^n.
$$
Let $\mu=(\mu_1,\ldots, \mu_{d_T})$ be the probability 
measure that is a left eigenvector of $P$. 

Assume that the initial edge is $e_k$ and that 
$D$ is $1$ on $e_1$ and $0$ otherwise. By linearity,
this suffices. The expected proportion of the edge
$e_1$ on a path of length $t$ chosen uniformly at random is
$$
\frac 1{t+1} \sum_{s=0}^{t}\frac {\sum_\ell T_{k1}^s T_{1\ell}^{t-s}}
{\sum_\ell T_{k\ell}^t}
\xrightarrow[{t\to\infty}]{} \frac{\sum_\ell \gamma_k^{-1}\gamma_1\mu_1\gamma_1^{-1}\gamma_\ell\mu_\ell}
{\sum_\ell \gamma_k^{-1}\gamma_\ell\mu_\ell}=\mu_1
$$
and 
$$T_{11}^n\sim\mu_1\lambda_1^n,$$
so that $c_1=\mu_1$ in (\ref{MLE-ck}).
A similar computation also handles the second moment and finishes the proof. 
\end{proof}

The constants $c_k$ can be readily obtained by linear algebra; for example, 
if $T$ has an invertible eigenvector matrix $V$, with the first 
column being the eigenvector of $\lambda_1$, and we let $I_{1}$ be the matrix with 
a $1$ at position $11$ and $0$s elsewhere, then $c_k=(VI_{1}V^{-1})_{kk}$. 

Next, we give three examples. 
In Figs.~\ref{ex1_lp} and~\ref{r110ether_lp}, we compare the approximation to the 
defect profile at a modest finite time 
to the limit given by Theorem~\ref{LP-irreducible-thm}.

\subsubsection{Two examples for {\it Rule 22\/}}\label{periodic-profiles-rule22}

For illustration we begin with perhaps 
the simplest nontrivial case, the fixed point $(10)^\infty$
for {\it Exactly 1\/}. Hence, $\sigma=2$ and $\pi=1$. We specify the tile to be $10$; we 
always assume that the leftmost state of the tile is at the origin, which specifies
the states of $\cE$. It is easy to check that a defect at $(x,t)$:
\begin{itemize}
\item  creates 3 children located on a 1 at $(x,t+1)$ and on $0$s at
$(x\pm 1, t+1)$, if $\ca_t(x)=1$; and
\item  creates 2 children on $1$s at  
$(x\pm 1, t+1)$, if $\ca_t(x)=0$. 
\end{itemize}
This describes the graph $\cE$, which has 5 edges,  
thus $d_T=5$,
$$
T=
\begin{bmatrix}
0 &0 &0 &1 &1\\ 
1 &1 &1 &0 &0\\ 
   0 &0 &0 &1 &1\\
   1 &1 &1 &0 &0\\
   1 &1 &1 &0 &0\\
\end{bmatrix}
$$
and $D=(-1,0,1,-1,1)$. 
The resulting density profile, which is
given in Fig.~\ref{ex1_lp:sub1}, 
is nonnegative on $W=[-1,1]$,
vanishes at the boundary, and its MLE is about $0.941$.
In fact, we can give the precise 
value for the MLE, 
$$
\lambda=\log\spr T'=\log\spr \begin{bmatrix} 1 & 2\\ 2 & 0 \end{bmatrix}=\log\frac{1+\sqrt{17}}2.
$$

For our second {\it Exactly 1\/} example, consider the doubly periodic configuration given 
by the tile $1^60^2$, 
which has $\sigma=8$, $\pi=6$, $\pi_0=3$ and $\sigma_0=4$.
Now the defects at the middle two $1s$ die due to the fact that $111\mapsto 1$ 
is a stable update. Thus we only need to consider 6 states for our graph $\cE$. 
This graph has no multiple edges, so we only need to specify the matrix $T'$ and 
the displacements associated with each entry. These are given in the Table~\ref{ex1-table2}, 
from which we conclude that $d_T=22$.  
The resulting Lyapunov profile is given in Fig.~\ref{ex1_lp:sub2}. In this case 
there is a nontrivial defect density $\rho$ that equals $3/4$ on $W=[-2/3,2/3]$, 
$L$ equals $\log(2)/6\approx 0.116$ at $\pm 2/3$, and the
MLE is about $0.638$. 

\begin{table}
\centering
\begin{quote}
\caption{Specification for the graph $\cE$ for the {\it Exactly 1\/} example
with tile $1^60^2$. 
The states 2 and 3, which produce no children, are left out.}\label{ex1-table2}
\end{quote}
\vspace{-0.8cm}
\begin{tabular}{| c | c | c | c | c | c | c | c |}
			\hline
 \begin{tabular}{@{}c@{}}$i\in\bZ_\sigma$\end{tabular}
& \begin{tabular}{@{}c@{}}tile\end{tabular}
& \begin{tabular}{@{}c@{}} $N_0^{(i)},D_0^{(i)}$\end{tabular}
& \begin{tabular}{@{}c@{}} $N_1^{(i)},D_1^{(i)}$\end{tabular}
& \begin{tabular}{@{}c@{}} $N_4^{(i)},D_4^{(i)}$\end{tabular}
& \begin{tabular}{@{}c@{}} $N_5^{(i)},D_5^{(i)}$\end{tabular}
& \begin{tabular}{@{}c@{}} $N_6^{(i)},D_6^{(i)}$\end{tabular}
& \begin{tabular}{@{}c@{}} $N_7^{(i)},D_7^{(i)}$\end{tabular}
  \\
			\hline
			\hline
0 & 1 &   & $1,-3$  & $4,0$  & $3,1$  & $1,2$  &   \\ \hline
1 & 1 &   &       & $2,-1$ & $2,0$  & $1,1$  &  \\ \hline
4 & 1 &$2,0$& $2,1$   &      &      &      & $1,-1$ \\ \hline  
5 & 1 &$3,-1$&$4,0$   & $1,3$  &      &      & $1,-2$ \\ \hline       
6 & 0 &$1,-2$&$3,-1$  & $3,2$  & $1,3$  &      &  \\ \hline       
7 & 0 &$1,-3$&$3,-2$  & $3,1$  & $1,2$  &      &  \\ \hline       
\end{tabular}
\end{table}
 
\begin{figure}[ht]
\begin{center}
\begin{subfigure}{.45\textwidth}
\centering
\includegraphics[trim=1cm 2cm 0cm 0cm, clip, width=8cm]{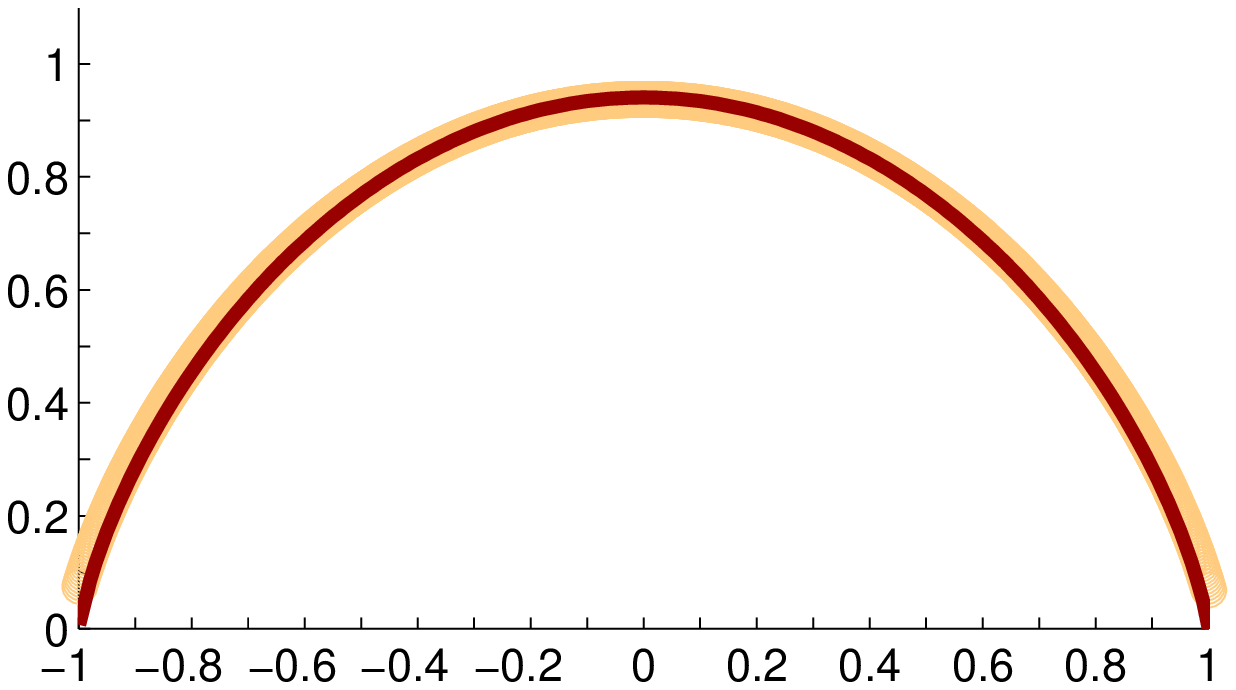}
\caption{}
  \label{ex1_lp:sub1}
\end{subfigure}
\hskip0.5cm
\begin{subfigure}{.45\textwidth}
\centering
\includegraphics[trim=1cm 2cm 0cm 0cm, clip, width=8cm]{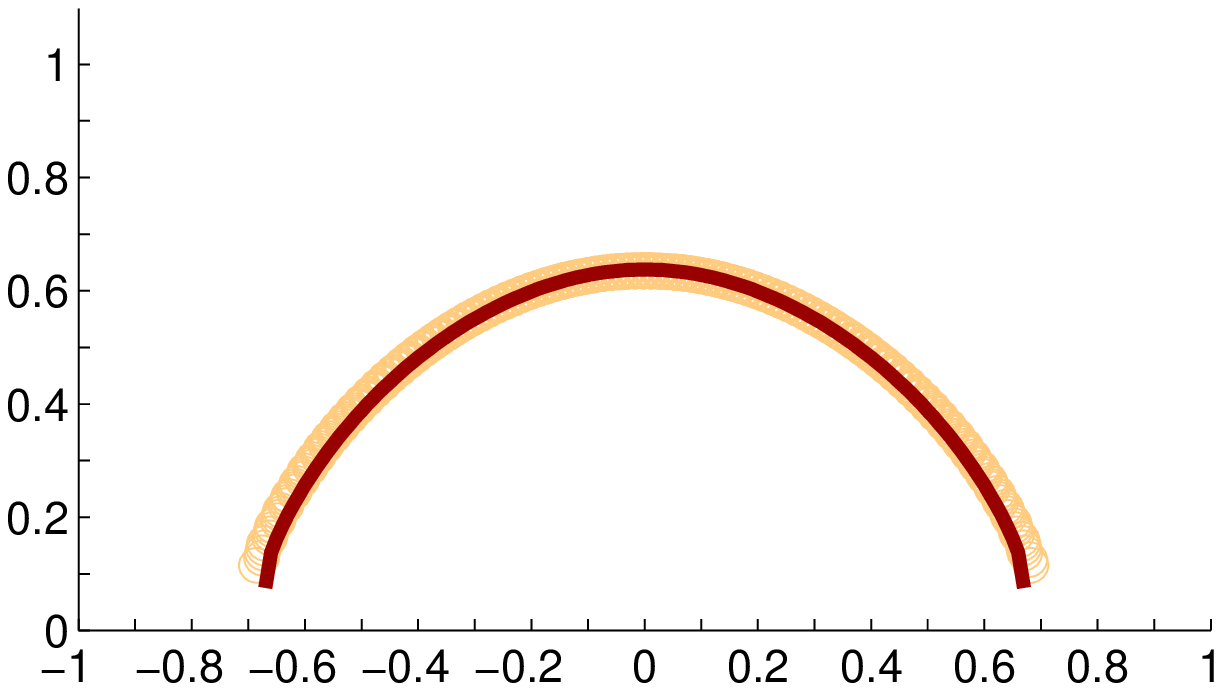}
\caption{}
  \label{ex1_lp:sub2}
\end{subfigure}
\end{center}
\vskip-0.5cm
\caption{The Lyapunov profiles (dark red) on $[-1,1]$ and their approximations 
(light red circles) at time
$560$ for two {\it Exactly 1\/} examples, with doubly periodic initial states (a) $(10)^\infty$, 
and (b) $(1^60^2)^\infty$.}
\label{ex1_lp} 
\end{figure}

\subsubsection{A {\it Rule 110\/} example}\label{periodic-profiles-rule110}

Perhaps the most important example of our method is the Lyapunov profile 
for the {\it Rule 110\/} {\it ether\/} \cite{Coo}. This is a doubly
periodic solution with $\sigma=14$,  $\tau=7$, 
$\tau_0=1$, $\sigma_0=10$, and tile $1^50^310^21^20$.
This ether supports a variety of gliders with complex interactions
(in fact, as complex as possible
\cite{Coo}). As mentioned in Section~\ref{intro}, it remains unresolved whether,
starting from the uniform product measure, the Lyapunov profile 
agrees with the one started from the ether.
 We now proceed to describe the 
latter profile. The expansion graph is rather sparse and is given in 
Table~\ref{rule110-table}:
for any $i$, and an edge $i\to j$, $j$ is given in the 
column corresponding to $D_j^{(i)}$ either $-1$, $0$ or $1$; these are the only 
displacement values and all corresponding $N_j^{(i)}=1$. Thus $d_T=26$.
 
\begin{table}[ht!]
\centering
\begin{quote}
\caption{Specification for the graph $\cE$ for the {\it Rule 110\/} ether. 
}
\label{rule110-table}
\end{quote}
\vspace{-0.4cm}
\begin{tabular}{| c | c | c | c | c |c|}
			\hline
 \begin{tabular}{@{}c@{}}$i\in\bZ_\sigma$\end{tabular}
& \begin{tabular}{@{}c@{}}tile\end{tabular}
& \begin{tabular}{@{}c@{}}$j$: $D_j^{(i)}=-1$\end{tabular}
& \begin{tabular}{@{}c@{}}$j$: $D_j^{(i)}=0$\end{tabular}
& \begin{tabular}{@{}c@{}}$j$: $D_j^{(i)}=1$\end{tabular}
  \\
			\hline
			\hline
0 & 1 & 4 &   & 6 \\ \hline
1 & 1 &   & 6 & 7 \\ \hline
2 & 1 & 6 & 7 & 8 \\ \hline
3 & 1 & 7 & 8 &       \\ \hline
4 & 1 & 8 & 9 &        \\ \hline
5 & 0 & 9 &10 &       \\ \hline
6 & 0 &10 &11 &      \\ \hline
7 & 0 &11 &   &       \\ \hline
8 & 1 &12 &13 &        \\ \hline
9 & 0 &   &14 &       \\ \hline
10 &0 &14 &   &2       \\ \hline
11 &1 & 1 &   &        \\ \hline
12 &1 &   & 3 &        \\ \hline
13 &0 & 3 & 4 &5       \\ \hline
\end{tabular}
\end{table}

The profile, given in Fig.~\ref{r110ether_lp}, is nonnegative on $[-8/9,2/3]$, 
vanishes at $2/3$ and equals $\log 3/9\approx 0.122$ at $-8/9$. 
The MLE equals about $0.647$ and is attained at
the MLE direction about $-0.276$.
We remark that the defect shape and values of 
$L$ at the boundaries, obtained here by a boundary 
analysis of defect dynamics, are closely connected to the spectral behavior of
perturbed nilpotent matrices \cite{EM}. 

\begin{figure}[ht!]
\vspace{1cm}
\begin{center}
\includegraphics[trim=0cm 0cm 0cm 1cm, clip, width=10cm]{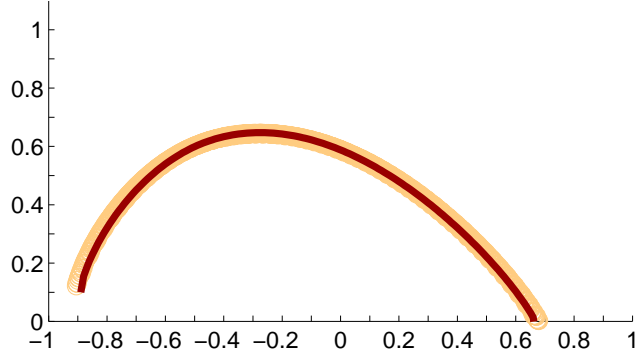}
\end{center}
\vspace{-1.5cm}
\caption{The Lyapunov profile (dark red) on $[-1,1]$ and its approximation (light red circles) at time
$560$ for the {\it Rule 110\/} ether.}
\label{r110ether_lp} 
\end{figure}

\subsubsection{Variational principle in the reducible case}\label{periodic-profiles-reducible}

If $T'$ is not essentially irreducible, but contains states 
that connect to several irreducible classes one can still characterize the Lyapunov profile
by a variational principle, which is, however, multidimensional.
We will state it below, but we first 
give two examples to show that the defect shape is 
not necessarily convex and that
the defect profile is not necessarily a concave function. The simplest ECA example 
is {\it Rule 184\/} with doubly periodic state with tile $01$ which has 
$\sigma=\pi=2$.  This generates $W=\{-1,1\}$ with $L=0$ on $W$. For a simple 
example with $W=[-2,2]$, consider
the CA with $\cN=\{0,\pm 1, \pm 2\}$ and the update function given by  
$00101\mapsto 0$, $01011\mapsto 1$, and in all other cases $abcde\mapsto c$. Clearly, 
$(01)^\infty$ is a fixed point, thus has $\pi=1$. Also, it is easy to see that, 
provided that the support of $\De_0$ includes both an even integer and 
an odd integer, the profile is given by
$$
L(\alpha)=\begin{cases}
-\dfrac{|\alpha|}2 \log\dfrac{|\alpha|}2-\left(1-\dfrac{|\alpha|}2\right) \log\left(1-\dfrac{|\alpha|}2\right)&\text{ if } |\alpha|\le 2,\\
-\infty &\text{ otherwise. }
\end{cases}
$$
In this case the MLE equals $\log 2$, and in both examples there are two MLE directions, namely 
$\pm 1$. 

\newcommand{\measures}{\mathcal P}
\newcommand{\kernels} {\mathcal K}  
Let $\measures$ be the set of probability measures on $\{1,\ldots, d_T\}$. 
For a given $\alpha\in \bR$, let 
$$
\measures_\alpha=\{(\mu_1,\ldots, \mu_{d_T})\in \cP: \sum_k\mu_kD(e_k)=\alpha\}.
$$
Write $k\leadsto \ell$ if $k=\ell$ or $T^n_{k\ell}$ is positive for some $n$; 
that is, an oriented path in the graph $\cE$ leads from edge $e_k$ to edge $e_\ell$. 
Moreover, for a given $b\in\{1,\ldots, d_T\}$, let 
$$
\begin{aligned}
\measures_b'=\{(\mu_1,\ldots, \mu_{d_T})\in \cP: &\text{ for all $k,\ell\in \{1,\ldots,d_T\}$,}\\
 &\text{ if } b\not\leadsto k\text{ then }\mu_k=0
\text{, and if } \ell\not\leadsto k\text{ and }k\not\leadsto \ell\text{ then }\mu_k\mu_\ell=0\}.
\end{aligned}
$$
For any $\mu\in \measures$, let $\kernels_\mu$ be 
the set of all $d_T\times d_T$ stochastic matrices $q=(q_{k\ell})$  that leave $\mu$ invariant, 
that is, they have positive entries and satisfy $\sum_\ell q_{k\ell}=1$, for all $k$, 
and $\sum_k\mu_kq_{k\ell}=\mu_\ell$, for all $\ell$.
The expression that plays a role related to the relative entropy is the function 
$H$ defined on $\kernels\times \measures$ by
$$
H(q,\mu)=\sum_{k,\ell} \mu_kq_{k\ell} \log \frac {T_{k\ell}}{q_{k\ell}}.
$$
 
\begin{theorem} \label{LP-general-thm}
Assume that a doubly periodic state $\eta$ is the 
initial CA state $\ca_0$. Fix also an initial set 
$\De_0$ and let 
$$
B_0=\{b\in \{1,\ldots, d_T\}: \text{ the edge }e_b\text{ originates from }x\mmod \sigma 
\text{ for some }x\in \De _0\}.$$
Then the Lyapunov profile is proper and given by the following triple supremum
$$
L(\alpha)=\sup_{b\in B_0}\,\,\,\sup_{\mu\in \measures_\alpha\cap \measures'_b}
\,\,\,\sup_{q\in \kernels_\mu} 
H(q,\mu).
$$
\end{theorem}

\begin{proof}
Assuming the defect paths must start with a 
fixed $b\in B_0$, the result follows from the general large deviation theorem 
for finite Markov chains (see Corollary 13.6 and Section 13.3 in \cite{RS}) and 
the Contraction principle (Section 4.2.1 in \cite{DZ}). Further, it is clear that the  
profile is obtained by the supremum over all possible choices of edges out of
$\De_0$.
\end{proof}

\section{Conclusions and open problems}

The introduced non-equilibrium defect dynamics allows a simultaneous study of both the 
spatial extent and 
local accumulation of perturbations on a CA trajectory. The resulting Lyapunov profiles
reveal quite a bit more information than the equilibrium version 
of Bagnoli et al.~\cite{BRR}. In particular, we provide a division of CA trajectories
into three classes: in expansive cases defects spread (on the lattice and in their 
state space), in collapsing cases they die out, and in marginal cases they do neither 
of the two. Employing a mixture of rigorous and empirical methods, we classify
all elementary CA starting from translation invariant product measures. 
We also make theoretical progress in the case of periodic initial conditions, 
where asymptotic shapes and large deviation rates are the main components of a 
Floquet theory for CA. 

Our approach retains some 
of the spirit of the Wofram's damage spreading \cite{Wol1}, although, as we have 
seen, it is fundamentally different and further insights into connections between the 
two would be welcome. In fact, the entire paper can be read as an invitation into a
new topic with a wealth of intriguing open problems (many of which were mentioned 
in previous sections), and we conclude with a selection of
them:

\begin{enumerate}
\item Can one prove that a CA trajectory has a proper Lyapunov profile under 
general conditions? Is there a simple example with a non-proper profile?
\item Can one understand which properties of a CA cause a phase transition between 
marginal and expansive dynamics 
as the initial density $p$ of $1$s 
varies, such as in the example at the end of Section~\ref{prelim-defdam}? Can 
one determine the critical $p$ in that example?
\item For {\it Rule 38\/} and other expansive stripes rules, is it possible to 
provide rigorous (numerical) bounds on the MLE and its direction? 
\item For general stripes CA, can one 
prove, under proper conditions, the difference between $W$ and $W_\damage$ discussed in
 Section~\ref{prelim-defdam}?
\item Is it possible to extend Theorem~\ref{wc-ub} to higher 
dimensions and thus give a general sufficient condition that a rule 
is marginal?
\item Does there exist a general algorithm to exactly determine the MLE 
for marginal CA, such as those in Table~\ref{ECA-marginal}?
\item Can one prove that all rules in Table~\ref{ECA-expansive} are indeed expansive?
\item Is it possible to classify glider collisions for CA in Table~\ref{ECA-mystery}
and then show that each rule belongs to the conjectured class? 
\item Can a rigorous damage spreading theory be developed for periodic states? 
\item Does the following version of irreducibility hold for all rules
in Table~\ref{ECA-expansive}: if $\xi_0$ is the uniform product 
measure and $A\subset\bZ$ is finite, then either 
$L_A\equiv -\infty$ or $L_A=L_\infty$?
\end{enumerate}

\section*{Acknowledgements}
This project was partially
funded by the Erasmus Mundus Programme of the European Commission under the
Transatlantic Partnership for Excellence in Engineering Project. We gratefully acknowledge
the assistance of STEVIN Supercomputer Infrastructure at Ghent University.
Janko Gravner was partially supported by the Simons Foundation Award \#281309
and the Republic of Slovenia's Ministry of Science
program P1-285.  


\begin{thebibliography}{99}

 

\bibitem[{BD}]{BD} J.~M.~Baetens, B.~De Baets, {\it
Phenomenological study of irregular cellular automata based on Lyapunov
exponents and Jacobians\/}, 
Chaos 20 (2010), 033112, 1--15. 

\bibitem[{BER}]{BER} F.~Bagnoli, S.~El Yacoubi, R.~Rechtman, {\it
Control of cellular automata\/}, Physical Review E 86 (2012), 
066201--7,
DOI: 10.1103/PhysRevE.86.066201. 

\bibitem[{BF}]{BF}
V.~Belitsky, P.~A,~Ferrari, {\it 
Ballistic annihilation and deterministic surface growth\/},
Journal of Statistical Physics 80 (1995), 517--543. 

\bibitem[{BG}]{BG} J.~M.~Baetens, J.~Gravner, {\it
Introducing Lyapunov profiles of
cellular automata\/},  in ``Proceedings of the 20th 
International Workshop on Cellular Automata and Discrete Complex Systems 
(AUTOMATA 2014) Himeji, Japan, July 2014,'' T.~Isokawa, K.~Imai, N.~Matsuin, F.~Peper, and 
H.~Umeo, editors, pp.~133--140.  {\tt arXiv:1509.06639}

\bibitem[{Big}]{Big}
J.~D.~Biggins, {\it 
The growth and spread of the general branching random walk\/}, 
Annals of Applied Probability 5 (1995), 1008--1024. 

\bibitem[{BRR}]{BRR} F.~Bagnoli, R.~Rechtman, S.~Ruffo, {\it 
Damage spreading and Lyapunov exponents in cellular automata\/}, 
Physics Letters A 172 (1992), 34--38.


\bibitem[{BNT}]{BNT} M.~Bramson, P.~Ney,  J.~Tao, {\it
The population composition of a multitype branching random walk\/},
Annals of Applied Probability 2 (1992), 519--765.

\bibitem[{CK}]{CK}
M.~Courbage, B.~Kami\'nski, {\it Space-time directional Lyapunov exponents for cellular
automata\/}, Journal of Statistical Physics 124 (2006) 1499--1509.

\bibitem[{Coo}]{Coo} M.~Cook,
{\it Universality in elementary cellular automata\/},
Complex Systems 15 (2004), 1--40.

\bibitem[{DS}]{DS}  R.~Durrett, J.~Steif, 
{\it Some rigorous results for the Greenberg-Hastings 
model\/}, Journal of Theoretical Probability (1991), 669--690. 

\bibitem[{DZ}]{DZ}  A.~Dembo, O.~Zeitouni, {\it 
Large Deviations Techniques and Applications\/}, Second Edition. Springer, 1998. 

\bibitem[{EM}]{EM} A.~Edelman, Y.~Ma, {\it 
Non-generic eigenvalue perturbations of Jordan blocks\/}, 
Linear Algebra and Applications 273 (1998), 45--63.
  
\bibitem[{FMM}]{FMM}
M.~Finelli, G.~Manzini, L.~Margara, {\it Lyapunov exponents versus expansivity and
sensitivity in cellular automata\/} Journal of Complexity 14 (1998), 210--233.

\bibitem[{FS}]{FS} S.~Friedland, H.~Schneider, {\it 
The growth of powers of a nonnegative matrix\/}, 
SIAM Journal on Algebraic Discrete Methods 1 (1980), 185--200.

\bibitem[{Gig}]{Gig}
M.-H.~Giga, Y.~Giga, {\it Evolving graphs by singular weighted curvature\/},
Archive for Rational Mechanics and Analysis 141 (1998), 117--198.

\bibitem[{Gra1}]{Gra1} P. Grassberger, {\it 
Chaos and diffusion in deterministic cellular automata\/}, Physica
D 10 (1984), 52--58.

\bibitem[{Gra2}]{Gra2} P. Grassberger, {\it 
Long-range effects in an elementary cellular automaton\/}, Journal of
Statistical Physics 45 (1986), 27--39.

\bibitem[{GG1}]{GG1}
J.~Gravner, D.~Griffeath, {\it
First passage times for discrete threshold growth
dynamics\/}, Annals of Probability 24 (1996), 1752--1778.

\bibitem[{GG2}]{GG2}
J.~Gravner, D.~Griffeath, {\it
Cellular automaton growth on $\bZ^2$:
theorems, examples, and problems\/}, 
Advances in Applied Mathematics 21 (1998), 241--304.

\bibitem[{GG3}]{GG3}
J.~Gravner, D.~Griffeath, {\it Random growth models with polygonal shapes\/}, 
Annals of Probability 34 (2006), 181--218.

\bibitem[{GG4}]{GG4}
J.~Gravner, D.~Griffeath, {\it The one-dimensional Exactly 1 cellular automaton:
replication, periodicity, and chaos from finite seeds\/},
Journal of Statistical Physics 142 (2011), 168--200.

\bibitem[{GG5}]{GG5}
J.~Gravner, D.~Griffeath, {\it Robust periodic solutions and evolution from seeds in one-dimensional edge cellular automata\/},
Theoretical Computer Science 466 (2012), 64--86.

\bibitem[{GH}]{GH} J.~Gravner, A.~Holroyd, {\it 
Percolation and disorder-resistance in cellular automata\/},
Annals of Probability 43 (2015), 1731--1776. 


\bibitem[{Jen}]{Jen}
E.~Jen, {\it Exact solvability and quasiperiodicity of one-dimensional cellular automata\/}, 
Nonlinearity 4 (1991), 251--276.

\bibitem[{LN}]{LN}
W.~Li, M.~G.~Nordahl, {\it Transient behavior of cellular automaton rule 110\/}, 
Physics Letters A 166 (1992), 335--339.


\bibitem[{Mar}]{Mar}
G.~J.~Martinez, {\it A note on elementary cellular automata classification\/}, 
 Journal of Cellular Automata 8 (2013), 233--259 . 
 

\bibitem[{Moo}]{Moo} G.~Moore, {\it 
Floquet theory as a computational tool\/}, 	
SIAM Journal on Numerical Analysis 
42 (2004), 2522--2568.

 \bibitem[{MSZ}]{MSZ}
G.~J.~Martinez, J.~C.~Seck-Tuoh-Mora, H.~Zenil, {\it Computation and universality:
Class IV versus Class III cellular automata\/}, 
 Journal of Cellular Automata 7 (2012), 393--430 . 

\bibitem[{RS}]{RS}  F.~Rassoul-Agha, T.~Sepp\"al\"ainen, 
``A Course on Large Deviations with an Introduction to Gibbs Measures.'' 
American Mathematical Society, Graduate Studies in Mathematics Volume 162, 2015.  

\bibitem[{Sch}]{Sch} 
R.~H.~Schonmann, {\it Finite size scaling behavior of a biased majority rule cellular
automaton\/}, Physica A 167 (1990), 619--627.

\bibitem[{She}]{She} 
M.~A.~Shereshevsky, {\it 
Lyapunov exponents for one-dimensional cellular
automata\/},
Journal of Nonlinear Science 2 (1992), 1--8.

\bibitem[{Tis1}]{Tis1} 
P.~Tisseur, {\it Cellular automata and Lyapunov exponents\/}, Nonlinearity 13 (2000), 
1547--1560.

\bibitem[{Tis2}]{Tis2} 
P.~Tisseur, {\it
Always finite entropy and Lyapunov exponents
of two-dimensional cellular automata\/}.
{\tt arXiv:math/0502440}

\bibitem[{TM}]{TM} 
T.~Toffoli, N.~Margolus, 
``Cellular Automata Machines.''
MIT Press, 1991.

\bibitem[{Vic1}]{Vic1} G.~Vichniac, {\it
Cellular automata models of disorder and organization\/}, 
in ``Disordered Systems and Biological Organization,'' 
E.~Bienenstock, F.~Fogelman Souli\'e, 
G.~Weisbuch, eds., Springer (1986), pp.~3--20. 

\bibitem[{Vic2}]{Vic2} G.~Vichniac, {\it 
Boolean derivatives on cellular automata\/}, Physica D 45 (1990), 63--74.

\bibitem[{Wol1}]{Wol1}
S.~Wolfram, {\it Universality and complexity in cellular automata\/}, Physica D 10 (1984), 1--35.

\bibitem[{Wol2}]{Wol2}
S.~Wolfram, {\it The Wolfram Atlas: Elementary Cellular Automata\/},  \\
{\tt http://atlas.wolfram.com/01/01/}

\bibitem[{Wil}]{Wil}
S.~J.~Willson, {\it On convergence of configurations\/}, Discrete Mathematics 23 (1978), 279--300.

\bibitem[{Yil}]{Yil}
A.~Yilmaz, {\it 
Quenched large deviations for random walk in a random environment\/}, 
Communications on Pure and Applied Mathematics 62 (2009), 1033--1075.

 \bibitem[{ZV}]{ZV} 
  H.~Zenil, E.~Villareal-Zapata, 
  {\it Computation and universality:
Class IV versus Class III cellular automata\/}, 
 International Journal of Bifurcation and Chaos 23 (2013), 
 18 pages, DOI: 10.1142/S0218127413501599. 


\end{thebibliography}
\end{document}